\documentclass[aihp]{imsart}

\RequirePackage{amsthm,amsmath,amsfonts,amssymb}
\usepackage{mathtools,dsfont, enumerate, xcolor, verbatim, graphicx, epstopdf, enumitem, cancel, subcaption, float}
\usepackage{ulem}
\usepackage{multirow, array}
\RequirePackage[numbers,sort&compress]{natbib}
\RequirePackage[colorlinks,citecolor=blue,urlcolor=blue]{hyperref}
\RequirePackage{graphicx}

\newcommand{\PreserveBackslash}[1]{\let\temp=\\#1\let\\=\temp}
\newcolumntype{C}[1]{>{\PreserveBackslash\centering}p{#1}}
\newcolumntype{R}[1]{>{\PreserveBackslash\raggedleft}p{#1}}
\newcolumntype{L}[1]{>{\PreserveBackslash\raggedright}p{#1}}

\startlocaldefs
\theoremstyle{plain}

\newtheorem{thm}{Theorem}[section]
\newtheorem{lem}[thm]{Lemma}
\newtheorem{prop}[thm]{Proposition}
\newtheorem{cor}[thm]{Corollary}
\newtheorem*{thm*}{Theorem}
\theoremstyle{remark}

\newtheorem{rem}[thm]{Remark}

\newcommand{\E}{\mathbb{E}}
\newcommand{\D}{{\rm d}}
\newcommand{\dd}{{\partial}}
\newcommand{\PP}{{\mathbb{P}}}
\newcommand{\bx}{{\mathbf x}}
\newcommand{\bX}{{\mathbf X}}
\newcommand{\bP}{{\mathbf P}}
\newcommand{\bE}{{\mathbf E}}
\newcommand{\bI}{\mathbf{1}}
\newcommand{\XX}{\mathbb X}
\newcommand{\YY}{\mathbb Y}

\newcommand{\BD}{\texttt{B}}
\newcommand{\DD}{\texttt{D}}
\newcommand{\ind}{{\mathbf 1}}

\newcommand\xqed[1]{%
    \leavevmode\unskip\penalty9999 \hbox{}\nobreak\hfill
    \quad\hbox{#1}}

\endlocaldefs

\begin{document}

\begin{frontmatter}
\title{Binary branching processes with Moran type interactions}
\runtitle{Binary branching processes with Moran type interactions}

\begin{aug}
\author[A]{\inits{A. M. G.}\fnms{Alexander M. G.}~\snm{Cox}\ead[label=e1]{a.m.g.cox@bath.ac.uk}\orcid{0000-0001-5151-9126}},
\author[B]{\inits{E.}\fnms{Emma}~\snm{Horton}\ead[label=e2]{emma.horton@warwick.ac.uk}\orcid{0000-0002-6027-2169}}
\and
\author[C]{\inits{D.}\fnms{Denis}~\snm{Villemonais}\ead[label=e3]{denis.villemonais@univ-lorraine.fr}\orcid{0000-0002-7534-1884}}
\address[A]{Department of Mathematical Sciences, University of Bath, Claverton Down, Bath, BA2 7AY, UK. \printead[presep={,\ }]{e1}}

\address[B]{Department of Statistics, University of Warwick, Coventry, CV4 7AL.\printead[presep={,\ }]{e2}}

\address[C]{Institut \'Elie Cartan de Lorraine, Bureau 123, Universit\'e de Lorraine, 54506, Vandoeuvre-l\`es-Nancy Cedex, France.\printead[presep={,\ }]{e3}}
\end{aug}

\begin{abstract}
The aim of this paper is to study the large population limit of a 
new collection of interacting particle systems (IPS) that encompasses branching models and fixed size Moran type IPS (see ~\cite{DelMoral2004,DelMoral2013,BurdzyHolystEtAl1996,BurdzyHolystEtAl2000,Villemonais2014}). We define an IPS where particles evolve, reproduce and die independently and, with a probability that may depend on the configuration of the whole system, the death of a particle may trigger the reproduction of another particle, while a branching event may trigger the death of an other one. We study the occupation measure of the new model, explicitly relating it to the Feynman-Kac semigroup of the underlying Markov evolution and quantifying the $L^2$ distance between their normalisations. We show that our model outperforms the fixed size Moran type IPS when used to approximate a birth and death process. We discuss several other applications of our model including the neutron transport equation~\cite{HortonEtAl2018, cox2020monte} and population size dynamics.
\end{abstract}

\begin{abstract}[language=french]
L'objectif de cet article est d'étudier la limite en grande population d'une nouvelle famille de systèmes de particules en interaction (IPS) qui englobe les modèles de branchement et les IPS de type Moran de taille fixe (voir ~\cite{DelMoral2004,DelMoral2013,BurdzyHolystEtAl1996,BurdzyHolystEtAl2000,Villemonais2014}). Nous définissons un IPS où les particules évoluent, se reproduisent et meurent indépendamment et, avec une probabilité qui peut dépendre de la configuration du système entier, la mort d'une particule peut déclencher la reproduction d'une autre particule, tandis qu'un événement de branchement peut déclencher la mort d'une autre particule. Nous étudions la mesure d'occupation du nouveau modèle, en la reliant explicitement au semigroupe de Feynman-Kac de l'évolution markovienne sous-jacente et en quantifiant la distance $L^2$ entre leurs normalisations. Nous montrons que notre modèle est plus performant que l'IPS de Moran de taille fixe lorsqu'il est utilisé pour approximer un processus de naissance et mort. Nous discutons de plusieurs autres applications de notre modèle, notamment l'équation de transport des neutrons~\cite{HortonEtAl2018, cox2020monte} et la dynamique de la taille des populations.
\end{abstract}



\end{frontmatter}

\section{Introduction}
Branching processes naturally arise as pertinent models in population size dynamics~\cite{Jagers1989,Jagersothers1995, Lambert2005}, neutron transport~\cite{CoxHarrisEtAl2019}, genetic dynamics~\cite{Marshall2008}, growth-fragmentation processes \cite{bertoin2017markovian, bertoin2019feynman} and cell proliferation kinetics~\cite{Yanev2010}, and as theoretical objects in their own right~\cite{Jagers1989,Dawson1993,IkedaNagasawaEtAl1968,IkedaNagasawaEtAl1968a,IkedaNagasawaEtAl1969}. These models are characterised by the independence of the branching and killing events in the system, which leads to a multiplicative behaviour, as well as interest in the scaling behaviour and the characterisation of the system at large times.
On the other hand, processes with Moran type interactions are natural models for finite populations with either variety-increasing or variety-reducing effects such as genetic drift, genetic mutations and natural selection. First introduced by Moran~\cite{Moran1958}, the Moran model describes the evolution of $N$ genes such that at an exponential rate, two particles are chosen uniformly at random from the population, one of which is killed and the other splits in two. Thus, the independence between particles is lost. 
We refer the reader to~\cite{Etheridge2011} and references therein for an overview of this model and of its extensions. This type of resampling has since been employed in a range of particle system models to numerically solve Feynman-Kac models, see~\cite{CloezCorujo2021, del2000moran, del2000branching, Villemonais2014} and references therein. 

The model we consider in this paper provides a combination of these two types of dynamics: (natural) branching and killing, as well as Moran type interactions. More precisely, when the system is initiated from $N$ particles, each particle evolves according to an independent copy of a given Markov process, $X$, until either a (binary) branching or killing event happens.  Here, binary refers to the fact that the particle is replaced by exactly two independent copies of itself.  If such a branching event occurs, with a probability that may depend on the configuration of the whole system, another particle is removed from the system a selection mechanism. Similarly, if a killing event occurs, with a probability which may also depend on the configuration of the whole system, another particle is duplicated via a resampling mechanism. We refer to this model as the {\it \color{black} binary branching model with Moran interactions}, or BBMMI for short.

Our main result (cf. Theorem~\ref{thm:main}) gives an explicit relation between the average of the empirical distribution of the particle system and the average of the underlying Markov process $X$. Indeed, letting $m_T$ denote the empirical distribution of the particle system at time $T$, we show that for any $T \ge 0$,
\begin{equation}
\E\left[\Pi_T^{A}\,\Pi_T^{B} m_T f\right] = \bE_{m_0}\left[f(X_{t})\exp\left(\int_0^{t} b(X_{s}) - \kappa(X_{s})\,\D s\right)\mathbf{1}_{t<\tau_\partial}\right],
\label{M21-intro}
\end{equation}
where $b$ is the branching rate, $\kappa$ is the soft killing rate, $\tau_\partial$ is the absorption time, $\Pi_T^A$ and $\Pi_T^B$ are weights that compensate for the resampling and selection events that occur up to time $T$, which we state explicitly in Theorem~\ref{thm:main}, and $\bE_{\cdot}$ is the expectation with respect to the law of $X$. As we will later discuss, by choosing particular model parameters, we recover both the classical many-to-one formula~\cite{HarrisRobertsEtAl2017} and the unbiased estimator proved in~\cite{Villemonais2014} for fixed size Moran type genetic algorithms. The second part of Theorem~\ref{thm:main} then provides a precise bound on the $L^2$ distance between the normalised semigroup on the right-hand side of~\eqref{M21-intro} and the normalised occupation measure of the particle system. In particular, it states that this distance converges to zero as the number of initial particles tends to infinity. 

In the third section of this article, we illustrate several applications of this model by focussing on a particular selection/resampling mechanism in our model. More precisely, we consider a binary branching process whose population size is constrained to remain in $\{N_{min},\ldots,N_{max}\}$, where $2\leq N_{min}\leq N_{max}<+\infty$. In order to constrain the size of the process, when the size of the population reaches $N_{max}$ (resp. $N_{min}$) and a natural branching (resp. killing) event occurs, we set the probability of selection (resp. resampling) to be $1$. We often refer to this model as the $N_{min}$--$N_{max}$ model. As will be clear in the next section, the choice $N_{min}=0$ and $N_{max}=+\infty$ is also allowed, so that the collection of models we consider ranges from constant population size models with neither branching nor killing (corresponding to the choice $N_{min}=N_{max}$) to branching population models (when $N_{min}=0$ and $N_{max}=+\infty$) with or without a selection/resampling mechanism.

For comparable results concerning the resampling algorithm, we refer the reader to~\cite{FerrariMaric2007,Rousset2006,CloezThai2016,OcafrainVillemonais2017} and to~\cite{delyon2017central, del2001asymptotic,LelievrePillaud-VivienEtAl2018} for associated central limit theorems. Whether a central limit theorem can be proved in our setting remains open, but could be approached by carefully studying the square increments of the martingale decomposition used in the proof of Theorem~\ref{thm:main}.

Our model and results are also reminiscent of the genetic algorithms introduced by Del Moral (see~\cite{DelMoral2004,DelMoral2013} and references therein), where a particle system with killing is constrained to remain at a constant size via a resampling mechanism. We also refer the reader to~\cite{BurdzyHolystEtAl1996,BurdzyHolystEtAl2000,GrigorescuKang2004,Rousset2006,Villemonais2011,CloezCorujo2021} and references therein. 

Finally, we note that our model fits into the more general class of controlled branching processes introduced by Sevastyanov and Zubkov in \cite{SevastyanovZubkov1974}, where the number of reproductive individuals in each generation depends on the size of the previous generation via a control function. This work was later extended by Yanev \cite{yanev1976} to allow for random control functions. We refer the reader to \cite{SevastyanovZubkov1974,GonzalezMolinaEtAl2002,VelascoGarciaEtAl2017, yanev2014critical} for discrete time versions of this process and to \cite{GarciaYanevEtAl2021} for the continuous time version. In these articles, the authors study the convergence of the survival probability in different regimes, as well as the expected population size. We also note that the control function can be seen as a way to model immigration and emigration of particles, \cite{yanev2014critical}. We refer the reader to \cite{Vatutin1978, KawazuWatanabe1971,Olofsson1996,Li2006,Lambert2008,BarczyBezdanyEtAl2021,Ban21} for results regarding these latter processes.

The rest of the paper is set out as follows. In the next section, we introduce some notation and assumptions, and formally describe the particle system introduced above. We will then state our main result in Theorem~\ref{thm:main} and discuss several implications. 

In section~\ref{sec:apps}, we focus on the $N_{min}$--$N_{max}$ model and consider a variety of applications in order to illustrate the scope of our results, as well as some of the possible difficulties and extensions. 
In particular, in Section~\ref{sec:BD}, we describe and study a stochastic population model with constrained size, based on the $N_{min}$--$N_{max}$ model. We give a sufficient condition ensuring that the number of resampling events does not explode in finite time when the killing rate of $X$ is unbounded. We also relate the long-time convergence of the empirical distribution of the population to the existence of a quasi-stationary distribution for an auxiliary sub-Markov process related to $X$ and show the uniform in time convergence of the renormalised empirical distribution of the particle system toward its conditional distribution, at a polynomial speed when $\sqrt{N_{max}}/N_{min}$ goes to infinity. We prove these results in the particular setting of this model, but our methods could be extended to a more general setting.

In Section~\ref{sec:PDMP}, we deal with the particular case where the underlying Markov process is a piecewise deterministic Markov process (PDMP). In this case, it is possible for two particles to be killed at the same time, meaning that our main result is not directly applicable without further work. Working in the particular setting of the neutron transport equation (NTE), see \cite{HortonEtAl2018, cox2020monte}, we overcome this difficulty using the notion of $h$-transform introduced in \cite{cox2020monte} for the NTE. We show that particles whose dynamics are given by an appropriately $h$-transformed process fulfil the assumptions of our main theorem, and thus, we may use the $N_{min}$--$N_{max}$ model with the $h$-transformed process in order to obtain estimates for quantities associated to the original process.

In Section~\ref{sec:numerics}, we focus on some numerical properties of the resampling/selection process. First, we make comparisons between the $N_{min}$--$N_{max}$ model and the fixed size Moran type interacting particle system (IPS) for the approximation of quantities related to non-conservative semigroups (cf.~\cite{DelMoral2004,DelMoral2013,BurdzyHolystEtAl1996,BurdzyHolystEtAl2000}), which corresponds to the particular case $N_{min}=N_{max}$ with no branching. We observe that in the case of a birth-and-death process, $X$, with branching rate given by $X_t\wedge M$ for some $M>0$ and state-dependent Markov transitions, the computational cost, measured as the number of resampling/selection events, is significantly lower when using the $N_{min}$--$N_{max}$ algorithm. We also study the bias and standard deviation of an estimator for the normalised empirical stationary distribution for each of the $N_{min}$--$N_{max}$ and fixed size Moran type IPS. Letting $M$ go to infinity, we provide numerical evidence that the number of interactions, bias and standard deviation stabilise for the $N_{min}$--$N_{max}$ process, whereas these quantities grow linearly for the fixed size Moran type IPS process. Based on these simulations, we conjecture that the results of Section~\ref{sec:mainresult} hold true in situations where the branching rate is unbounded and thus where the Moran type process is undefined. 
In the final part of Section~\ref{sec:numerics}, we introduce a two-level method to deal with some of the numerical discrepancies that occur in the large-time limit when computing Lyapunov exponents.

Finally, in Section~\ref{sec:proof} we provide the proof of Theorem~\ref{thm:main}. 

\section{Main results}\label{sec:main}

\subsection{Description of the model}
\label{sec:descr}


Let $(\Omega,\mathcal F,(X_t)_{t\in[0,+\infty)})$ be a continuous time progressively measurable Markov process with values in a measurable state space $E\cup\dd$, where $\dd$ is an absorbing (measurable) set such that $\dd \cap E=\emptyset$. Denoting by $\tau_\dd=\inf\{t\geq 0,\ X_t \in \dd\}$ its absorption time, it follows that 
\[
X_t\in \dd,\ \forall t\geq \tau_\dd.
\]
We denote by $\bP_x$ its law when initiated at $x \in E$ and by $\bE_x$ the corresponding expectation operator. We extend, whenever necessary, any measurable function $f:E\to[0,+\infty)$ by $f\equiv 0$ on $\dd$. We also assume that we are given two bounded functions $b:E\to\mathbb R_+$ and $\kappa:E\to\mathbb R_+$, denoting respectively, the single particle branching rate and the single particle killing rate.

We will shortly define a particle system, where particles move according to copies of $X$ between interactions. Before doing so, we make the following assumption, necessary for the system to be well defined. It  entails that two independent copies of $X$ are absorbed simultaneously with probability $0$.


\medskip\noindent\textbf{Assumption 1.} For any $x\in E$ and $t\in [0,+\infty)$, $\bP_x(\tau_\dd=t)=0$ and $\bP_x(\tau_\dd>t)>0$.

\bigskip

We also introduce the following notation. Denote by $\mathcal P_f(\mathbb N)$ the collection of finite subsets of $\mathbb N:= \{1, 2, \dots\}$. Then, for $i\in\mathbb N$, we let
\begin{align*}
&b^i:(E\cup\dd)^{\mathcal P_f(\mathbb N)}\to[0,+\infty),\\ 
&\kappa^i:(E\cup\dd)^{\mathcal P_f(\mathbb N)}\to[0,+\infty)  
\end{align*}
be a collection of bounded measurable functions.
Also for $i\in\mathbb N$, let
\begin{align*}
&p^i:(E\cup\dd)^{\mathcal P_f(\mathbb N)}\to [0,1],\\ 
&q^i:(E\cup\dd)^{\mathcal P_f(\mathbb N)}\to [0,1]
\end{align*}  
be measurable functions.  We will assume that, for any $s\in \mathcal P_f(\mathbb N)$, for any collection of points $\{x_i : i\in s, x_i \in E\}$, and any $i_0\in s$,
 \begin{align}
 \label{eq:balance}
 b^{i_0}(x_i,i\in s)-\kappa^{i_0}(x_{i},i\in s)=b(x_{i_0})-\kappa(x_{i_0})
 \end{align}
 and that
 \begin{align}
 \label{eq:constraint}
p^{i_0}(x_i,i\in s)=0\text{ whenever }| s|=1,
 \end{align}
  where $|s|$ denotes the number of elements in the set $s$. See Remarks~\ref{rem:bkappa} and~\ref{rem:pq} for a discussion of the implications of these conditions.

We now proceed with an algorithmic description of the dynamics of the particle system, which we call the binary branching model with Moran type interactions (BBMMI). The formal construction of the process is a non-trivial task, and is given in the supplementary material \cite{BBMMI-sup}.
To this end, fix $\bar{N}_0\geq 2$. We consider the particle system  $(\bar S_t,(X^i_t)_{i\in \bar S_t})_{t\in[0,+\infty)}$, where the component $\bar S_t\in\mathcal P_f(\mathbb N)$ is the set enumerating the particles in the system at time $t$, initiated with $\bar S_0=\{1,2,\ldots,\bar{N}_0\}$ and $X^i_0=x_i\in E$ for all $i\in\bar S_0$. The number of particles in the system at any time $t$ is denoted by $\bar N_t=|\bar S_t|$.

\medskip\noindent\textbf{Evolution of the BBMMI.}\label{Alg1} 
\begin{enumerate}
	\item The particles $X^i$, $i\in\bar S_0$, evolve as independent copies of $X$, and we consider  the following times: 
	\begin{align*}
	\tau^{b,i}_1:=\inf\{t\geq 0,\ \int_0^t b^i(X^i_s,\,i\in \bar S_0)\,ds\geq e_1^i\},
	\end{align*}
    and
    \begin{align*}
    \tau^{\kappa,i}_1:=\inf\{t\geq 0,\ \int_0^t \kappa^i(X^i_s,\,i\in \bar S_0)\,ds\geq E_1^i\}
    \end{align*}
	and
	\begin{align*}
	\tau^{\dd,i}_1:=\inf\{t\geq 0,\ X^{i}_t\in \dd\},
	\end{align*}
	 where $e_1^i,E_1^i$, $i = 1, \dots, \bar{N}_0$ are exponential random variables with parameter $1$, and are independent of each other and everything else.
	 \item Denoting by $i_0$ the index of the (unique) particle such that $\tau^{b,i_0}_1\wedge \tau^{\kappa,i_0}_1\wedge \tau^{\partial,i_0}_1=\tau_1$, where $\tau_1=\min_{i\in \bar S_0} \tau^{b,i}_1\wedge \tau^{\kappa,i}_1\wedge \tau^{\partial,i}_1$, we set $\bar S_t=\bar S_0$ for all $t<\tau_1$ and
	 \begin{enumerate}
        \item if $\tau_1=\tau^{b,i_0}_1$, then a \textit{branching event} occurs;
        \item if $\tau_1=\tau^{\kappa,i_0}_1$, then a \textit{soft killing event} occurs;
	 	\item if $\tau_1=\tau^{\dd,i_0}_1$, then a \textit{hard killing event} occurs.
	 \end{enumerate}
	 \item Then a resampling or selection event may occur, depending on the following situations:
	 \begin{enumerate}
	 
	 	\item[] {\bf killing:} if a (hard or soft) killing event occurred at the preceding step, then we say that $i_0$ is \textit{killed} at time $\tau_1$ and 
         \begin{itemize}
             \item  with probability $p^{i_0}(X^i_{\tau_1},\,i\in\bar S_0)$, the particle $i_0$ is removed from the system and a \textit{resampling event} occurs: one chooses $j_0$ uniformly in $S_0\setminus\{i_0\}$ and sets
             \[
             X^{\max S_0+1}_{\tau_1}=X^{j_0}_{\tau_1}\text{ and }\bar S_{\tau_1}:=\bar S_0\setminus\{i_0\}\cup\{\max \bar S_0+1\};
             \]
             observe that the number of particles in the system at time $\tau_1$ is then $\bar N_{\tau_1}=\bar N_{0}$; we say that $j_0$ is \textit{duplicated} at time $\tau_1$.
             \item with probability $1-p^{i_0}(X^i_{\tau_1},\,i\in\bar S_0)$, the particle $i_0$ is removed from the system;  the set of particles at time $\tau_1$ is enumerated by $\bar S_{\tau_1}:=\bar S_0\setminus\{i_0\}$ and the number of particles in the system  is then $\bar N_{\tau_1}=\bar N_{0}-1$;
         \end{itemize} 
	 	\item[] {\bf branching:} if a branching event occurred at the preceding step, then we say that $i_0$ has \textit{branched} at time $\tau_1$ and
         \begin{itemize}
             \item with probability $q^{i_0}(X^i_{\tau_1},\,i\in\bar S_0)$, a new particle is added to the system at position $X^{i_0}_{\tau_1}$ and a \textit{selection event} occurs: one chooses $j_0$ at random uniformly in $S_0\cup\{\max S_0+1\}$ and removes the particle  $j_0$ from the system:
             \[
             X^{\max S_0+1}_{\tau_1}=X^{i_0}_{\tau_1}\text{ and }\bar S_{\tau_1}:=\{\max \bar S_0+1\}\cup \bar S_0\setminus\{j_0\};
             \]
             observe that the number of particles at time $\tau_1$ is thus $\bar N_{\tau_1}=\bar N_0$; we say that $j_0$ is \textit{removed} at time $\tau_1$ and that $\max S_0+1$ is \textit{born} at time $\tau_1$;
             \item with probability $1-q^{i_0}(X^i_{\tau_1},\,i\in\bar S_0)$,  a new particle is added to the system at position $X^{i_0}_{\tau_1}$:
              \[
             X^{\max S_0+1}_{\tau_1}=X^{i_0}_{\tau_1}\text{ and }\bar S_{\tau_1}=\bar S_0\cup\{\max\bar S_0\},
             \]
             and we say that $\max S_0+1$ is \textit{born} at time $\tau_1$.
         \end{itemize}
	 \end{enumerate}
\end{enumerate}
After time $\tau_1$ the system evolves as independent copies of $X$ until the next killing/branching event, denoted by $\tau_2$, and at time $\tau_2$ it may undergo a resampling/selection event as above. By iteration, we define the sequence $\tau_0:=0<\tau_1<\tau_2<\cdots<\tau_n<\cdots$.

\begin{figure}[h!]
  \begin{subfigure}{0.5\textwidth}
    \includegraphics[width=8cm, height=5.5cm]{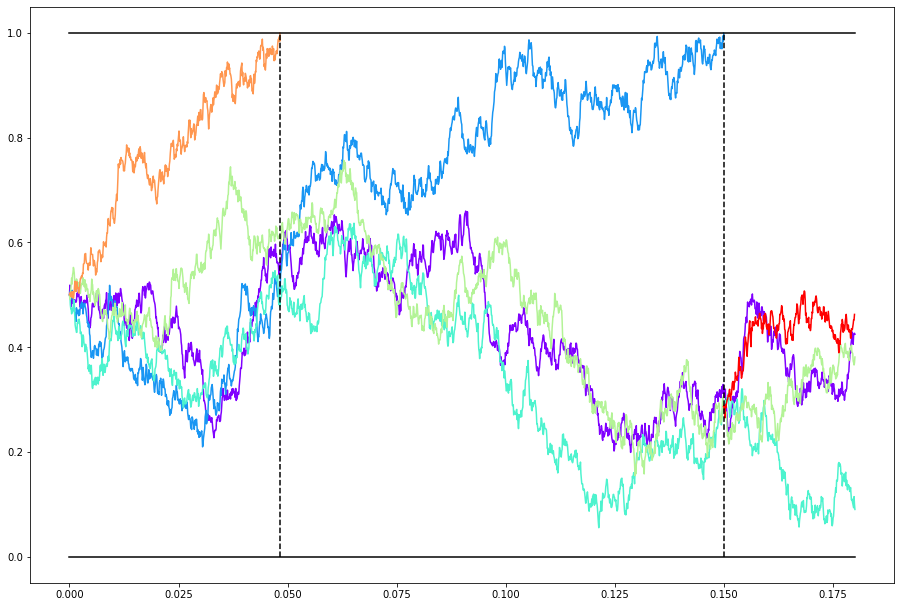} 
    \caption{Resampling. Just before time 0.05, the orange particle is killed but no resampling occurs. Around time 0.15, the blue particle is killed and the green particle is resampled to produce the red particle.\newline \newline}
    \label{fig:forcebranch}
  \end{subfigure}\hspace{0.05\textwidth}
  \begin{subfigure}{0.475\textwidth}
    \includegraphics[width=8cm, height=5.5cm]{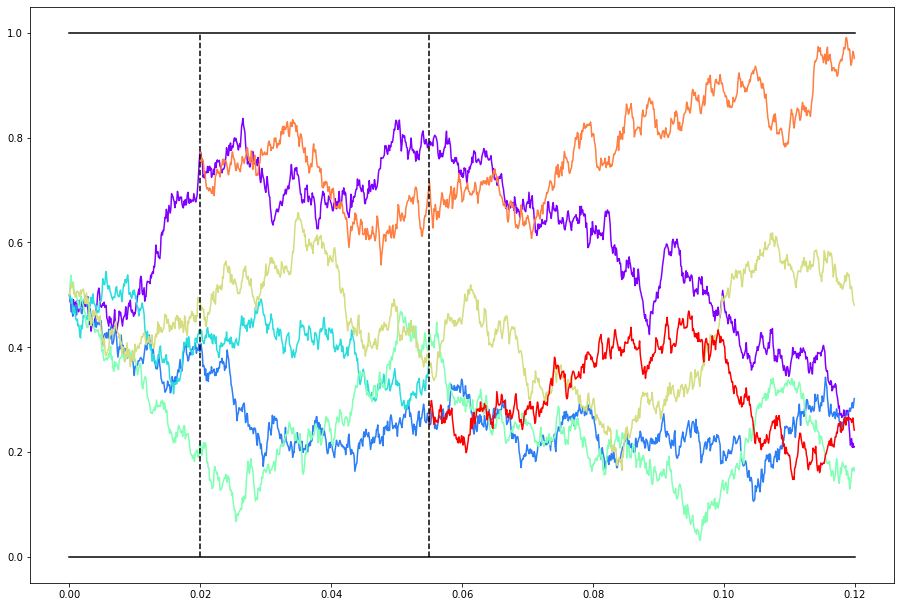}
    \caption{Selection. At time 0.02, the purple particle branches, producing the orange particle but no selection occurs. At the second dashed line, the dark blue particle branches to produce the red particle and the light blue particle is selected and removed from the system.}
    \label{fig:forcekill}
  \end{subfigure}
  \caption{Simulations of the {BBMMI} model where $X$ is a Brownian motion, killed at $0$ and $1$, with constant branching rate.}
  \label{fig:image2}
\end{figure}

We will also make use of the following assumption, which ensures that the process described above is well defined at any time $t\geq 0$.

\medskip\noindent\textbf{Assumption 2.} The sequence $(\tau_n)_{n\in\mathbb{N}}$ converges to $+\infty$ almost surely.

\smallskip

In words, the above dynamics describe an interacting particle system such that, between interactions, the individual particles evolve as copies of $X$, undergoing binary branching at rate $b^i$ (where $i$ denotes the index of the particle) and being absorbed either at rate $\kappa^i$ or when the particle hits an absorbing state. On the event that a particle has branched, with probability $p^i$, a uniformly chosen particle is removed from the system. Conversely, on the event that a particle has been absorbed, with probability $q^i$, a uniformly chosen particle is forced to branch. 

Let us now make several remarks concerning the above particle system and potential parameter choices. 

\begin{rem}
    The enumeration of the particle system  does not keep track of the genealogy of the particles in the system in order to simplify notation; this could be done by replacing $\mathcal P_f(\mathbb N)$ by $\mathcal P_f(\Omega)$, the collection of finite subsets of 
    \[
      \Omega := \bigcup_{n \ge 0} \mathbb N^n.
    \]
    A particle with label $v = v_1v_2 \dots v_n \in \Omega$ means that the particle is the $v_n$-th child of the $v_{n-1}$-th child \dots of $v_1$. In this case, the functions $b^i,\kappa^i,p^i, q^i$, should then be indexed by elements of $\Omega$ instead of $\mathbb N$.
\end{rem}

\smallskip


\begin{rem}
    \label{rem:bkappa}
    The functions $b^{i_0}$ and $\kappa^{i_0}$ are the rate at which the particle $i_0$ in the system branches and is killed respectively (in addition to hard killing events occurring when a particle hits the absorbing set $\partial$).
     We emphasise that the balance condition~\eqref{eq:balance} is crucial to ensure the relation with the semigroup $Q$ defined in the next section, and should not be considered a technical assumption.
        Extending our results to situations where $b^{i_0}$ and $\kappa^{i_0}$ do not satisfy~\eqref{eq:balance} is the subject of ongoing work.
%
\end{rem}

\smallskip

\begin{rem}
    \label{rem:pq}
    The functions $p^{i_0}$ (respectively $q^{i_0}$) denote the probability that a killing event (respectively a branching event) triggers a branching (respectively a killing) among the particles in the system. We call such an event a \textit{resampling event} (respectively a \textit{selection event}). In the context of Moran type models, they correspond to \textit{selection by death} events (respectively \textit{selection by birth} events). Condition~\eqref{eq:constraint} prevents a resampling event from occurring when the population decreases to $0$ particles. We emphasize that, since $\partial$ is a set, $p^{i_0}$ may depend on the killing location.
\end{rem}

\smallskip

\begin{rem}\label{rem-NminNmax}
    The dynamics of the IPS allow one to constrain the size of the process to remain between two bounds $N_{min}\neq 1$ and $N_{max}$, with $0\leq N_{min}\leq N_{max}$. In order to do so, one simply chooses 
\[
  p^{i_0}(x_i,\,i\in s)=\mathbf{1}_{\{|s|=N_{min}\}}\quad  \text{  and  } \quad q^{i_0}(x_i,\,i\in s)=\mathbf{1}_{\{|s|=N_{max}\}}. 
\]  
Throughout the rest of this article, we refer to the IPS with these choices for $p^i$ and $q^i$ as the $N_{min}$--$N_{max}$ process.
Other natural choices include, for instance, 
\[
  p^{i_0}(x_i,\,i\in s)=\frac{1}{\sharp s+1}\mathbf{1}_{|s|\geq 2} \quad \text{  and  } \quad q^{i_0}(x_i,\,i\in s)=1-\frac{1}{\sharp s+1},
\]
i.e. resamplings become less probable when the number of particles increases and selection becomes stronger when the number of particles increases, respectively.
Note also that, since $X$ is only assumed to be Markov, it can carry the time as a component, which would allow  $N_{min}$ and $N_{max}$ to vary with time (in the former situation) or the selection pressure to vary with time (in the later example), by simply letting $p,q$ depend on time.

Condition~\eqref{eq:balance} allows to consider models where the activity (branching/killing) of the particles increase when they are close or far from each others. For instance, if $(E,d)$ is a metric space with $b$ and $\kappa$ fixed,
\begin{align*}
b^{i_0}(x_i,i\in s)=b(x_{i_0})+\left(\sum_{i\in S} d(x_{i_0},x_i)\right)\wedge 1\text{ and }\kappa^{i_0}(x_i,i\in s)=\kappa(x_{i_0})+\left(\sum_{i\in S} d(x_{i_0},x_i)\right)\wedge 1
\end{align*}
in which case the activity of the particle $i_0$ is increased when it is far from the other particles. A similar construction replacing $\sum_{i\in S} d(x_{i_0},x_i)$ by its inverse would lead to an increased activity when $x_{i_0}$ is close the other particles.
\end{rem}

\smallskip

\begin{rem}
    In our model, we consider general killing, which corresponds to a conservative Markov process $X$ killed at a possibly unbounded rate. Finding general criteria ensuring that Assumptions 1 and 2 hold is usually involved (it is trivial when the killing rate is bounded and there is no hard killing). In Section~\ref{sec:BD} we  provide  a criterion based on a Foster-Lyapunov type assumption ensuring the non-explosion of the system when the killing rate $\kappa$ is unbounded (and when there is no hard killing) and also show in Section~\ref{sec:PDMP} that Assumption~2 holds true for a specific PDMP model, obtained as the $h$-transform of a  Neutron Random Walk killed at the boundary of domain. There is also a substantial literature on these types of problems in the presence of hard killing for the more classical fixed size Moran type systems, 
   and, in particular, the interested reader may find it useful to adapt the methods developed in~\cite{BurdzyHolystEtAl2000,GrigorescuKang2004,Villemonais2011,Villemonais2014} to our setting. 
\end{rem}

\bigskip

In the rest of this article, we denote by $\sigma_n$ and $\rho_n$ the times of the $n^{th}$ selection and resampling events, respectively.  We let $A_t$ denote the total number of resampling events up to time $t$, $B_t$ denote the total number of selection events up to time $t$ and we set $C_t = \inf\{n \ge 0 : t < \tau_{n + 1}\}$ to be the total number of events up to time $t$. We will also use the following quantities throughout, 
\begin{equation}
\Pi_t^A \coloneqq \prod_{n=1}^{A_t}\left(\frac{\bar N_{\rho_n}-1}{\bar N_{\rho_n}}\right), \qquad \Pi_t^B \coloneqq \prod_{i=1}^{B_t}\left(\frac{\bar N_{\sigma_n}+1}{\bar N_{\sigma_n}}\right).\label{Pi}
\end{equation}

\subsection{$L^2$ bounds for the empirical distribution of the process}
\label{sec:mainresult}

 For all bounded measurable functions $f:E\to\mathbb{R}$, all $t\geq 0$ and all $x\in E$, we set
\begin{align*}
	Q_t f(x)&=\bE_x\left[f(X_{t})\exp\left(\int_0^{t} (b(X_{s})-\kappa(X_s))\,\D s\right)\mathbf{1}_{t<\tau_\partial}\right],
\end{align*}
where we recall that $\tau_\partial = \inf\{t \ge 0 : X_t \in \partial\}$. This defines a Feynman-Kac semigroup $(Q_t)_{t\geq 0}$, which is related to the binary branching model where particles move as copies of $X$ that are killed at rate $\kappa$ and branch at rate $b$ resulting in the creation of two independent copies of the original particle. The relation between $Q$ and this process is given by the well known many-to-one formula (see for instance~\cite{HarrisRobertsEtAl2017} and references therein, see also Remark~\ref{rem:nmin0nmaxinfty}). In this section, we consider the BBMMI defined above and study its relation with $Q$.

In what follows, 
we let $m_t$ (resp $\hat{m}_t$) denote the empirical measure (resp. normalised empirical measure) of the particle system,
\begin{equation}
m_t = \sum_{i \in\bar S_t} \delta_{X_t^{i}}, \quad \text{and} \quad \hat{m}_t = \frac{1}{\bar{N}_0} m_t, \qquad t \ge 0.
\label{occmeas}
\end{equation}
The following result first provides an equality between a weighted version of the empirical measure of the system at any time $T$ and the semigroup $Q$ at time $T$. The second part is dedicated to the control of the $L^2$ distance between renormalised versions of $Q_T$ and $m_T$.

\medskip

\begin{thm} \label{thm:main}
Under Assumptions~1 and~2, the BBMMI satisfies, for all time $T\geq 0$ and all bounded measurable functions $f:E\to\mathbb{R}$, the following many-to-one formula
\begin{equation}
m_0 Q_T f=\E\left(\Pi^A_T\,\Pi^B_T m_T f\right).
\label{M21}
\end{equation}
Moreover, 
\begin{equation}
\left\|\frac{m_0Q_T f}{m_0 Q_T\mathbf{1}_E}- \frac{m_T(f)}{m_T(\mathbf{1}_E)}\mathbf 1_{\bar N_T\neq 0}\right\|_2\leq C\,\exp(c\|b\|_\infty T) \frac{\|f\|_\infty}{m_0 Q_T\mathbf 1_E/{\bar N_0}}\,\frac{1}{\sqrt{\bar N_0}},
\label{L2bound}
\end{equation}	
where $C,c$ are positive constants.
\end{thm}

\medskip

The proof of this result is given in Section~\ref{sec:proof}, where a variance estimate for the unbiased particle measure $\Pi_T^A\Pi_T^Bm_T f$ is obtained (see Step~4 therein). Let us first make some further comments on the model and the above result.

\begin{rem}
  In Theorem~\ref{thm:main}, $m_0$ is assumed deterministic with total mass $\bar N_0$. The result immediately extends to the case where $m_0$ is random by taking the expectation in~\eqref{M21} and~\eqref{L2bound} (in the latter, the right hand term may be infinite). Moreover, the expressions may also be written in terms of the normalised measure, $\hat m$, in which case \eqref{M21} becomes $\hat m_0 Q_T f=\E\left(\Pi^A_T\,\Pi^B_T \hat m_T f\right)$ and \eqref{L2bound} can be written as
\begin{equation*}
\left\|\frac{\hat m_0Q_T f}{\hat m_0 Q_T\mathbf{1}_E}- \frac{\hat m_T(f)}{\hat m_T(\mathbf{1}_E)}\mathbf 1_{\bar N_T\neq 0}\right\|_2\leq C\,\exp(c\|b\|_\infty T) \frac{\|f\|_\infty}{\hat m_0 Q_T\mathbf 1_E}\,\frac{1}{\sqrt{\bar N_0}},
\end{equation*}
making explicit the `classical' dependence ${\bar N_0}^{-1/2}$ on the number of particles which appears on the right-hand side.
\end{rem}

\smallskip

 \begin{rem}\label{rem:nmin0nmaxinfty}
     When $p=q\equiv 0$,  the particle system has the dynamics of the classical binary branching process where, from their point of creation particles move as independent copies of $X$ and, when at $x \in E$, particles are either killed at rate $\kappa(x)$ or they branch at rate $b(x)$ at which point the parent particle is replaced by two daughter particles. 
 In this situation, we have $\Pi^A_T=\Pi^B_T=1$ almost surely, and we thus recover the classical many-to-one formula:
 \begin{equation}\label{eq:M21}
   \psi_t[g](x) = \mathbf{E}_x\left[{\rm e}^{\int_0^t b(X_s) - \kappa(X_s)\mathrm d s}g(X_t)\mathbf{1}_{t < \tau_\partial} \right],
 \end{equation}
 where $\psi_t[g]$ is the linear semigroup of the binary branching process.
 
      It is also possible to treat branching processes with non-local branching. In this case, when a branching event occurs (at rate $b(x))$, the parent is removed from the system and replaced by two particles at (random) positions $x_1, x_2 \in E$, which may be different to $x$. Denoting by $(\psi_t)_{t \ge 0}$ the linear semigroup of this process, in this case, the many-to-one formula is given by
  \[
    \psi_t[g](x) = \mathbf{E}_x\left[{\rm e}^{\int_0^t b(Y_s) - \kappa(Y_s)\mathrm d s}g(Y_t)\mathbf{1}_{t < \tau_\partial} \right],
  \] 
  where $(Y_t)_{t \ge 0}$ behaves like $(X_{t})_{t \ge 0}$, except that at rate $2b(Y_t)$, the process $Y_t$ jumps to a new location given by $(\delta_{x_1} + \delta_{x_2})/2$. We refer the reader to section~\ref{sec:PDMP} for an example of a branching process with non-local branching. 
\end{rem}
  
\smallskip

\begin{rem}
Although we have remained in the setting of binary branching in this article, it should be possible to extend the definition of the BBMMI to allow for more general offspring distributions when a (natural) branching event occurs by modifying the selection mechanism. We intend to address this in future work.
\end{rem}

\smallskip

\begin{rem}
    The definition of our model allows a population with varying size by including branching, death and Moran type interactions, which are key features in the context of population dynamics. However, when restricted to the situation where $p=q\equiv 1$ (so that the size of the process remains constant equal to $\bar N_0$) we recover the standard Moran particle model (see~\cite{DelMoralMiclo2002, del2000branching,Rousset2006,CloezCorujo2021} for similar results), where the process is constrained to remain of constant size $\bar N_0$ and with uniformly bounded killing rate. 
    Our main contribution to this constant size setting is that  our assumptions allow for hard killing events (a feature observed in several models, such as diffusion processes killed at the boundary of a domain), that we consider general Markov processes (allowing implicit dependence with respect to time and with respect to the past of the process), and do not require continuity of the branching/killing rates with respect to the empirical measure of the process.
\end{rem}

\begin{rem}
    \label{rem:uniformconv}
    In the fixed size setting with bounded killing rate, it is well known that strong mixing conditions on the semi-group $Q$ are sufficient to entail uniform convergence with respect to time (see e.g. \cite{DelMoralMiclo2002,DelMoralMiclo2003,DelMoral2004,Rousset2006}, and more recently~\cite{DelVillemonais2018,OcafrainVillemonais2017,ToughNolen2020,ArnaudonMoral2020}). The same proof as in~\cite{del2000branching}, but using the convergence result of Theorem~\ref{thm:main}, also applies.
    More precisely, if one assumes that the killing rate is bounded (with no hard killing) and that there exists a probability measure $\nu_Q$ on $E$ and some constants $C_Q,\gamma_Q>0$ such that
    \begin{align}
    \label{eq:strong-mixing}
    \left\|\frac{\mu Q_t}{\mu Q_t\mathbf 1_E}-\nu\right\|_{TV}\leq C_Q e^{-\gamma_Q t},
    \end{align}
    for all $t\geq 0$ and all probability measures $\mu$ on $E$, one deduces that, for all initial empirical measure  $m_0$ of the particle system and all bounded measurable function $f:E\to\mathbb R_+$,
    \begin{align*}
    \sup_{T\in [0,+\infty)}\left\| \frac{m_T(f)}{m_T(\mathbf{1}_E)}-\frac{m_0 Q_T}{m_0 Q_T\mathbf 1_E}\right\|_2
    \leq (2C_Q+C)N_{min}^{-\alpha/2}\,\|f\|_\infty,
    \end{align*} 
    where 
    ${\displaystyle \alpha=\gamma_Q/((c+1)\|b\|_\infty+\|\kappa\|_\infty+\gamma_Q)\in(0,1)}$  and $C,c$ are from Theorem~\ref{thm:main}. In particular,
    \[
    \left\|\frac{m_T(f)}{m_T(\mathbf{1}_E)}-\nu_X(f)\right\|_2\leq  \left[(2C_Q+C)N_{min}^{-\alpha/2}+C_Q\,e^{-\gamma_Q T}\right]\,\|f\|_\infty.
    \]    
     In Section~\ref{sec:BD}, we show that such a time-uniform convergence can be proved when $\kappa$ is not bounded, when $C_Q$ depends on the initial distribution $\mu$, and when $f$ is not necessarily bounded.
\end{rem}

\smallskip

\begin{rem}
  The main property of the process $(\bar S_t, (X^i_t)_{t\in\bar S_t})_{t\geq 0}$ that is used in the proof of Theorem~\ref{thm:main} is the identity
  \begin{align}
    \label{eq:NewMarkov}
    \E^t\left[Q_{T-(t+s)}f(X^1_{t+s})\mathbf{1}_{t+s<\tau_1}\right]
      &=\texttt E^t\left[Q_{T-(t+s)}f(Y^1_{t+s})e^{-\int_t^{t+s} h_u\,\D u}\bI_{t+s<\tau_\dd}\right]\mathbf{1}_{t<\tau_1},
  \end{align}
  where $\texttt E^t$ is the expectation on a probability space $\Omega'$ such that $(Y^i_u)_{u\geq t}$, $i\in\bar S_0$, are independent copies of $X$, starting from $X^i_t$ at time $t$, and where  
\[
h_u:=\sum_{j\in\bar S_0} b^j((Y^i_u)_{i\in\bar S_0})+\kappa^j((Y^i_u)_{i\in\bar S_0}).
\]
  This identity is derived from the Markov property, and the assumption on the dynamics of the process $(\bar S_t, (X^i_t)_{t\in\bar S_t})_{t\geq 0}$. As such, it is possible to see that Theorem~\ref{thm:main} still holds under milder conditions on the construction. Specifically, suppose that there exist a filtered probability space $(\Omega, \mathcal{F}, (\mathcal{F}_t)_{t \ge 0}, \mathbb{P})$ and a progressively measurable process $(\bar S_t, (X^i_t)_{t\in\bar S_t})_{t\geq 0}$ defined on this space, such that the rates $b^i, \kappa^i, p^i,q^i$ are no longer functions, but merely progressively measurable processes, satisfying the equivalents\footnote{Technically, we need that the integrals of the respective processes are equal up to modification, that is, if we define $I^{i,1}_t:= \int_0^t (b_s^i-\kappa_s^i) \, \D s$ and $I^{i,2}_t:= \int_0^t (b^i(X_s^i)-\kappa(X_s^i)) \, \D s$, then we require the processes $I^{i,1}$ and $I^{i,2}$ to agree up to a modification. Note that as a consequence of the difference $b(\cdot)-\kappa(\cdot)$ being bounded, this implies that the processes are also indistinguishable.}
  of \eqref{eq:balance}, \eqref{eq:constraint}:
  \begin{align*}
    b^i_t-\kappa^i_t & = b(X_t^{i})-\kappa(X_t^i)    \\
    p^i_t & = 0 \quad \text{ whenever } |\bar S_t| = 1.
  \end{align*}
  Then the proof of Theorem~\ref{thm:main} still holds, provided that condition \eqref{eq:NewMarkov} can be verified in the more general case, and where we interpret $\E^t$ as the expectation conditional on the full sigma-algebra $\mathcal{F}_t$. Of course, such a property is something that may, in certain settings, be obtained through construction.

  An example where such a construction might be natural is in a distributed computing environment. In this case, one may look to break the particles into subsets which are handled by separate processors with minimal communication. In this case, one may choose to impose $N_{min}-N_{max}$-like criterion on each separate processor. Roughly, each processor might be allocated an initial set of particles, which evolve without communication (by only having killing/branching events), so long as the number of particles remains within $[N_{min},N_{max}]$. Since global properties of the particle system are only needed when resampling or selection events occur, the need for communication events in the system, which cause computational bottlenecks, could be reduced by choosing to allow killing or branching events provided the number of particles on a given processor does not move outside a specified range. In such a case, the number of particles on a given processor would depend on the allocation of particles to given processors (and which may itself be randomised independently of the history of the particle system). In practice, communication between servers will still be necessary, but potentially less frequently, and hence with lower computational burden, than with other methods. In these examples, the processes $b^i$ and $\kappa^i$ would not be simply functions of the particle positions (since we also need to know which particles are on the same processor), but would be expected to be $(\mathcal{F}_t)_{t \ge 0}$-progressively measurable, for an appropriate filtration, which may then be larger than the filtration generated by the particle histories.

\end{rem}

\section{Applications}
\label{sec:apps}

\subsection{Population size dynamics with constrained population size}\label{sec:BD}

In this section, we introduce a  branching model for a population where the size is constrained between $N_{min}\geq 2$ and $N_{max}\geq N_{min}$ and demonstrate that the results from the previous section may be applied. In this model, we associate with each individual a health state $x\in E:= \mathbb Z_+^d$, $d\geq 1$, which evolves according to the dynamics of a multi-dimensional birth and death process during the life of an individual. The parameters of the pure jump Markovian dynamics are denoted by
\[
\begin{cases}
\beta_i(x)&=\text{ transition rate from $x$ to $x+e_i$},\\
\delta_i(x)&=\text{ transition rate from $x$ to $x-e_i$},
\end{cases}
\]
where $e_i \in \mathbb Z_+^d$ denotes the vector with $1$ in the $i^{th}$ position and $0$ everywhere else, $\beta_i, \delta_i:\mathbb Z_+^d\to[0,+\infty)$ for all $i\in\{1,\ldots,d\}$, and $\delta_i(x)=0$ for all $x=(x_1,\ldots,x_d)$ such that $x_i=0$ (so that the process cannot leave the set $\mathbb Z_+^d$).

When its health state is $x\in\mathbb Z_+^d$, at rate $\kappa(x)$ an individual dies, where $\kappa:\mathbb Z_+^d\to [0,+\infty)$, and at rate $b(x)$, it gives birth to an identical individual with equal health state $x$, where $b:\mathbb Z_+^d\to [0,+\infty)$ is bounded. Note that, in our model, a higher health state may  thus be associated with a worse health condition. This corresponds, for instance, to the situation where $x$ is the size of a (harmful) parasite population or the concentration of a virus.

In addition, an external mechanism controls the total size of the individuals population: the total number of individuals is bounded above by $N_{max}$ by removing an individual chosen randomly and uniformly in the population when the population size reaches $N_{max}+1$, and the total number of individuals is lower bounded by $N_{min}$ by cloning an individual chosen randomly and uniformly when the population size reaches $N_{min}-1$ (in the situation where $N_{min}=0$, this last mechanism does not operate).

Formally, the state space of our population model is $F:=\bigcup_{N=N_{min}}^{N_{max}} [E]^N$, where $[E]^N$ denotes the unordered $N$-tuples of $E$. We represent the elements of $F$ by $\bx=[\bx_1,\ldots,\bx_N]$ and note that repetitions are allowed. The extended infinitesimal generator of the process acting on bounded mesurable functions $f:F\to\mathbb R$, is given, for any $\bx=[\bx_1,\ldots,\bx_N]\in F$, by 
\begin{align*}
\bar L f(\bx)&=\sum_{k=1}^N \sum_{i=1}^d \left(\beta_i(\bx_{k,i})\Delta_{\bx,\bx+e^N_{k,i}}f+\delta_i(\bx_{k,i})\Delta_{\bx,\bx-e^N_{k,i}}f\right) \\
&\qquad\qquad +\sum_{k=1}^N b(\bx_k)\Delta_{\bx,\bx\cup[\bx_k]}f+\sum_{k=1}^N\kappa(\bx_k)\Delta_{\bx,\bx\setminus[\bx_k]}f(\bx)\\
&\qquad\qquad+\mathbf{1}_{N=N_{min}}\sum_{k=1}^N \kappa(\bx_k)\frac{1}{N-1}\sum_{\ell=1,\ell\neq k}^N \Delta_{\bx\setminus[\bx_k],\bx\cup[\bx_\ell]\setminus[\bx_k]}f\\
&\qquad\qquad+\mathbf{1}_{N=N_{max}}\sum_{k=1}^N b(\bx_k)\frac{1}{N+1}\sum_{\ell=1}^N (1+\mathbf{1}_{\ell=k})\Delta_{\bx\cup[\bx_k],\bx\cup[\bx_k]\setminus[\bx_\ell]}f,
\end{align*}
where 
\begin{itemize}
    \item $\bx_{k,i}$ is the $i^{\text{th}}$ component of $\bx_k$,
    \item  $\bx+e^N_{k,i}$ is $\bx$ but with $\bx_{k,i}$ replaced by $\bx_{k,i}+1$,
    \item $\bx-e^N_{k,i}$ is $\bx$ but with $\bx_{k,i}$ replaced by $\bx_{k,i}-1$,
    \item for all $\mathbf x,\mathbf y$, $\Delta_{\bx,\mathbf y}f:=f(\mathbf y)\mathbf 1_{\mathbf y\in F}-f(\bx)\mathbf 1_{\mathbf x\in F}$,
    \item for all $y\in E$, $\bx\cup[y]=[\bx_1,\ldots,\bx_N,y]$,
    \item $\bx\setminus[\bx_k]=[\bx_1,\ldots,\bx_{k-1},\bx_{k+1},\ldots,\bx_N]$.
\end{itemize} 
In a similar manner to the previous section, we denote by $A_T$ the number of times an individual is duplicated by the size constraining mechanism before time $T$, and by $B_T$ the number of times an individual is removed from the population by the size constraining mechanism. We also set $(\bX(t))_{t\geq 0}$ to be the population process, $(|\bX(t)|)_{t\geq 0}$ its size, and $(X_t)_{t\in\geq 0}$ a multi-dimensional birth and death process with (state dependent) transition rates $\beta$ and $\delta$. The notation $\mathbb P_\bx$ (respectively $\mathbf P_x$) is used for the law of the population process $\bX$ starting from $\bx\in F$ (respectively for the law of $X$ starting from $x\in E$). We denote as usual by $\mathbb E_x$ and $\mathbf E_x$ the associated expectations.

The first difficulty is to ensure that the constrained branching model defined above does not explode in finite time, where explosion is interpreted in the sense that infinitely many jumps of the system happen in finite time. Indeed, since $\kappa$ is not assumed to be bounded, the rate at which new individuals are added to the system is not bounded. 
The following result provides a sufficient criterion for the non-explosion of the process but before stating it, let us first introduce some definitions that will be used in what follows. 

We say that two functions $V$ and $\kappa$ are \textit{co-monotone} if $(\kappa(x)-\kappa(y))(V(x)-V(y))\geq 0$ for all $x,y\in E$. We say that $V$ and $\kappa$ are \textit{almost co-monotone} if there exists a function $\kappa'$ co-monotone with $V$ such that $\kappa-\kappa'$ is bounded. We say that $V$ \textit{goes to infinity at infinity} if the set of $x \in E$ such that $V(x) < c$ is finite for all $c >0$.

\begin{prop}
    \label{prop:nonexplPopModel}
    Assume that there exists a  function $V: E\to[1,+\infty)$ almost co-monotone with $\kappa$, such that $V$ tends to infinity at infinity, and such that, for all $x\in E$,
    \begin{align}
    \label{eq:lyapa}
    \sum_{i=1}^d \left[\beta_i(x) (V(x+e_i)-V(x))+\delta_i(x)(V(x-e_i)-V(x))\right] \leq CV(x),
    \end{align}
    for some constant $C\in\mathbb{R}$. Then the process with transition rates given by $\bar L$ is non-explosive.
\end{prop}

\begin{rem}
    The proof of this non-explosion result makes use of the extended infinitesimal generator of the  particle system and of Dynkin's formula. As a consequence, its proof can be extended without too much effort to more general continuous time Markov processes, by replacing the left hand term in~\eqref{eq:lyapa} by the extended infinitesimal generator of the more general process. One difficulty can appear though when the process is not a simple pure jump process on a discrete state space:  to prove  in a general context that $\bar L$ (adapted to fit the general individual dynamics of the particles) is indeed the extended infinitesimal generator of the particle system can be challenging.
\end{rem}

\begin{rem}
    More complex resampling/selection mechanisms could be considered with only slight modifications of the proof of Proposition~\ref{prop:nonexplPopModel}. For instance, the main line of the proof applies as well if $N_{max}=+\infty$, or if we assume that, immediately after each branching time and with probability $1+\frac1{N_{max}+1}-\frac1N$, a uniformly chosen individual is removed from the population, and/or that, immediately after each killing time and with probability $\frac{N_{min}-1}{N}$, a uniformly chosen individual is duplicated (where $N$ is the population size just after the event). Note that these choices of parameters may be used to model an increasing competition between individuals when the size of the population increases.
\end{rem}

\begin{proof}[Proof of Proposition~\ref{prop:nonexplPopModel}]
    We define 
    \begin{align*}
    \bar V:&F\to \mathbb [1,+\infty)\\
    :&\bx=[\bx_1,\ldots,\bx_N]\mapsto \sum_{k=1}^N V(\bx_k),
    \end{align*}
    so that $\bar V$ goes to infinity at infinity and for $\bx = [\bx_1,\ldots,\bx_N] \in F$,
    \begin{align*}
    \bar L \bar V(\bx)&=\sum_{k=1}^N \sum_{i=1}^d \left[\beta_i(\bx_k) (V(\bx_k+e_i)-V(\bx_k))+\delta_i(\bx_k)(V(\bx_k-e_i)-V(\bx_k))\right]\\
    &\qquad\qquad +\sum_{k=1}^N \left(b(\bx_k)-\kappa(\bx_k)\right)V(\bx_k)\\
    &\qquad\qquad+\mathbf{1}_{N=N_{min}}\sum_{k=1}^N \kappa(\bx_k)\frac{1}{N-1}\sum_{\ell=1,\ell\neq k}^N V(\bx_l)\\
    &\qquad\qquad-\mathbf{1}_{N=N_{max}}\sum_{k=1}^N b(\bx_k)\frac{1}{N+1}\sum_{\ell=1}^N (1+\mathbf{1}_{\ell=k})V(\bx_\ell).
    \end{align*}
    Now let $\kappa'$ be co-monotone with $V$ such that $\kappa''=\kappa-\kappa'$ is bounded. Then, using the fact that $b$ is bounded, we deduce that 
    \begin{align}
    \bar L \bar V(\bx)
    &\leq  (C+\|b\|_\infty)\overline{V}(\bx)+\mathbf{1}_{N=N_{min}}\left(\frac{1}{N-1}\sum_{k=1}^N \kappa(\bx_k)\sum_{{\stackrel{\ell=1}{\ell\neq k}}}^N V(\bx_\ell)-\sum_{k=1}^N\kappa(\bx_k) V(\bx_k)\right)\nonumber\\
    \notag\\
    &=(C+\|b\|_\infty)\overline{V}(\bx)+\frac{\mathbf{1}_{N=N_{min}}}{N-1}\left(\sum_{k=1}^N \kappa(\bx_k)\sum_{\ell=1}^N V(\bx_\ell)-N\sum_{k=1}^N \kappa(\bx_k)V(\bx_k)\right)\nonumber\\
    \notag\\
    &=  (C+\|b\|_\infty)\overline{V}(\bx)+\frac{N^2\mathbf{1}_{N=N_{min}}}{N-1}\left[\left(\sum_{k=1}^N \frac{\kappa(\bx_k)}{N}\right)\left(\sum_{k=1}^N \frac{V(\bx_k)}{N}\right)-\sum_{k=1}^N \frac{\kappa(\bx_k)V(\bx_k)}{N} \right]\nonumber\\
    \notag\\
    &\leq (C+\|b\|_\infty+4\|\kappa''\|_\infty)\overline{V}(\bx)\nonumber\\
    &\qquad\qquad+\frac{N^2\mathbf{1}_{N=N_{min}}}{N-1}\left[\left(\sum_{k=1}^N \frac{\kappa'(\bx_k)}{N}\right)\left(\sum_{k=1}^N \frac{V(\bx_k)}{N}\right)-\sum_{k=1}^N \frac{\kappa'(\bx_k)V(\bx_k)}{N} \right]\nonumber\\
    \notag\\
    &\leq (C+\|b\|_\infty+4\|\kappa''\|_\infty)\overline{V}(\bx),\label{eq:inequseful}
    \end{align}
    where we used the fact that $\kappa'$ and $V$ are co-monotone and the FKG inequality for the last inequality.
    
    \smallskip
    
    For all $n\geq 1$, let $U_n=\{\bx\in F,\ \max_{k\in\{1,\ldots,|\bx|\}} \bx_k> n\}$ and define the random variable $\theta_n:= \inf\{t\geq 0,\ (X_t^1,\ldots,X_t^N)\in U_n\}$. Note that $\theta_n$ is a stopping time since the constrained branching process is c\`adl\`ag.
    Moreover, denoting by $C_n$ the set of points in $F$ that can be reached in one jump from $F\setminus U_n$, we observe that $C_n$ is bounded and hence that  $\bar V(X_{t\wedge \theta_n})$ is almost surely uniformly bounded in $t\geq 0$. In particular, we deduce, using Dynkin's formula and the inequality~\eqref{eq:inequseful}, that, for all $n\geq 0$, all $t\geq 0$ and all $\bx\in F$, 
    \[
    \E_\bx(\bar V(\bX_{t\wedge\theta_n}))\leq \bar V(\bx)+(C+\|b\|_\infty+4\|\kappa''\|_\infty)\int_0^t \E_x(\bar V(\bX_{s\wedge\theta_n}))\,\D s.
    \]
    Hence, using Gr\"onwall's inequality, we deduce that
    \begin{align*}
    \E_\bx(\bar V(\bX_{t\wedge\theta_n}))\leq \bar V(\bx)e^{t(C+\|b\|_\infty+4\|\kappa''\|_\infty)}.
    \end{align*}
    Since $\{\bX_{t\wedge\theta_n}\in U_n,\theta_n\leq t\}=\{\theta_n\leq t\}$ up to a negligible event, we conclude that
    \[
    \inf_{\mathbf y\in U_n} \bar V(\mathbf y) \mathbb{P}_\bx(\theta_n\leq t)\leq \bar V(\bx)\,e^{t(C+\|b\|_\infty+4\|\kappa''\|_\infty)}.
    \]
    Denote by $\tau_\infty$ the explosion time of the process. Since the jump rate of the process is bounded up to time $\theta_n$, we deduce that $\theta_n\leq \tau_\infty$ almost surely. Hence
    \[
    \mathbb{P}_\bx(\tau_\infty\leq t)\leq \frac{\bar V(\bx)e^{t(C+\|b\|_\infty+4\|\kappa''\|_\infty)}}{\inf_{\mathbf y\in U_n} \bar V(\mathbf y)}\xrightarrow[n\to+\infty]{}0,
    \]
    so that $\tau_\infty=+\infty$ almost surely. This concludes the proof.
\end{proof}

Since the process defined in this section is a particular instance of the process described in Section~\ref{sec:descr}, the following corollary is a direct consequence of Proposition~\ref{prop:nonexplPopModel} and Theorem~\ref{thm:main}.

\begin{cor}
    Under the assumptions of Proposition~\ref{prop:nonexplPopModel}, we have, for all $T>0$, all bounded measurable functions $f:\mathbb Z_+^d\to \mathbb R$ and all initial position $\bX(0)\in F$,
    \[
    \mathbb E_{\bX(0)}\left[\Pi^A_T\Pi^B_T\sum_{k=1}^{|\bX(T)|}f(\bX(T)_k)\right]=\sum_{k=1}^{|\bX(0)|} \mathbf E_{\bX(0)_k}\left[f(X_T)\exp\left(\int_0^T (b(X_s)-\kappa(X_s))\,\mathrm ds\right)\right]
    \]
    where
    \[
    \Pi^A_T=\left(\frac{N_{min}-1}{N_{min}}\right)^{A_T}\text{ and }\Pi^B_T=\left(\frac{N_{max}+1}{N_{max}}\right)^{B_T}.
    \]
    In addition, for some constant $C>0$,
    \begin{multline*}
    \left\|\frac{\sum_{k=1}^{|\bX(0)|} \mathbf E_{\bX(0)_k}\left[f(X_T)\exp\left(\int_0^T (b(X_s)-\kappa(X_s))\,\mathrm ds\right)\right]}{\sum_{k=1}^{|\bX(0)|} \mathbf E_{\bX(0)_k}\left[\exp\left(\int_0^T (b(X_s)-\kappa(X_s))\,\mathrm ds\right)\right]}- \frac{\sum_{k=1}^{|\bX(T)|}f(\bX(T)_k)}{|\bX(T)|}\right\|_2\\
    \leq \frac{C  \exp(\|b\|_\infty T) \|f\|_\infty}{\frac{1}{|\bX(0)|}\sum_{k=1}^{|\bX(0)|} \mathbf E_{\bX(0)_k}\left[\exp\left(\int_0^T (b(X_s)-\kappa(X_s))\,\mathrm ds\right)\right]}\,\frac{1}{\sqrt{|\bX(0)|}}.
    \end{multline*}
\end{cor}

Our aim is now to study the long-time behaviour of the (normalised) empirical distribution of the health parameter across the population. As mentioned in Remark~\ref{rem:uniformconv}, adaptation of classical arguments from~\cite{del2000branching,DelMoralMiclo2002,DelMoralMiclo2003,DelMoral2004} shows that bounded killing rate and a strong mixing conditions on the semi-group $Q$ are sufficient to obtain uniform in time estimates. Using for instance the results of~\cite{ChampagnatVillemonais2021}, one can checks that this is the case if $\kappa$ is bounded and if  there exist constants $C_1>0,C_2>0,\eta_0>0,n_0\geq 0$ such that, for all $x\in\mathbb Z_+^d$ satisfying $|x|\geq n_0$,
\begin{align}
&\label{eq:assumptionQSD1}
\sum_{i=1}^d (\delta_i(x)-\beta_i(x))\geq C_1\,|x|^{1+\eta_0}\\
&\label{eq:assumptionQSD3} \sum_{i=1}^d (\delta_i(x)-\beta_i(x))\geq C_2(\|b\|_\infty-b(x)+\kappa(x))|x|.
\end{align}

We consider now the more delicate situation where mixing condition is weaker, that is when~\eqref{eq:strong-mixing} holds true but with $C_Q$ depending on the initial distribution. More precisely, we consider the situation where there exist two positive functions $\psi_1,\psi_2:E\to (0,+\infty)$ such that $\psi_2\leq \psi_1$ and such that, for all $f:E\to\mathbb R$ with $|f|\leq \psi_1$, we have
\begin{align*}
\left|\frac{\mathbf E_\mu\left(f(X_t)\exp\left(\int_0^t b(X_s)-\kappa(X_s)\,\mathrm ds\right) \right)}{\mathbf E_\mu\left(\exp\left(\int_0^t b(X_s)-\kappa(X_s)\,\mathrm ds\right) \right)}-\nu_{X}(f)\right|\leq C\,e^{-\gamma t}\frac{\mu(\psi_{1})}{\mu(\psi_{2})},\quad\forall t\geq 0,
\end{align*}
where $\nu_X$ is a probability measure on $E$. This type of properties has been extensively studied in the past years and several criteria have been provided (see e.g.~\cite{ChampagnatVillemonais2017a,FerreRoussetEtAl2018,HinrichsKolbEtAl2018,GuillinNectouxEtAl2020,BansayeCloezEtAl2019,DelMoralHortonEtAl2022}).

Compared to the strong mixing condition mentioned in Remark~\ref{rem:uniformconv}, the novelty of the following result is that we consider unbounded killing rates, that we allow unbounded test functions $f$ and that the convergence of the renormalized semi-group is non-uniform in the initial distribution.

\begin{thm}
  Assume that there exists $a_0>0$ such that
\begin{align*}
\theta_N(x) \xrightarrow[|x|\to+\infty]{} +\infty.
\end{align*}
where
\begin{align*}
\theta_N(x):=\sum_{i=1}^d \left[(e^{-a_0}-1)\delta_i(x)-(e^{a_0}-1) \beta_i(x)\right]+\frac{N}{N-1}\left(\kappa(x)-\max_{|y|\leq |x|} \kappa(y)\right).
\end{align*}
Then the process is non-explosive and, for any sequence of initial positions $(\bX^{(N_{min})}(0))_{N_{min}\geq 2}$ such that
\begin{align}
 \label{eq:boundonLyapouX}
\sup_{N_{min}\geq 2 }\frac{1}{|\bX^{(N_{min})}(0)|}\sum_{k=1}^{|\bX^{(N_{min})}(0)|}\exp\left(a_0|\bX^{(N_{min})}_k(0)|\right)<+\infty,
\end{align} 
we have, for all $f:E\to\mathbb R$ such that $f(x)\leq \exp(a x)$ for some $a\in(0,a_0)$,
\begin{align*}
\limsup_{N_{min}\to+\infty}\sup_{T\in[0,+\infty)} \mathbb E_{\bX^{(N_{min})}(0)}\left|\frac{m_0 Q_T f}{m_0 Q_T\ind_E}-\widetilde m_T f\right|=0.
\end{align*}
\end{thm}

\begin{rem}
    For the sake of readability, we only provide a proof for the particular model of this section. 
    However, the method used therein can be easily adapted to a more general situation, at the expense of additional technicalities which depend on the underlying dynamic of the particles. In the general case, and technical difficulties put aside, $\theta_N(x)$ should be replaced by $-\frac{\mathcal LW(x)}{W(x) }+\frac{N}{N-1}\left(\kappa(x)-\max_{|y|\leq |x|} \kappa(y)\right)$ for some well chosen Lyapunov function $W$ which goes to infinity at infinity.
\end{rem}

\begin{rem}
    This result can be adapted to include the case where $X$ is a Markov process with \textit{hard} killing at a boundary, provided one can prove that the particle system does not degenerate toward the boundary. However this can be challenging, see for instance~\cite{GrigorescuKang2007,Villemonais2011,BieniekBurdzyEtAl2012}.
\end{rem}

\begin{proof}
    Let us first observe that our assumptions entail that, for all $a\in (0,a_0]$,
    \[
    (e^{-a}-1)\sum_{i=1}^d \delta_i(x)-(e^{a}-1)\sum_{i=1}^d \beta_i(x)\xrightarrow[|x|\to+\infty]{} +\infty.
    \]
    According to~\cite{ChampagnatVillemonais2017a} and setting $\psi_{1,a}=\exp(a|x|)$ for all $a\in (0,a_0]$ and $x\in E=\mathbb Z_+^d$, this implies that there exist a bounded positive function $\psi_{2,a}$ and $c_{2,a},\gamma_{2,a}>0$ such that $Q_t\psi_{2,a}\geq c_{2,a}e^{-\gamma_{2,a}t}$, a probability measure $\nu_{QS}$ on $E$ such that, for all $f:E\to\mathbb R$ with $|f|\leq \psi_{1,a}$, 
    \begin{align}
    \label{eq:conv-QSD}
    \left|\frac{\mathbf E_\mu\left(f(X_t)\exp\left(\int_0^t b(X_s)-\kappa(X_s)\,\mathrm ds\right) \right)}{\mathbf E_\mu\left(\exp\left(\int_0^t b(X_s)-\kappa(X_s)\,\mathrm ds\right) \right)}-\nu_{X}(f)\right|\leq C_a\,e^{-\gamma_a t}\frac{\mu(\psi_{1,a})}{\mu(\psi_{2,a})},\quad\forall t\geq 0,
    \end{align}
    where $C_a$ and $\gamma_a$ are positive constant that may depend on $a\in(0,a_0]$. Note that, in order to apply the results of~\cite{ChampagnatVillemonais2017a}, we can choose $\psi_{2,a}=\psi_{2,a_0}$ for all $a\in(0,a_0]$, so that we can choose $\psi_{2,a}$ (resp. $c_{2,a}$, $\gamma_{2,a}$) to not depend on $a$. For this reason, in what follows, we will use instead the notation $\psi_2$ (resp. $c_2$, $\gamma_2$).

    We define the functions $W=\psi_{1,a_0}$ and  $\bar W: F\to \mathbb R_+$ by
    \begin{align*}
    \bar W(\mathbf x)=\frac1N \sum_{k=1}^N W(\mathbf x_k)=\frac1N \sum_{k=1}^N \exp(a_0\mathbf x_k).
    \end{align*}
    One checks that, for all $\mathbf x\in F$,
    \begin{align*}
    \bar L\bar W(\mathbf x)&=\frac1N\sum_{k=1}^N \sum_{i=1}^d \left[\beta_i(\bx_k) (W(\bx_k+e_i)-W(\bx_k))+\delta_i(\bx_k)(W(\bx_k-e_i)-W(\bx_k))\right]\\
    &\quad\quad +\sum_{k=1}^N b(\bx_k)\left(\frac{W(\bx_k)}{N+1}-\frac{\bar W(\bx)}{N+1}\right)+\sum_{k=1}^N \kappa(\bx_k)\left(-\frac{W(\bx_k)}{N-1}+\frac{\bar W(\bx)}{N-1}\right)\\
    &\leq -\frac1N\sum_{k=1}^d\theta_N(\bx_k)W(\bx_k)+\|b\|_\infty \bar W(\bx) +\frac{N}{N-1}\left(\frac1N\sum_{k=1}^N \kappa(\bx_k)\right)\left(\frac1N\sum_{k=1}^N W(\bx_k)\right)\\
    &\quad\quad-\frac1{N-1}\sum_{k=1}^N(\max_{|y|\leq\bx_k} \kappa(y))W(\bx_k)\\
    &\leq -\frac1N\sum_{k=1}^d(\theta_N(\bx_k)-\|b\|_\infty)W(\bx_k).
    \end{align*}
    
    Using Dynkin's formula and the fact that $\theta_N(x)\to+\infty$ when $|x|\to+\infty$, this classically entails that the process is non-explosive and that there exists some constant $M>0$ such that, for any $\bX(0)\in F$, 
    \begin{align}
    \label{eq:supExp}
    \sup_{t\geq 0} \mathbb  E_{\bX(0)}(\bar W(\bX(t)))\leq \bar W(\bX(0))+M,
    \end{align}
    and hence, using~\eqref{eq:boundonLyapouX}, we deduce that
    \begin{align}
    \label{eq:Atoinfty}
    \sup_{t\geq 0} \mathbb P_\bx \left(\bar W(\bX_t)\geq A\right)\leq \sup_{t\geq 0} \frac{\mathbb  E(\bar W(\bX(t)))}{A} \xrightarrow[A\to+\infty]{} 0.
    \end{align}
    Since $\psi_{2}$ is positive (and hence locally bounded away from $0$) and $W(x)\to+\infty$ when $|x|\to+\infty$, we also deduce that
    \begin{align}
    \label{eq:epstoinfty}
    \sup_{t\geq 0} \mathbb P_\bx \left(\bar\psi_{2}(\bX(t))\leq \varepsilon\right)\xrightarrow[\varepsilon\to0]{}0,
    \end{align}
    where $\bar\psi_{2}(\bx)=\frac1N\sum_{k=1}^N \psi_{2}(\bx_k)$.
    
 From now on, fix $a\in(0,a_0)$ and let $f:E\to \mathbb R$ such that $|f|\leq \psi_{1,a}$. Then, for all $T\geq 0$, we have, setting $p=a_0/a>1$ and using Jensen's inequality,
    \begin{align}
    \mathbb E_{\bX(0)}\left|\frac{m_0 Q_T f}{m_0 Q_T\ind_E}-\widetilde m_T f\right|
    &\leq \mathbb E_{\bX(0)}\left(\left|\frac{m_0 Q_T f}{m_0 Q_T\ind_E}-\widetilde m_T f\right|\ind_{\widetilde m_{T-t}\psi_{1,a}\leq A\text{ and }\widetilde m_{T-t}\psi_{2}\geq \varepsilon}\right)\nonumber\\
    &\qquad +\frac{m_0 Q_T \psi_{1,a}}{m_0 Q_T\ind_E}\mathbb P_{\bX(0)}\left(\widetilde m_{T-t}\psi_{1,a}> A\text{ or }\widetilde m_{T-t}\psi_{2}< \varepsilon\right) \nonumber\\
    &\qquad+\mathbb E_{\bX(0)}\left(\widetilde m_T W\right)\mathbb P_{\bX(0)}\left(\widetilde m_{T-t}\psi_{1,a}> A\text{ or }\widetilde m_{T-t}\psi_{2}< \varepsilon\right),\label{eq:ineq1}
    \end{align}
    where the second and third terms go to $0$ when $A$ and $\varepsilon$ go to $0$ by~\eqref{eq:Atoinfty} and~\eqref{eq:epstoinfty}.
    Writing $\mathcal E_{A,\varepsilon}=\left\{\widetilde m_{T-t}\psi_{1,a}\leq A\text{ and }\widetilde m_{T-t}\psi_{2}\geq \varepsilon\right\}$, we decompose the first term as follows: for any $t\in [0,T]$, we have
    \begin{align}
    \mathbb E_{\bX(0)}\left(\left|\frac{m_0 Q_T f}{m_0 Q_T\ind_E}-\widetilde m_T f\right|\ind_{\mathcal E_{A,\varepsilon}}\right)
    &\leq \mathbb E_{\bX(0)}\left(\left|\frac{m_0 Q_T f}{m_0 Q_T\ind_E}-\frac{\widetilde m_{T-t} Q_t f}{\widetilde m_{T-t} Q_t \ind_E}\right|\ind_{\mathcal E_{A,\varepsilon}}\right)\nonumber \\
    &\qquad +
    \mathbb E_{\bX(0)}\left(\left|\frac{\widetilde m_{T-t} Q_t f}{\widetilde m_{T-t} Q_t \ind_E}-\widetilde m_T f\right|\ind_{\mathcal E_{A,\varepsilon}}\right)\nonumber \\
    &\leq C_ae^{-\gamma_a t}\frac{A}{\varepsilon}+ \mathbb E_{\bX(0)}\left(\left|\frac{\widetilde m_{T-t} Q_t f}{\widetilde m_{T-t} Q_t \ind_E}-\widetilde m_T f\right|\ind_{\mathcal E_{A,\varepsilon}}\right),\label{eq:ineq2}
    \end{align}
    where we used~\eqref{eq:conv-QSD}.  Then, for any $K>0$, we have
    \begin{align*}
    \mathbb E_{\bX(0)}\left(\left|\frac{\widetilde m_{T-t} Q_t f}{\widetilde m_{T-t} Q_t \ind_E}-\widetilde m_T f\right|\ind_{\mathcal E_{A,\varepsilon}}\right)
    &\leq \mathbb E_{\bX(0)}\left(\left|\frac{\widetilde m_{T-t} Q_t (f\wedge K)}{\widetilde m_{T-t} Q_t \ind_E}-\widetilde m_T (f\wedge K)\right|\ind_{\mathcal E_{A,\varepsilon}}\right)\\
    &\qquad + \mathbb E_{\bX(0)}\left(\frac{\widetilde m_{T-t} Q_t (\psi_{1,a}\ind_{\psi_{1,a}> K})}{\widetilde m_{T-t} Q_t \ind_E}+\widetilde m_T (\psi_{1,a}\ind_{\psi_{1,a}> K})\right)\\
    &\leq \mathbb E\left(\frac{\Vert Q_t\mathbf{1}_E\Vert_{\infty}}{\widetilde m_{T-t} Q_t\mathbf 1_E}\frac{K}{\sqrt{ m_{T-t}\mathbf 1_E}}\ind_{\mathcal E_{A,\varepsilon}}\right)\\
    &\qquad +K^{1-a_0/a}\mathbb E_{\bX(0)}\left(\frac{\widetilde m_{T-t} Q_t W}{\widetilde m_{T-t} Q_t \ind_E}+\widetilde m_T W\right)
    \end{align*}
    where we used Theorem~\ref{thm:main}.  
    According to~\cite{ChampagnatVillemonais2017a}, there exists some constant $C>0$ such that
    \begin{align*}
    \frac{\widetilde m_{T-t} Q_t W}{\widetilde m_{T-t} Q_t \ind_E}\leq C+\widetilde m_{T-t} W=C+\bar W(\bX(t)).
    \end{align*}
    Using in addition~\eqref{eq:supExp} and the fact that $m_{T-t}\mathbf 1_E\geq N_{min}$, we deduce that 
    \begin{align*}
    \mathbb E_{\bX(0)}\left(\left|\frac{\widetilde m_{T-t} Q_t f}{\widetilde m_{T-t} Q_t \ind_E}-\widetilde m_T f\right|\ind_{\mathcal E_{A,\varepsilon}}\right)
    &\leq \mathbb E\left(\frac{\Vert Q_t\mathbf{1}_E\Vert_{\infty}}{\widetilde m_{T-t} Q_t\mathbf 1_E}\frac{K}{\sqrt{ N_{min}}}\ind_{\mathcal E_{A,\varepsilon}}\right)+K^{1-a_0/a}(C+2M+2M'),
    \end{align*}
    where $M'>0$ is the finite supremum in~\ref{eq:boundonLyapouX}.
     Assuming without loss of generality that $\psi_{2}\leq 1$, we observe that 
    \[
    \widetilde m_{T-t} Q_t\mathbf 1_E\geq \widetilde m_{T-t} Q_t\psi_{2}\geq c_{2}e^{-\gamma_2 t} \widetilde m_{T-t}\psi_{2}\geq c_{2}e^{-\gamma_2 t} \varepsilon\text{ on }\mathcal E_{A,\varepsilon},
    \]
    so that, using the inequality $\Vert Q_t\mathbf{1}_E\Vert_{\infty}\leq e^{\|b\|_\infty t}$,
    \begin{align*}
    \mathbb E_{\bX(0)}\left(\left|\frac{\widetilde m_{T-t} Q_t f}{\widetilde m_{T-t} Q_t \ind_E}-\widetilde m_T f\right|\ind_{\mathcal E_{A,\varepsilon}}\right)
    &\leq \frac{e^{\|b\|_\infty t}}{c_{2,a}e^{-\gamma_2 t}\varepsilon}\frac{K}{\sqrt{ N_{min}}}+K^{1-a_0/a}(C+2M+2M').
    \end{align*}
    Finally, using this,~\eqref{eq:ineq1} and~\eqref{eq:ineq2}, we deduce that
    \begin{align*}
    \mathbb E_{\bX(0)}\left|\frac{m_0 Q_T f}{m_0 Q_T\ind_E}-\widetilde m_T f\right|
    &\leq  C_ae^{-\gamma_a t}\frac{A}{\varepsilon}+ \frac{e^{\|b\|_\infty t}}{c_{2,a}e^{-\gamma_2 t}\varepsilon}\frac{K}{\sqrt{ N_{min}}}+K^{1-a_0/a}(C+2M+2M')+g(A,\varepsilon),
    \end{align*}
    where $g(A,\varepsilon)\to0$ when $A$ or $\varepsilon$ goes to $0$. 
        We easily deduce that, for all $\eta>0$, the exists $T_\eta>0$ such that
    \begin{align*}
    \limsup_{N_{min}\to+\infty}\sup_{T\geq T_\eta} \mathbb E_{\bX(0)}\left|\frac{m_0 Q_T f}{m_0 Q_T\ind_E}-\widetilde m_T f\right|\leq \eta.
    \end{align*}
    In addition, it is straightforward from~Theorem~6 that
    \begin{align*}
    \limsup_{N_{min}\to+\infty}\sup_{T< T_\eta} \mathbb E_{\bX(0)}\left|\frac{m_0 Q_T f}{m_0 Q_T\ind_E}-\widetilde m_T f\right|=0.
    \end{align*}
    We conclude that
    \begin{align*}
    \limsup_{N_{min}\to+\infty}\sup_{T\in[0,+\infty)} \mathbb E_{\bX(0)}\left|\frac{m_0 Q_T f}{m_0 Q_T\ind_E}-\widetilde m_T f\right|=0.
    \end{align*}
\end{proof}

\subsection{Piecewise deterministic Markov processes}\label{sec:PDMP}
Due to Assumption 1 in Section \ref{sec:main}, the results presented thus far do not hold for the case where $(X_t)_{t \ge 0}$ is a piecewise deterministic Markov process (PDMP) and $E$ is a bounded domain with absorbing boundary conditions, since two independent copies of the process may hit $\partial E$ at the same time. In this section, we consider a particular example, namely the neutron transport equation (NTE), in order to discuss a way to avoid this problem. 

The NTE is a balance equation that describes the behaviour of neutrons in a fissile medium such as a nuclear reactor. In such systems neutrons move in straight lines with a fixed speed until they either come into contact with the boundary of the reactor, at which point they are absorbed, or they collide with the nucleus of an atom. When the latter occurs, either the neutron undergoes a scattering event where the particle bounces off the nucleus and continues its motion but with a new velocity, or a fission event occurs where the collision causes new neutrons to be produced with identical spatial positions, but potentially different velocities. 

In this setting, we have $E = D \times V$ where $D \subset \mathbb R^3$ is the spatial domain, which we assume to be open and bounded with smooth boundary, $\partial D$, and $V = \{\upsilon \in \mathbb R^3 : v_{min} < |\upsilon| < v_{max}\}$ is the velocity domain, where $0 < v_{min} \le v_{max} < \infty$. The NTE can then be stated as follows:

\begin{align}
\frac{\partial}{\partial t}\psi_t(r, \upsilon) &=\upsilon\cdot\nabla\psi_t(r, \upsilon)  -(\sigma_{\mathtt{s}}(r, \upsilon) + \sigma_{\mathtt{f}}(r, \upsilon))\psi_t(r, \upsilon)\notag\\
&+ \sigma_{\texttt{s}}(r, \upsilon)\int_{V}\psi_t(r, \upsilon') \pi_{\texttt{s}}(r, \upsilon, \upsilon'){\rm d}\upsilon' + \sigma_{\texttt{f}}(r, \upsilon) \int_{V}\psi_t(r, \upsilon') \pi_{\texttt{f}}(r, \upsilon, \upsilon'){\rm d}\upsilon',
\label{bNTE}
\end{align}
where
\begin{align*}
\sigma_{\texttt{s}}(r, \upsilon) &: \text{ the rate at which scattering occurs from incoming velocity $\upsilon$ at position $r$,}\\
\sigma_{\texttt{f}}(r, \upsilon) &: \text{  the rate at which fission occurs from incoming velocity $\upsilon$  at position $r$,}\\
\pi_{\texttt{s}}(r, \upsilon, \upsilon') &: \text{  probability density that an incoming velocity $\upsilon$ at position $r$ scatters to an } \\
 &\hspace{0.5cm}\text{outgoing velocity, with probability $\upsilon'$ satisfying }\textstyle{\int_V}\pi_{\texttt{s}}(r, \upsilon, \upsilon'){\rm d}\upsilon'=1,\text{ and }\\
 \pi_{\texttt{f}}(r, \upsilon, \upsilon') &:  \text{  density of expected neutron yield at velocity $\upsilon'$ from fission with }   \\
 &\hspace{0.5cm}\text{incoming velocity  $\upsilon$ satisfying } \textstyle{\int_V\pi_{\texttt{f}}}(r, \upsilon, \upsilon'){\rm d}
 \upsilon' <\infty.
\end{align*}

Moreover, we impose the following initial and boundary conditions
\begin{equation}
\left.
\begin{array}{ll}
\psi_0(r, \upsilon) = g(r, \upsilon) &\text{ for }r\in D, \upsilon\in{V},
\\
&
\\
\psi_t(r, \upsilon) = 0& \text{ for }r\in \partial D
\text{ if }\upsilon
\cdot{\bf n}_r>0,
\end{array}
\right\}
\label{BC}
\end{equation}
where  ${\bf n}_r$ is the outward facing normal of $D$ at $r\in \partial D$ and $g: D\times {V}\to [0,\infty)$ is a bounded, measurable function. 

\smallskip

Under the assumptions, 
\begin{itemize}
\item[(A1)] $\sigma_{\mathtt{s}}$, $\sigma_{\mathtt{f}}$, $\pi_{\mathtt{s}}$ and $\pi_{\mathtt{f}}$ are uniformly bounded away from infinity.
\item[(A2)] $\inf_{r \in D, \upsilon, \upsilon' \in V}(\sigma_{\texttt{s}}(r,\upsilon)\pi_{\texttt{s}}(r, \upsilon, \upsilon')+ \sigma_{\texttt{f}}(r,\upsilon)\pi_{\texttt{f}}(r, \upsilon, \upsilon')) > 0$,
\end{itemize}
it was shown in~\cite{CoxHarrisEtAl2019, HortonEtAl2018, HK} that the NTE can be modelled using a weighted PDMP, also known as the neutron random walk (NRW). More precisely, setting
\begin{align}
\alpha(r,\upsilon) &= \sigma_{\texttt{s}}(r,\upsilon)+ \sigma_{\texttt{f}}(r,\upsilon)\int_V\pi_{\texttt{f}}(r, \upsilon,\upsilon'){\rm d}\upsilon' \label{alpha}\\
\pi(r, \upsilon, \upsilon') &= (\alpha(r,\upsilon))^{-1}\left[\sigma_{\texttt{s}}(r,\upsilon)\pi_{\texttt{s}}(r, \upsilon, \upsilon')+ \sigma_{\texttt{f}}(r,\upsilon)\pi_{\texttt{f}}(r, \upsilon, \upsilon')\right],\label{pi}\\
\beta(r,\upsilon) &= \alpha(r,\upsilon) -  \sigma_{\texttt{s}}(r,\upsilon)-  \sigma_{\texttt{f}}(r,\upsilon), \label{beta}
\end{align}
we define the NRW $((R_t, \Upsilon_t)_{t \ge 0}, \mathbf{P}_{(r, \upsilon)})$ as follows. From an initial configuration $(r, \upsilon) \in D \times V$, the particle will propagate linearly until it is either absorbed at the boundary of the domain or, at rate $\alpha$ the process scatters off a nucleus and a new velocity is chosen according to $\pi$. In order words, the infinitesimal generator associated to $((R_t, \Upsilon_t)_{t \ge 0}, \mathbf{P})$ is given by
\begin{equation}\label{NRW-gen}
	Lf(r,\upsilon)=\mathtt{T}f+\alpha(r,\upsilon)\int_{V} (f(r,\upsilon')-f(r,\upsilon))\,\pi(r,\upsilon,\upsilon'),\D \upsilon,
\end{equation}
where $\mathtt{T}$ is defined for regular functions by $\mathtt{T}f(r,\upsilon)=\upsilon\cdot\nabla_r f(r,\upsilon)$ and, more generally and when the limit exists, by 
\[
	\mathtt{T}f(r,\upsilon)=\lim_{\varepsilon\to 0^+} \frac{f(r+\varepsilon\upsilon)-f(r,\upsilon)}{\varepsilon}.
\]
Then, for a bounded, measurable function $g$, and $(r, \upsilon)\in D \times V$
\[
\psi_t[g](r, \upsilon) \coloneqq \bE_{(r, \upsilon)}\left[{\rm e}^{\int_0^t\beta(R_s, \Upsilon_s){\rm d}s}g(R_t, \Upsilon_t)\mathbf{1}_{(t < \tau_D)}\right],
\]
solves~\eqref{bNTE}--\eqref{BC}, where $\tau_D$ is the first exit time of the NRW from $D$.

\smallskip

It was also shown in \cite{HK} that if (A1) and (A2) are satisfied, and that if there exists an $ \varepsilon>0$ such that
\begin{itemize}
\item[(B1)] $\textstyle{D_{\varepsilon} \coloneqq \{r \in D : \inf_{y\in \partial D}|r - y| > \varepsilon{\texttt v}_{\texttt{max}}\}}$ is non-empty and connected,
\item[(B2)] there exist $0 < s_{\varepsilon} < t_{\varepsilon}$ and ${\color{black}\gamma} > 0$ such that, for all $r \in D \backslash D_{\varepsilon}$, there exists $K_r \subset V$ measurable such that $\text{Vol}({K_r}) \ge {\color{black}\gamma} >0$ and for all $\upsilon \in K_r$, $r + \upsilon s \in D_{\varepsilon}$ for every $s \in [s_{\varepsilon}, t_{\varepsilon}]$ and $r+\upsilon s \notin \partial D$ for all $s\in [0, s_{\varepsilon}]$,
\end{itemize}
the semigroup $(\psi_t)_{t \ge 0}$ exhibits the following Perron Frobenius behaviour.

\bigskip
\begin{figure}[h!]
\includegraphics[scale = 0.5]{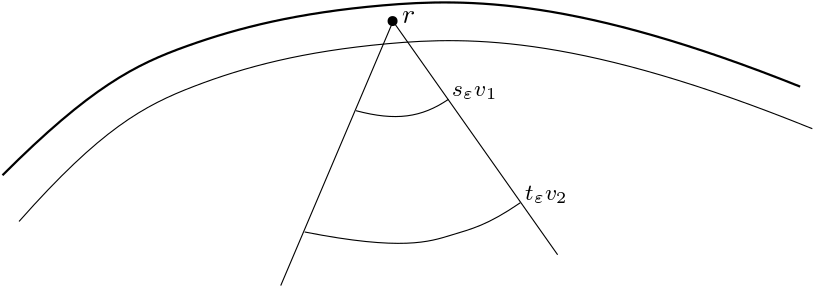}
\caption{Illustration of B2 in $\mathbb R^2$. Assumption B2 implies that the set $\{r + \upsilon s : \upsilon \in K_r, s \in [s_\varepsilon, t_\varepsilon]\} \subset D_\varepsilon$, which is illustrated by the area enclosed by $v_1s_\varepsilon$ and $v_2t_\varepsilon$, with $v_1 = \min\{|\upsilon| : \upsilon \in K_r\}$ and $v_2$ denoting the maximum.} 
\end{figure}

\begin{thm*}[\cite{HK}]
There exists  a $\lambda_*\in\mathbb{R}$, a positive right eigenfunction $\varphi \in L^+_\infty(D\times V)$ and a left eigenmeasure which is absolutely continuous with respect to Lebesgue measure on $D\times V$ with density $\tilde\varphi\in L^+_\infty(D\times V)$, both having associated eigenvalue ${\rm e}^{\lambda_* t}$, and such that $\varphi$  (resp. $\tilde\varphi$) is uniformly (resp. a.e. uniformly) bounded away from zero on each compactly embedded subset of $D\times V$. In particular, for all $g\in L^+_{\infty}(D\times V)$,
\begin{equation}
\langle\tilde\varphi, \psi_t[g]\rangle = {\rm e}^{\lambda_* t}\langle\tilde\varphi,g\rangle\quad  \text{(resp. } 
\psi_t[\varphi] = {\rm e}^{\lambda_* t}\varphi
\text{)}, \quad t\ge 0.
\label{leftandright}
\end{equation}
Moreover, there exist $C,\varepsilon>0$ such that 
\begin{equation}
\sup_{g\in L^+_\infty(D\times V): \Vert g\Vert_\infty\leq 1}  \left\|{\rm e}^{-\lambda_* t}{\varphi}^{-1}{\psi_t[g]}-\langle\tilde\varphi, g\rangle\right\|_\infty \leq   C {\rm e}^{-\varepsilon t}, \text{ for all $t \ge 0$ large enough}.
\label{spectralexpsgp}
\end{equation}
\end{thm*}

This theorem therefore provides a way to build Monte Carlo simulations to estimate the eigenvalue $\lambda_*$ and the eigenfunctions $\varphi$ and $\tilde\varphi$ via the NRW. We refer the reader to \cite{cox2020monte} for further details. However, although simulating a single weighted path has advantages over simulating an entire tree of neutrons, the transience of this process means that many of the particles exit the domain relatively quickly and therefore, a large number of simulations are required in order to obtain information about the system. This also leads to problems with the variance of the estimators. In order to deal with this problem, the notion of `$h$-transform' was developed in \cite{cox2020monte}, where one biases the NRW to prevent it from exiting the domain. This idea is similar to that presented in~\cite{Rousset2006}, where the author use $h$-transforms to build Monte-Carlo approximations. Therein, the main result requires that the resulting $h$-transformed process has bounded killing rate. In our case, although this can be theoretically obtained, the natural candidates for function $h$ result in a process with unbounded killing rate. We show below that the associated BBMMI is non-explosive, which ensures that Theorem~\ref{thm:main} applies.

\smallskip 

More precisely, let $h$ be a bounded positive function on $D \times V$ such that $h = 0$ on $\partial D$ and define the following change of measure,
\begin{equation}
\left.\frac{{\rm d} \mathbf{P}^h_{(r,\upsilon)}}{{\rm d} \mathbf{P}_{(r,\upsilon)}}\right|_{\sigma((R_s,\Upsilon_s), s\leq t)}:  = 
\exp\left(-\int_0^t  \frac{{\emph{\texttt J}}h (R_s,\Upsilon_s)}{h (R_s,\Upsilon_s)}{\rm d} s\right)\prod_{i=1}^{N_t}\frac{h(R_{T_{i}}, \Upsilon_{T_{i}})}{h(R_{T_i}, \Upsilon_{T_{i-1}})} \, \mathbf{1}_{(t<\tau^{D})},
\label{hCOM}
\end{equation}
where 
\begin{equation}
{\emph{\texttt J}}g(r, \upsilon) = \alpha(r, \upsilon)\int_V[g(r, \upsilon') - g(r, \upsilon)]\pi(r, \upsilon, \upsilon'){\rm d}\upsilon',
\label{NRWscatter}
\end{equation}
$(T_i)_{i \ge 0}$ are the scatter times of the NRW with $T_0 = 0$, and $N_t = \sup\{i : T_i \le t\}$.

\smallskip

Then, $((R, \Upsilon), \mathbf{P}^h)$ also defines a NRW but with scattering operator 
\begin{equation}
{\emph{\texttt J}}_h g(r, \upsilon) = \alpha(r, \upsilon)\int_V[g(r, \upsilon') - g(r, \upsilon)]\frac{h(r, \upsilon')}{h(r, \upsilon)}\pi(r, \upsilon, \upsilon'){\rm d}\upsilon'.
\label{hNRWscatter}
\end{equation}

The idea is that the factor $h(r, \upsilon')/h(r, \upsilon)$ in the above integral forces the NRW to scatter when it is approaching the boundary, thus preventing it from being killed. Thus, if we can show that this process satisfies Assumption 1, we may use the BBMMI process in conjunction with the $h$-transformed process to estimate the leading eigentriple $(\lambda_*, \varphi, \tilde\varphi)$, for example. Indeed, it is straightforward to show that the leading eigentriple of the $h$-transformed process is given by $(\lambda_*, \varphi/h, h\tilde\varphi)$ and so we may use the BBMMI process to estimate this triple, and hence obtain an estimate for the original quantities $(\lambda_*, \varphi, \tilde\varphi)$. The rest of this section is dedicated to verifying Assumption 1 for the transformed process. We refer the reader to Section \ref{sec:numerics} for the numerical aspects. 

To this end, fix $\delta=1/(2\|\alpha\|_\infty)$ and let $\phi:[0,+\infty)\to[0,\delta]$ be a regular non-decreasing function such that $\phi(x)=x$ for all $x\in[0,\delta/2]$ and $\phi(x)=\delta$ for all $x\geq \delta$. We set
\[
h(r,\upsilon)=\phi\left(\kappa_{r,\upsilon}^D\right),
\]
where $\kappa_{r, \upsilon}^D = \inf\{t > 0 : r + \upsilon t \notin D\}$.
Using the following result, one can express expectations with respect to $\mathbf P$ as Feynman-Kac formulas related to $\mathbf P^h$. In what follows, we say that $(r,\upsilon)$ is in a vicinity of $\partial D$ if $h(r,\upsilon)< \delta/2$. 

\begin{lem}
	Suppose (A1), (A2) hold. A process with law $\mathbf P^h_{(r,v)}$ does not hit the boundary with probability one.
        Moreover, setting $b=\left(\frac{Lh}{h}\right)_+$ and $\kappa=\left(\frac{Lh}{h}\right)_-$, where $L$ was defined in \eqref{NRW-gen}, the function $b$ is bounded on $D\times V$ and $\kappa$ is locally bounded. Let $\mathbf{P}^\kappa_{r,\upsilon}$ be the law of a process with law $\mathbf{P}^h_{r,\upsilon}$ with additional soft killing at rate $\kappa$, then we have
	\begin{align}
	\mathbf{E}_{(r,\upsilon)}\left[f(R_t,\Upsilon_t)\, \mathbf{1}_{(t<\tau^{D})}\right]=h(r,\upsilon)\, \mathbf{E}^\kappa_{(r,\upsilon)}
	\left[\exp\left(\int_0^t  b(R_s,\Upsilon_s)\,{\rm d} s\right)\frac{f(R_t,\Upsilon_t)}{h(R_{t}, \Upsilon_{t})} \mathbf{1}_{(t<\tau_\kappa)}\right],
	\label{hCOM4bis}
	\end{align}
	where $\tau_\kappa$ denotes the soft killing time of the process with law $\mathbf{P}^\kappa_{r,\upsilon}$. 
\end{lem}

\begin{proof}
	We have
\begin{align}
&\psi_t[g](r,\upsilon)\notag\\
&= \mathbf{E}^h_{(r,\upsilon)}\left[
\exp\left(\int_0^t  \frac{{\emph{\texttt J}}h (R_s,\Upsilon_s)}{h (R_s,\Upsilon_s)} +\beta(R_s,\Upsilon_s){\rm d} s\right)\prod_{i=1}^{N_t}\frac{h(R_{T_i}, \Upsilon_{T_{i-1}})}{h(R_{T_{i}}, \Upsilon_{T_{i}})}
g(R_t, \Upsilon_t)
\mathbf{1}_{(t<\tau^{D})}\right]
\label{hsemigp}
\end{align}
Note that
\begin{equation}
\inf_{r\in D,\upsilon\in V}  \lim_{s\to\kappa_{r,\upsilon}^D}\frac{|\upsilon| (\kappa_{r,\upsilon}^D-s)}{h (r+\upsilon s,\upsilon)} >0,
\label{inf}
\end{equation}
where
$\notag
\kappa_{r,\upsilon}^{D} := \inf\{t>0 : r+\upsilon t\not\in D\}$. Hence, Theorem 7.1 in \cite{cox2020monte} implies that, under $\mathbf{P}^h_{(r,\upsilon)}$, $r\in D, \upsilon\in V$, $(R,\Upsilon)$ does not hit the boundary with probability one.
In particular, the conditions of Remark~7.1 of \cite{cox2020monte} are satisfied and hence
\begin{equation}
\left.\frac{{\rm d} \mathbf{P}^h_{(r,\upsilon)}}{{\rm d} \mathbf{P}_{(r,\upsilon)}}\right|_{\sigma((R_s,\Upsilon_s), s\leq t)}= 
\exp\left(-\int_0^t  \frac{Lh (R_s,\Upsilon_s)}{h (R_s,\Upsilon_s)}{\rm d} s\right)\frac{h(R_{t}, \Upsilon_{t})}{h(r,\upsilon)} \, \mathbf{1}_{(t<\tau^{D})},
\label{hCOM2}
\end{equation}
where we recall that $L$ was defined in \eqref{NRW-gen}.
Finally, we obtain, for all measurable functions $f:D\times V\to[0,+\infty)$,
\begin{align}
\mathbf{E}_{(r,\upsilon)}\left[f(R_t,\Upsilon_t)\, \mathbf{1}_{(t<\tau^{D})}\right]=h(r,\upsilon)\, \mathbf{E}^h_{(r,\upsilon)}
\left[\exp\left(\int_0^t  \frac{Lh (R_s,\Upsilon_s)}{h (R_s,\Upsilon_s)}{\rm d} s\right)\frac{f(R_t,\Upsilon_t)}{h(R_{t}, \Upsilon_{t})} \right]
\label{hCOM3}
\end{align}
where have used the fact that $\mathbf{1}_{(t<\tau^{D})}=1$, $\mathbf{P}^h_{(r,\upsilon)}$-almost surely, deduced from~\eqref{inf}. Now, introducing $b=\left(\frac{Lh}{h}\right)_+$ and $\kappa=\left(\frac{Lh}{h}\right)_-$, and considering the law $\mathbf{P}^\kappa_{r,\upsilon}$ of a process with law $\mathbf{P}^h_{r,\upsilon}$ with soft killing at rate $\kappa$, we deduce that
\begin{align}
\mathbf{E}_{(r,\upsilon)}\left[f(R_t,\Upsilon_t)\, \mathbf{1}_{(t<\tau^{D})}\right]=h(r,\upsilon)\, \mathbf{E}^\kappa_{(r,\upsilon)}
\left[\exp\left(\int_0^t  b(R_s,\Upsilon_s)\,{\rm d} s\right)\frac{f(R_t,\Upsilon_t)}{h(R_{t}, \Upsilon_{t})} \mathbf{1}_{(t<\tau_\kappa)}\right],
\label{hCOM4}
\end{align}
where $\tau_\kappa$ denotes the soft killing time of the process (with law $\mathbf{P}^h_{r,\upsilon}$ with soft killing at rate $\kappa$).

Now, we observe that, in a vicinity of $\partial D$,
\begin{align}
Lh(r,\upsilon)&=-1+\alpha(r,\upsilon)\int_{V} (h(r,\upsilon')-h(r,\upsilon))\,\pi(r,\upsilon,\upsilon')\,\D \upsilon' \label{eq:Lh}\\
&\leq -1+\|\alpha\|_\infty \delta=-1/2 \nonumber
\end{align}
and hence that
\[
\sup_{(r,\upsilon)\in D\times V} \frac{Lh(r,\upsilon)}{h(r,\upsilon)}<+\infty.
\]
This implies that $b$ is bounded. The fact that $\kappa$ is locally bounded is an immediate consequence of the regularity of $h$.
\end{proof}

Our aim is now to apply the results of Section~\ref{sec:mainresult} to the Feynman-Kac expression~\eqref{hCOM4bis} in order to obtain information about $\mathbf P$. Assumption~1 is clearly satisfied by a process with law $\mathbf P^\kappa_{(r,\upsilon)}$, and it only remains to check that Assumption~2 holds true in order to apply Theorem~\ref{thm:main}. This is the purpose of the following result. In what follows, we refer to a bounded BBMMI as any BBMMI whose selection mechanism imposed by the parameters $p^i$ results in a size constrained process, i.e. there is a constant $N_{max}$ such that $\bar N_t \le N_{max}$ for all $t \ge 0$.

\begin{prop}
    \label{prop:nonexplosion}
    Under the assumptions (A1), (A2), (B1) and (B2), any {\color{black} bounded} BBMMI driven by $\mathbf P^\kappa_{(r,\upsilon)}$ satisfies Assumption~2. 
\end{prop}

\begin{proof}
We denote by $\sigma_1<\sigma_2<...$ the sequence of times at which events occur for a given particle in the system (each time is either a  scattering, a branching or a soft killing). Note that each time $\sigma_n$ is well defined since,  almost surely, the number of event occurring before a time $T$ goes to infinity when $T\to+\infty$. Since one can use the strong Markov property at the birth time of the given particle, we assume without loss of generality that the particle is already alive at time $0$.

From the expression of the scattering operator ${\emph{\texttt J}}_h$ in~\eqref{hNRWscatter}, one observes that the scattering rate toward a direction in a set $V'\subset V$ is given by $\int_{V'}\alpha(r,\upsilon)\frac{h(r,\upsilon')}{h(r,\upsilon)}\pi(r,\upsilon,\upsilon')\mathrm d\upsilon'$. But, under the assumptions (B1--2), there exists $\varepsilon>0$, $\beta>0$ and $\eta>0$ such that, for any point $r\in D$, there exists $V_r$ with $r+[0,\beta] V_r\subset D$ and $\text{Leb}_d((r+[0,\beta] V_r)\cap \{r'\in D,\ d(r',\partial D)\geq \varepsilon\})>\eta$, where $\text{Leb}_d$ denotes the $d$-dimensional Lebesgue measure. In particular, using the fact that $h(r,\upsilon')$ is lower bounded on $V_r$ (uniformly over $r\in D$). We deduce that there exists a constant $\underline s>0$ such that the scattering rate toward a direction in $ V_r$ is lower bounded by 
$\underline s /h(r,\upsilon)$, uniformly in $(r,\upsilon)\in D\times V$.

Using \eqref{eq:Lh} to bound $Lh$ from below, there exists a constant $\bar c$ such that the total rate at which a particle undergoes an event (either a  scattering, a branching or a (soft) killing) is upper bounded by $\bar c /h(r,\upsilon)$, uniformly in $(r,\upsilon)\in D\times V$. Hence each event has a probability greater than $\underline s/\bar c$ to be a scattering event toward a direction in $V_r$, independently of the past of the process, and in particular independently of the time event $\sigma_n$: formally, for all $n\geq 0$,
\[
\mathbf P^\kappa\left(\sigma_n\text{ is a scattering event and }(R_{\sigma_n},\Upsilon_{\sigma_n})\in \{R_{\sigma_n-}\}\times V_{R_{\sigma_n-}}\mid \sigma_n, R_{\sigma_n-},\Upsilon_{\sigma_n-}\right)\geq \underline s\,/\bar c>0,
\]
where $(R,\Upsilon)$ denotes the position and direction of the given particle. 

Now, using the fact that the killing rate and the scattering rate are uniformly bounded in $\{r'\in D,\ d(r',\partial D)\geq \varepsilon\}$ and the fact that the total branching rate of the system is uniformly bounded,  we deduce that there exists a constant $\underline p>0$ and a time $T>0$ such that, for all $r\in D$ and $\upsilon\in V_r$, $\PP(\sigma_1 \geq  T)\geq \underline{p}$. 

Hence, we deduce from the strong Markov property that
\begin{align*}
\mathbf P^\kappa(\sigma_{n+1}<T) &\leq  \mathbf E^\kappa\left[\ind_{\sigma_n<T}\, \mathbf P^\kappa_{X_{\sigma_n}}(\sigma_1<T)\right]\\
                          &\leq  \mathbf E^\kappa\left[\ind_{\sigma_n<T}\,\left(1-\ind_{\Upsilon_{\sigma_n}\in V_{R_{\sigma_n}}}\mathbf P^\kappa_{X_{\sigma_n}}(\sigma_1<T)\right)\right]\\
                          &\leq \mathbf E^\kappa\left[\ind_{\sigma_n<T}\,\left(1-\ind_{\Upsilon_{\sigma_n}\in V_{R_{\sigma_n}}}\underline p\right)\right]\\
                          &\leq \mathbf P^\kappa(\sigma_n<T)(1-\underline s\,/\bar c\,\underline p).
\end{align*}
This shows that the number of events occurring before time $T$ is stochastically dominated by a geometric random variable with parameter $\underline s\,/\bar c\,\underline p>0$, so that it is finite almost surely and that its expectation is bounded by a constant $\bar a$. Since this is true for all particles, we deduce that total number of events has an expectation bounded by $N_{max}\bar a$. Finally, observing that this bound does not depend on the initial position, we deduce using the Markov property at times $T,2T,\ldots$ that the number of event occurring before a time horizon $nT$ is bounded by $kN_{max}\bar a$, and hence that it is finite almost surely.

This concludes the proof of Proposition~\ref{prop:nonexplosion}.
\end{proof}

\subsection{Numerical properties of the constrained branching process}\label{sec:numerics}

The aim of this section is to discuss some of the numerical properties of the $N_{min}$--$N_{max}$ process defined in Remark~\ref{rem-NminNmax} of Section~\ref{sec:descr}, which we restate here for convenience. For two integers $N_{min}\neq 1$ and $N_{max}$, with $0\leq N_{min}\leq N_{max}$, we define the $N_{min}$--$N_{max}$ process as the particle system defined in Section~\ref{sec:descr} with the parameter choices
\[
  p^{i_0}(x_i,\,i\in s)=\mathbf{1}_{\{|s|=N_{min}\}} \text{  and  } q^{i_0}(x_i,\,i\in s)=\mathbf{1}_{\{|s|=N_{max}\}}, \quad i_0 \in s. 
\]  
In Section~\ref{sec:compFVNminNmax1}, we first make comparisons between this algorithm and the fixed size Moran type algorithm studied in \cite{DelMoral2004,DelMoral2013}. 
We will often use the abbreviation FV IPS to refer to the latter, in reference to the name {\it Fleming Viot interacting particle system}, which has been used in several instances in the literature.
In Section~\ref{sec:filtering}, we introduce a filtering method to deal with numerical discrepancies occurring in situations where $\Pi_T^A\Pi^B_T$ has high variance.

\smallskip

\subsubsection{Comparison with the FV IPS}
\label{sec:compFVNminNmax1}

We recall that the FV IPS simulates a set of $N$ sub-Markov processes, until the first time one of the particles dies, at which point a resampling event occurs by uniformly selecting one of the remaining $N-1$ particles and duplicating it.  Note that the FV IPS is a particular instance of the $N_{min}$--$N_{max}$ IPS, with a null branching rate and $N_{min}=N_{max}=N$. In particular, Theorem~\ref{thm:main}  applies to the FV IPS (see~\cite{Villemonais2014} for an analogous result in the FV case).

\smallskip

In some situations, it is possible to use the FV IPS to approximate quantities associated with branching processes. Indeed, it is well known that for a branching process $X = (X_t)_{t \ge 0}$, under fairly weak assumptions, the following many-to-one formula holds
\[
\psi_t[f](x) \coloneqq \mathbb{E}_{\delta_x}[\langle f, X_t\rangle] = \mathbf{E}_x\left[{\rm e}^{\int_0^t \beta(Y_s) {\rm d} s}f(Y_t) \right],
\]
where $\beta$ is determined by the branching rate and mean offspring of $X$, and the law of $Y$ is uniform amongst all paths of $X$. 
%
%
%
%
%
%
%
%
In the case where $\beta$ is uniformly bounded above, one may compensate the right-hand side above by its maximum, $\bar\beta$, to introduce a penalised process: $\psi_t^\mathtt{k}[f](x) \coloneqq {\rm e}^{-\bar\beta t}\psi_t[f](x)$. The penalisation has the effect of killing $Y$ at an additional rate $\bar\beta - \beta$, thus yielding a sub-Markov process with semigroup $\psi_t^\mathtt{k}$. This then allows one to use the FV IPS to obtain estimates of this semigroup and, in some situations, of its associated quasi-stationary distribution, for example. In general, the use of the FV IPS for the approximation of quantities related to some non-conservative semigroups is widespread, see for instance~\cite{BurdzyHolystEtAl1996,DelMoral2004,Rousset2006,GrigorescuKang2004,Villemonais2011,GroismanJonckheere2013,CerouEtAl2019,CerouDelyonEtAl2020,CerouGuyaderEtAl2020}. 

In this section, we show on a simple example that the $N_{min}$--$N_{max}$ algorithm allows one to overcome some of the limitations of the FV IPS. For simplicity, we only consider the situation where $N_{min}=N_{max}$. In this situation, the $N_{min}$--$N_{max}$ algorithm is the classical Moran interacting particle system, and identical strategies have already been discussed in the literature (see e.g.~\cite{CloezCorujo2021,DelMoral2004,Rousset2006}). Our aim is to illustrate some numerical properties of these strategies for a toy model.

\paragraph{A simple branching birth and death process.}
We consider a birth and death process on $E=\{1, \dots, M\}$ for some parameter $M \in \mathbb N$. For a particle occupying state $x$, one of the following things may occur:
\begin{itemize}
\item the particle jumps with rate $x^2$, to state $1\vee(x - 1)$ with probability $\frac{x}{x+1}$  and to state  $(x + 1)\wedge M$ with probability $\frac1{x+1}$,
\item at rate $b(x) = x$, a new particle is produced at the same site, which will continue independently the same behaviour as the original particle. 
\end{itemize}

In this situation where the state space is finite and irreducible, it is well known that the semigroup $(\psi_t)_{t\geq 0}$ admits a limiting distribution up to an exponential rescaling (see for instance~\cite{DarrochSeneta1967}): there exist a constant $\lambda_M\in \mathbb R$, a positive function  $g_M:\{1,\ldots,M\}\to (0,+\infty)$ and a probability measure $\nu_M$ on $\{1,\ldots,M\}$ such that
\begin{align*}
\sup_{f,x}\left|e^{-\lambda_M t} \psi_t[f](x)-g_M(x)\nu_M(f)\right|\leq C_Me^{-\gamma_M t}, \qquad t \ge 0,
\end{align*}
for some positive constants $C_M,\gamma_M$, where the supremum is taken over all functions $f:\{1,\ldots,M\} \to (0, \infty),\ \|f\|_\infty\leq 1$ and all $x\in \{1,\ldots,M\}$.

The associated birth and death process with killing evolves on $\{1,\ldots,M\}$ with the same jump rates, without branching, and with the additional killing rate $\kappa(x)=M-x$. Denoting as above by $\psi^\mathtt{k}$ the associated sub-Markov semigroup, we have $\psi^\mathtt{k}_t=e^{-Mt}\psi_t$ and hence
\begin{align*}
\sup_{f,x}\left|e^{-(\lambda_M-M) t} \psi^\mathtt{k}_t[f](x)-g_M(x)\nu_M(f)\right|\leq C_Me^{-\gamma_M t},
\end{align*}
 where the supremum is taken over all functions $f:\{1,\ldots,M\} \to (0, \infty),\ \|f\|_\infty\leq 1$ and all $x\in \{1,\ldots,M\}$.

\paragraph{Theoretical convergence of the $N_{min}$--$N_{max}$ and of the FV IPSs.}
In what follows, $m^M_T$ refers to the empirical measure of the $N_{min}$--$N_{max}$ particle system at time $T$, associated to the above branching birth and death process, and $\mu^M_T$ refers to the empirical measure of the FV IPS with $N:=N_{min}=N_{max}$ particles at time $T$ associated to the above killed birth and death process.

In particular, it is well known (se e.g. Theorem~3.27 in~\cite{del2000branching}) that, in the situation where $N_{min}=N_{max}=N$,
 \[
\left\|\frac{m^M_T(f)}{N}-\nu_M(f)\right\|_2\leq \frac{C_{FV} \|f\|_\infty}{N^{\alpha/2}}+C_X\,e^{-\gamma_X T}\,\|f\|_\infty.
\]   
for some $\alpha>0$ and $C_{FV}>0$. Using the fact that the $N_{min}$--$N_{max}$ process is ergodic, we deduce that its normalised empirical stationary distribution $\mathcal X^M_{N}$ (the normalised empirical distribution of a random vector of particles distributed according to the stationary distribution of the $N_{min}$--$N_{max}$ process) satisfies
\[
\mathcal X^M_{N}(f)\xrightarrow[N\to+\infty]{L_2}\nu_M(f),
\]
for all function $f:\{1,\ldots,M\}\to \nu_M$.
Similarly, the normalised empirical stationary distribution $\mathcal Y^M_N$ of the FV particle system satisfies
\[
\mathcal Y^M_{N}(f)\xrightarrow[N\to+\infty]{L_2}\nu_M(f).
\]
These theoretical convergence results ensure that both algorithms can be used to approximate the limiting distribution $\nu_M$. It remains to compare their performances via numerical simulations.

\paragraph{Numerical comparison of the $N_{min}$--$N_{max}$ and the FV IPSs.}
We now consider the problem of numerically approximating the mean of $\nu_M$, that is $\nu_M(f_M)$ with $f_M(x)=x$ for all $x\in \{1,\ldots,M\}$. The estimator is $\hat\theta_N=\mathcal X^M_{N}(f_M)$  (resp.  $\hat \theta_N=\mathcal Y^M_{N}(f_M)$) for the $N_{min}$--$N_{max}$ (resp. FV) IPS, and we compare three important metrics determining the performance of the algorithms:
\begin{itemize}
	\item the bias of the estimator, $|\mathbb E(\hat\theta_N)-\nu_M(f_M)|$,
	\item the standard deviation of the estimator, $\text{Std}(\hat\theta_N)$,
	\item the number of interaction events, $(A_T+B_T)/T$, which corresponds to the number $B_T$ of selection events in the $N_{min}$--$N_{max}$ algorithm and the number $A_T$ of resampling events in FV IPS, per unit of time.
\end{itemize}
The last metric $(A_T+B_T)/T$ is particularly important since it entails that the total number of interaction events grows (at least) linearly in time. 

For $N=10$  and then $N=100$, we let the upper bound  of the state space $M$ vary in $\{10,100,1000,+\infty\}$, and present the results in Table~\ref{tab:N10} for $N=10$, and in Table~\ref{tab:N100} for $N=100$. For each metric, a lower number indicates that the algorithm performs better.

\begin{table}
\begin{tabular}{ |c|c|C{3.5cm}|C{3.5cm}|C{3.5cm}| } 
 \hline
 \multicolumn{2}{|c|}{$N=10$\phantom{$\Big)$}}   & $|\mathbb E(\hat\theta_N)-\nu_M(f_M)|$ & $\text{Std}(\hat\theta_N)$ & $(A_T+B_T)/T$\\ 
  \hline
 \multirow{2}{*}{$M=10$}& $N_{min}$--$N_{max}$ &\phantom{$\Big)$} $0.08$ & $0.30$ & $14.0$ \\
\cline{2-5}
                        &FV  & \phantom{$\Big)$} $0.10$ & $0.41$ & $87.2$ \\
 \hline
 \multirow{2}{*}{$M=100$}& $N_{min}$--$N_{max}$ & \phantom{$\Big)$} $0.08$ & $0.30$ & $14.0$ \\
\cline{2-5}
&FV  & \phantom{$\Big)$} $0.20$ & $0.51$ & $988$ \\
 \hline
 \multirow{2}{*}{$M=1000$}& $N_{min}$--$N_{max}$ & \phantom{$\Big)$} $0.08$ & $0.30$ & $14.0$ \\
\cline{2-5}
&FV  & \phantom{$\Big)$} $0.22$ & $0.53$ & $9989$ \\
 \hline
 \multirow{2}{*}{$M=+\infty$}& $N_{min}$--$N_{max}$ & \phantom{$\Big)$} $0.08$ & $0.30$ & $14.0$  \\
\cline{2-5}
&FV  & \phantom{$\Big)$} $*$ & $*$ & $*$\\
 \hline 
\end{tabular}
\caption{\label{tab:N10} For each value of $M\in\{10,100,1000,+\infty\}$, we display the bias of the estimator $\hat \theta_N$, its standard deviation, and the number of events,  for the $N_{min}$--$N_{max}$ algorithm (for which $\hat\theta_N=\mathcal X^M_{N}(f_M)$) and for the FV IPS (for which $\hat\theta_N=\mathcal Y^M_{N}(f_M)$) with $N=N_{min}=N_{max}=10$ particles. Note that the FV IPS is not defined when $M=+\infty$.}
\end{table}

\begin{table}[H]
	\begin{tabular}{ |c|c|C{3.5cm}|C{3.5cm}|C{3.5cm}| } 
		\hline
		\multicolumn{2}{|c|}{$N=100$\phantom{$\Big)$}}   & $|\mathbb E(\hat\theta_N)-\nu_M(f_M)|$ & $\text{Std}(\hat\theta_N)$ & $(A_T+B_T)/T$\\ 
		\hline
		\multirow{2}{*}{$M=10$}& $N_{min}$--$N_{max}$ &\phantom{$\Big)$} $0.01$ & $0.12$ & $144$ \\
		\cline{2-5}
		&FV  & \phantom{$\Big)$} $0.02$ & $0.18$ & $857$ \\
		\hline
		\multirow{2}{*}{$M=100$}& $N_{min}$--$N_{max}$ & \phantom{$\Big)$} $0.01$ & $0.12$ & $144$ \\
		\cline{2-5}
		&FV  & \phantom{$\Big)$} $0.10$ & $0.39$ & $9866$ \\
		\hline
		\multirow{2}{*}{$M=1000$}& $N_{min}$--$N_{max}$ & \phantom{$\Big)$} $0.01$ & $0.12$ & $144$ \\
		\cline{2-5}
		&FV  & \phantom{$\Big)$} $0.20$ & $0.50$ & $99873$ \\
		\hline
		\multirow{2}{*}{$M=+\infty$}& $N_{min}$--$N_{max}$ & \phantom{$\Big)$} $0.01$ & $0.12$ & $144$  \\
		\cline{2-5}
		&FV  & \phantom{$\Big)$} $*$ & $*$ & $*$\\
		\hline 
	\end{tabular}
\caption{\label{tab:N100} For each value of $M\in\{10,100,1000,+\infty\}$, we display the bias of the estimator $\hat \theta_N$, its standard deviation, and the number of events,  for the $N_{min}$--$N_{max}$ algorithm (for which $\hat\theta_N=\mathcal X^M_{N}(f_M)$) and for the FV IPS (for which $\hat\theta_N=\mathcal Y^M_{N}(f_M)$) with $N=N_{min}=N_{max}=100$ particles. Note that the FV IPS is not defined when $M=+\infty$.}
\end{table}

\bigskip

We can see that, for this particular example, the $N_{min}$--$N_{max}$ algorithm performs better in all situations, with respect to all metrics. Moreover, the bias of the FV approximation increases when $M$ increases, while the variance of the estimator also increases and the computation cost increases. Without surprise, increasing $M$ does not change the dynamic of the $N_{min}$--$N_{max}$ algorithm, since very few particles, if any, reach the boundary $M$ even when $M=10$.  Further, this suggests that it should be possible to extend the $N_{min}$--$N_{max}$ algorithm to the case of unbounded branching rates, and this is supported by the numerical simulations when $M=+\infty$ (for which the FV IPS is of course not defined), although it is not covered by our main result. We intend to address these questions in future work.

\begin{rem} \label{rem:timeInhom}
We note that the birth-death process studied above is a typical example of the type of setting where the $N_{min}$--$N_{max}$ process may outperform the FV IPS. Here, we briefly outline another example where this is also the case.

\smallskip

We consider a branching random walk on $\{0, \dots, n\}$ for some $n \ge 1$ with a time inhomogeneous branching rate. More precisely, let $(X_t)_{t \ge 0}$ denote a continuous time random walk on $\mathbb Z$ that moves one step to the right with probability $p \in (0, 1)$ and one step to the left otherwise. Each step is taken after an (independent) exponentially distributed time with parameter $1$.

\smallskip

Now let $b, k: \{0, \dots, n\} \to [0, \infty)$ be bounded measurable functions and $\mathcal{E}_b$, $\mathcal{E}_k$ be exponential random variables, that are independent of each other and $X$, and define
\[
  \tau_1^b = \inf\{t > 0 : \int_0^t b(X_s){\rm d} s > \mathcal{E}_b\},
\]
and
\[
  \tau_1^k = \inf\{t > 0 : \int_0^t k(X_s){\rm d} s > \mathcal{E}_k\}.
\]
Further, set $\tau_1 = \tau_1^b \wedge \tau_1^k$. If $\tau_1 = \tau_1^b$, a new particle is added to the system at position $X_{\tau_1^b}$, which continues the same behaviour as the original particle. If $\tau_1 = \tau_1^k$, the particle is removed from the system. We can then define $\tau_2$ to be the first time after $\tau_1$ that one of the particles alive at time $\tau_1$ branches or is killed, and this process continues iteratively. 

\smallskip

Letting $N_t$ denote the number of particles alive at time $t \ge 0$, and $\{X_t^i \,: \, i = 1, \dots, N_t\}$ their positions, the branching random walk is defined as
\[
  Z_t = \sum_{i = 1}^{N_t}\delta_{X_t^i}, \qquad t \ge 0.
\]

\smallskip

We now consider a time-inhomogeneous version of the above process. We consider a sequence of switching times $0 = T_0^{\texttt{off}}< T_1^{\texttt{on}} < T_1^{\texttt{off}} < T_2^{\texttt{on}} < T_2^{\texttt{off}} < \dots$, defined via
\[
  T_i^{\texttt{on}}  = T_{i - 1}^{\texttt{off}} + \mathbf{e}_{i - 1}(s_{\texttt{on}}), \qquad i \ge 1,
\]
and 
\[
  T_i^{\texttt{off}}  = T_{i}^{\texttt{on}} + \mathbf{e}_{i}(s_{\texttt{off}}), \qquad i \ge 1,
\]
where $\mathbf{e}_{i}(s_{\texttt{off}})$ and $\mathbf{e}_{i}(s_{\texttt{on}})$ are independent (of each other and everything else) exponential random variables with rate $s_{\texttt{off}}$ and $s_{\texttt{on}}$, respectively. For $T_{i-1}^{\texttt{off}} \le t < T_{i}^{\texttt{on}}$, $Z_t$ evolves as above but with $b = 0$, i.e. each particle moves as a random walk until it is (possibly) killed, until time $T_{i}^{\texttt{on}}$. At time $T_{i}^{\texttt{on}}$, we sample from an exponential random variable with rate $B \in (0, \infty)$ and set $b$ to be this value. For $T_{i}^{\texttt{on}} \le t < T_{i}^{\texttt{off}}$, $Z_t$ then evolves as the branching random walk with branching rate $b$ for each particle. At time $T_{i}^{\texttt{off}}$, we again set the branching rate to be $0$, and continue this process until some terminal time $T$. 

\smallskip

We can see that in this case, while the branching rate is bounded almost surely on any compact time interval, we may simulate the $N_{min}$--$N_{max}$ process.
However, we cannot use the fixed size Moran type  algorithm since we do not know the upper bound for the branching rate. 

\smallskip

\end{rem}

\subsubsection{Two-level Monte Carlo methods for estimating Lyapunov coefficients}
\label{sec:filtering}

In many applications, it can be of interest to estimate numerically the growth rate of the semigroup $Q$ for large times. A natural conjecture is that for large times $T$, $m_0Q_T f \approx \langle m_0, \psi\rangle \langle \varphi, f\rangle \mathrm{e}^{\lambda T}$, where $\varphi$ and $\psi$ are the left eigenmeasure and right-eigenfunction respectively of the operator $Q$, and $\lambda$ is the \textit{Lyapunov exponent} for $Q$. One can try to estimate $\lambda$ by computing the quantity $m_0Q_Tf$ for large times $T$.

By Theorem~\ref{thm:main}, one method for estimating $\lambda$ would be through equation \eqref{M21}, so that we can simulate the $(m_T,\Pi_T^A, \Pi_T^B)$-system and estimate $\lambda$ by:
\begin{equation}\label{eq:1}
\hat{\lambda} = \frac{1}{T}\log\left( \Pi_T^A \Pi_T^B m_T(\mathbf{1})\right) = \frac{1}{T}\log\left( \Pi_T^A \Pi_T^B \right),
\end{equation}
for $T$ large. In practice, however, this is numerically a poor choice of estimator, due to large-deviation effects. The presence of the `exponential' functionals, $\Pi_T^A, \Pi_T^B$ mean that significant contributions to the expectation will be made by paths of the $N_{min}$--$N_{max}$ process which have probability $~e^{-\delta T}$, and so infeasibly many simulations will be needed to compute an unbiased estimate for $\hat{\lambda}$. A simple solution is to replace the estimator by
\begin{equation} \label{eq:2}
\bar{\lambda} = \frac{1}{\delta t}\log\left[\frac{1}{n} \sum_{i < n} \left( \frac{\Pi_{t_{i+1}}^A \Pi_{t_{i+1}}^B}{\Pi_{t_{i}}^A \Pi_{t_{i}}^B} \right)\right],
\end{equation}
where $t_i = T \frac{i}{N}$ and $\delta t = \frac{T}{N}$. In many examples, under conditions which ensure that the process is ergodic and thus converges to a stationary distribution, as $T \to \infty$, for fixed $\delta t$, the estimator $\bar{\lambda}$ can be shown to converge to $\lambda$ (e.g. \cite{del_moral_particle_2003}).

In some cases the expression in \eqref{eq:2} may not be valid and we can only use estimates of the form in \eqref{eq:1}. Examples may be when the process is not ergodic, or the system is not time-homogenous, such as the example given in Remark~\ref{rem:timeInhom}.  In such circumstances, we propose a novel Monte Carlo method to compute estimates of $\lambda$. Specifically, we note that the $N_{min}$--$N_{max}$ process  $\mathcal{X}_t := \{ X_t^1, \ldots, X_t^{N_t}\}$ defines a continuous time Markov process, while the processes $A_t$ and $B_t$ corresponding to the number of selection and resampling events, are themselves additive functionals of the Markov process. Estimation of $\lambda$ corresponds to computing a Feynman-Kac type exponential functional for a Markov process, and it is well known that estimating such functionals can be performed efficiently using interacting particle simulations, \cite{del2000moran, del2000branching, DelMoral2013, DelMoral2004}.

To try to manage this error, we propose the following algorithm. Informally, we simulate multiple independent copies of the $N_{min}$--$N_{max}$ process, in continuous time, using the algorithm described in Section~\ref{sec:descr}. At fixed time intervals, we perfom a resampling of the independent  $N_{min}$--$N_{max}$ processes, based on the individual weights accrued since the last resampling.

This suggests an algorithm as follows, assuming that $b^i$, $\kappa^i$, $p^i$ and $q^i$ have been chosen: 

{\bfseries Two-level Monte Carlo algorithm for computing Lyapunov Exponents}
\begin{enumerate}
	\item Fix $T, N^*, \delta t,  N_0, x_0$. Initiate with $t = 0$, and $\mathcal{X}^1, \ldots, \mathcal{X}^{N^*}$ each containing $N_0$ particles of initial value $x_0$. Set  $W = 1$. 
	\item Run each $\mathcal{X}^i$ from time $t$ to time $t' = \min\{t+\delta t, T\}$  according to the $N_{min}-N_{max}$ algorithm described in Section~\ref{sec:descr}, setting $w_i :=  \frac{\Pi_{t'}^{A^i} \Pi_{t'}^{B^i}}{\Pi_{t}^{A^i} \Pi_{t}^{B^i}}$.
	\item Set $W := W \times \left(\frac1{N^*}\sum_{i=1}^{N^*} w_i\right)$.
	\item If $t' < T$, set $t = t'$, resample each $\mathcal{X}^i$ from the existing particles $\{\mathcal{X}^1, \ldots, \mathcal{X}^{N^*}\}$ with probabilities proportional to $w_1, \ldots, w_{N^*}$, independently (multinomial resampling). Return to step 2.
	\item If $t'=T$, stop, and return
	\begin{equation*}
	\hat{\lambda} := \frac{1}{T} \log\left( W\right).
	\end{equation*}
\end{enumerate}

Here, the choice of multinomial resampling could be replaced by other resampling schemes. Similarly, the resampling could be conditional on properties of the particles (e.g. the effective sample size). We compare the effectiveness of the two methods in Figure~\ref{fig:PFMC}, for a simple Markov process where the true value of the Lyapunov exponent can be calculated accurately numerically. We note that both methods have similar computational effort, however the iterated method achieves a much better numerical estimate of $\lambda$.

\begin{figure}[H]
	\centering
	\begin{subfigure}[t]{.475\textwidth}
		\includegraphics[width=\linewidth]{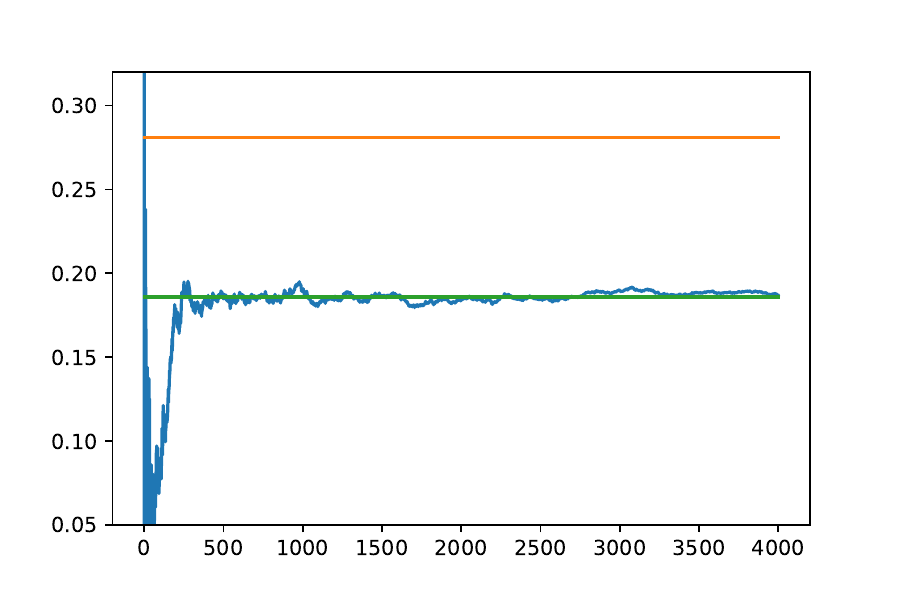}
		\caption{Single Trajectory Simulation}
	\end{subfigure}%
	\begin{subfigure}[t]{.475\textwidth}
		\includegraphics[width=\linewidth]{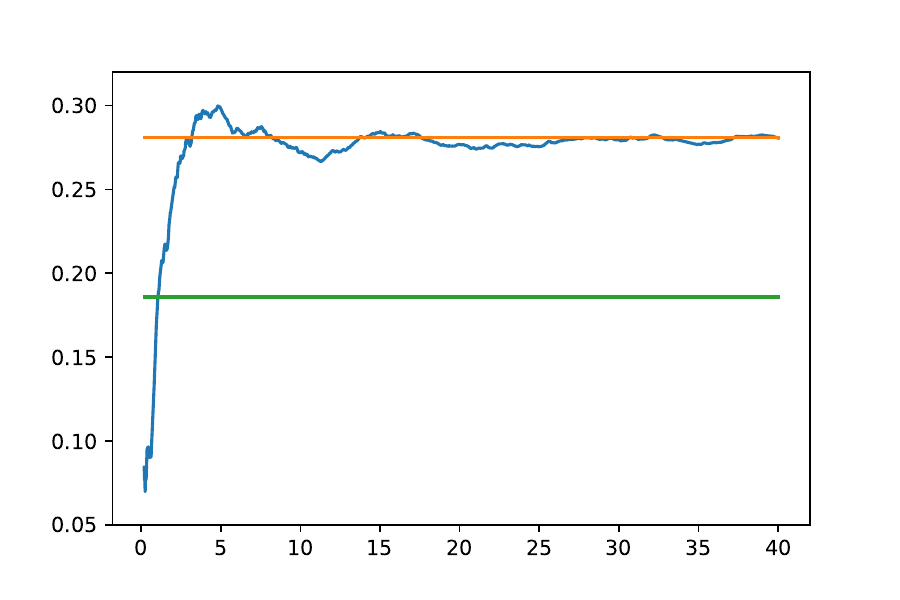}
		\caption{Two-level MC Simulation}
	\end{subfigure}
	\caption{We use two methods described above to estimate the eigenvalue $\lambda$. The first simulation (left) estimates $\lambda$ using \eqref{eq:1}, based on a single trajectory run up to time $T=4000$. The true value of $\lambda$ is highlighted with the orange line, the green line represents an analytical estimate of the convergence value for the simulation (in this case, an estimate of the average growth of the exponential weights under the stationary measure for the system). In the second figure (right), $\lambda$ is estimated using an interacting particle system, as described above. In this case, the particle filter has 100 particles, and is run for a total time of $T=40$, giving comparable computational effort. It should be noted however that the second case also has a computational cost associated with the resampling scheme applied, and so the overall computational cost will be higher depending on the resampling scheme applied, and the frequency of resampling events. }
	\label{fig:PFMC}
\end{figure}

\section{Proof of Theorem~\ref{thm:main}}
\label{sec:proof}

This section is dedicated to the proof of Theorem~\ref{thm:main}, which is rather long and so we break it up into five steps. In the first three steps, the main idea is to define two martingales $\mathbb{M}$ and $\mathcal{M}$ such that
\begin{align}
\nu_t^f - \nu_0^f &= \Pi_t^A\Pi_t^B(\mathbb M_t-\mathbb M_{\tau_{C_t}})+\sum_{n=1}^{A_t} \Pi_{\rho_n}^A\Pi_{\rho_n}^B\left(\mathcal{M}_{\rho_n} - \mathcal{M}_{\rho_n-}\right)+\sum_{n=1}^{B_t} \Pi_{\sigma_n}^A\Pi_{\sigma_n}^B(\mathcal{M}_{\sigma_n} - \mathcal{M}_{\sigma_n-})\notag\\
&\phantom{ \Pi_t^A\Pi_t^B(\mu_t-\mu_{\tau_{C_t}}}\phantom{+\sum_{n=1}^{B_t} \Pi^A_{\sigma_{n-1}}\Pi^B_{\rho_{n-1}}(\mathbb M_{\rho_n}-\mathbb M_{\rho_{n-1}})}
+\sum_{n=1}^{C_t} \Pi^A_{\tau_{n-1}}\Pi^B_{\tau_{n-1}}(\mathbb M_{\tau_n}-\mathbb M_{\tau_{n-1}}),
\label{mgdecomp2}
\end{align}
where $\sigma_n$ and $\rho_n$ denote the $n^{th}$ selection and resampling events, respectively, and 
\[
   \nu_t^f = \Pi_t^A\Pi_t^B\sum_{i \in\bar S_t} Q_{T-t}f(X_t^{i}).
\]
The martingale $\mathbb{M}$ compensates the evolution of the branching process between resampling/selection times. The martingale $\mathcal{M}$ compensates the evolution of the particle system during resampling/selection events.

The many-to-one formula \eqref{M21} then follows. The final two steps make use of the following technical lemma to control the $L^2$ norms of the martingale increments in order to obtain \eqref{L2bound}.

\begin{lem}
    \label{lem:usefulbound}
    For all $\ell\geq 1$,  for all $T\geq 0$,
    \[
    \E\left(\left(\Pi_T^B\right)^{\ell}\right)\leq \exp(c_\ell\|b\|_\infty  T),
    \]
    where
    \[
    c_{\ell}:=\sup_{x\geq 2} x\left((1+1/x)^\ell-1\right).
    \]
    In addition, denoting by $\beta_T$ the total number of branching events (with or without selection) up to time $T$, we have
    \[
    \E\left(\beta_T\,\left(\Pi_T^B\right)^{2}\right)\leq N_0\exp((3+c_4) \|b\|_\infty T/2).
    \]
\end{lem}

The above estimates are the main reason why we require $b$ to be bounded. These types of estimates are not difficult to derive in some particular situations where the branching rate is unbounded (typically for time inhomogeneous branching rates, or in situations where the branching rates is high in rarely visited regions), and in these situations, our proofs adapt easily. We defer the proof of the above lemma to the end of the section and now proceed to the proof of Theorem~\ref{thm:main}.

\bigskip

\begin{proof}[Proof of Theorem~\ref{thm:main}]
$ $\newline
\textit{Step 1. The martingale $\mathbb{M}$.} For all $i\in\mathbb N$, we define respectively the \textit{birth time} $\BD_i$ and the \textit{death time} $\DD_i$ of the particle with index $i$ as
\begin{align*}
\BD_i=\inf\{t\geq 0,\ i\in\bar S_t\}\text{ and }\texttt D_i=\inf\{t\geq \BD_i,\ i\notin\bar S_t\}.
\end{align*}

Fix $T > 0$ and $t\in [0,T]$. For all particles $i\in\mathbb N$, we set
\begin{equation}
\mathbb{M}^{i,n}_t = 
\begin{cases}
0, \quad \text{if } t \le \BD_i\vee \tau_n, \\
Q_{T-t}f(X^{i}_t) - Q_{T-\BD_i\vee \tau_n}f(X_{\BD_i\vee\tau_n}^{i}), \quad \text{if } t \in [\BD_i\vee \tau_n, \DD_i\wedge \tau_{n+1}),\\
(1-\mathbf{1}_{D^n(i)}+\mathbf{1}_{B^n(i) })Q_{T-\DD_i\wedge \tau_{n+1}}f(X_{\DD_i\wedge \tau_{n+1}}^{i}) - Q_{T-\BD_i\vee \tau_n}f(X^{i}_{\BD_i\vee \tau_n}) , \quad \text{if } t\ge \DD_i\wedge \tau_{n+1},
\end{cases}
\end{equation}
where $B^n(i)$ denotes the event ``$\tau_{n+1}$ is a branching time for the $i$-th particle'' and $D^n(i)$ denotes the event ``$\tau_{n+1}=\DD_i$ is a soft killing time for the $i$-th particle''. 

\smallskip

By definition of the process between the time intervals $[\tau_n,\tau_{n+1})$,  to show that $\mathbb{M}$ is indeed a martingale, it is sufficient to prove that
\[
	\mathbb{M}^{1,0}_t:=\begin{cases}
	Q_{T-t}f(X^1_t) - Q_{T}f(X^1_0), \quad \text{if } t \in [0, \tau_1),\\
	(1-\mathbf 1_{D(1)}+\mathbf{1}_{B(1)})Q_{T-\tau_1}f(X_{\tau_1}) - Q_{T}f(X_{0}), \quad \text{if } t\ge \tau_1,
	\end{cases}
\]
is a martingale. 
We observe that, by construction, $\tau_1=\tau^\partial\wedge \tau^\kappa\wedge\tau^b$ where 
\begin{align*}
\tau^b&=\inf_{j\in \bar S_0}\inf\{t\geq 0,\ \int_0^t b^j((X^i_u)_{i\in\bar S_u})\,\mathrm du\geq e^b\},\\
\tau^\kappa&=\inf_{j\in \bar S_0}\inf\{t\geq 0,\ \int_0^t \kappa^j((X^i_u)_{i\in\bar S_u})\,\mathrm ds\geq e^\kappa\},\\
\tau^\partial&=\inf_{j\in \bar S_0}\inf\{t\geq 0\ X^j_t\in \partial\}.
\end{align*}

\smallskip

Then, for all $s,t\geq 0$ such that $s+t\leq T$, setting $\mathbb{E}^t=\E[\cdot\mid \sigma(\bar S_u,\,\bar X_u,0\leq u \leq t)]$ and using the fact that $Q_u f\equiv 0$ on $\partial$ for all $u\geq 0$, we have
\begin{align}
\label{eq:MartStep0}
\E^t(\mathbb{M}^{1,0}_{t+s})=\E^t&\left[\left(Q_{T-(t+s)}f(X^1_{t+s}) - Q_{T}f(X_0)\right)\mathbf{1}_{t+s<\tau_1}\right]\notag\\
&\qquad+\E_x^t\left[\left((1-\mathbf 1_{D(1)}+\mathbf{1}_{B(1)})Q_{T-\tau_1}f(X_{\tau_1}) - Q_{T}f(X_0)\right)\ind_{t<\tau_1\leq t+s}\right]\nonumber\\
&\qquad+\left((1-\mathbf 1_{D(1)}+\mathbf{1}_{B(1)})Q_{T-\tau_1}f(X_{\tau_1}) - Q_{T}f(X_0)\right)\ind_{\tau_1\leq t}
\end{align}
Using the fact that the particles are independent copies of $X$ before time $\tau_1$, and the Markov property at time $t$, we obtain
\begin{align*}
\E_x^t\left[Q_{T-(t+s)}f(X^1_{t+s})\mathbf{1}_{t+s<\tau_1}\right]
&=\texttt E^t\left[Q_{T-(t+s)}f(Y^1_{t+s})e^{-\int_t^{t+s} h_u\,\D u}\ind_{t+s<\tau_\dd}\right]\mathbf{1}_{t<\tau_1},
\end{align*}
where $\texttt E^t$ is the expectation on a probability space $\Omega'$ such that $(Y^i_u)_{u\geq t}$, $i\in\bar S_0$, are independent copies of $X$, starting from $X^i_t$ at time $t$, and where  
\[
h_u:=\sum_{j\in\bar S_0} b^j((Y^i_u)_{i\in\bar S_0})+\kappa^j((Y^i_u)_{i\in\bar S_0}).
\]
 Then, using the Markov property at time $t+s$, we obtain
\begin{align*}
\E_x^t\left[Q_{T-(t+s)}f(X^1_{t+s})\mathbf{1}_{t+s<\tau_1}\right]
&=\texttt E^t\left[e^{\int_{t+s}^T b(Y^1_u)-\kappa(Y^1_u)\,\D u}f(Y^1_{T})e^{-\int_t^{t+s} h_u\,\D u}\ind_{t+s<\tau_\dd}\right]\mathbf{1}_{t<\tau_1}\\
&=\texttt E^t\left[e^{\int_{t}^T b(Y^1_u)-\kappa(Y^1_u)\,\D u}f(Y^1_{T})e^{-\int_t^{t+s} h^1_u\,\D u}\ind_{t+s<\tau_\dd}\right]\mathbf{1}_{t<\tau_1},
\end{align*}
where $h^1_u:=2b^1((Y^i_u)_{i\in\bar S_0})+\sum_{j\in\bar S_0,j\neq 1} b^j((Y^i_u)_{i\in\bar S_0})+\kappa^j((Y^i_u)_{i\in\bar S_0})$.
Similarly, since $\kappa^1$ is the intensity of the soft killing time for the first particle, and $b^1$ is the intensity of its branching time, we obtain
\begin{align*}
&\E_x^t\left[(1-\ind_{D(1)}+\ind_{B(1)})Q_{T-\tau_1}f(X^1_{\tau_1})\mathbf{1}_{t<\tau_1\leq t+s}\right]\\
&=\texttt E^t\left[\int_t^{(t+s)\wedge\tau_\dd}(h_u-\kappa^1((Y^i_u)_{i\in\bar S_0})+b^1((Y^i_u)_{i\in\bar S_0}))Q_{T-u}f(Y^1_{u})e^{-\int_t^u h_v\,\mathrm dv}\mathrm du\right]\,\ind_{t<\tau_1}\\
&\qquad+\texttt E^t\left[(1-\ind_{Y^1_{\tau_\dd}\in\dd })e^{-\int_t^{\tau_\dd} h_u\,\mathrm du}\ind_{\tau_\dd\leq t+s}Q_{T-\tau_\dd}f(Y^1_{\tau_\dd})\right]\,\ind_{t<\tau_1}.
\end{align*}
Hence, using the fact that $Y$ satisfies the strong Markov property at time $\tau_\partial$, and the fact that $Q_{T-\tau_\partial}f(Y^1_{\tau_\partial})=0$ on the event $\{Y^1_{\tau_\dd}\in\dd\}$, we deduce that
\begin{align*}
&\E_x^t\left[(1-\ind_{D(1)}+\ind_{B(1)})Q_{T-\tau_1}f(X^1_{\tau_1})\mathbf{1}_{t<\tau_1\leq t+s}\right]\\
&=\texttt E^t\left[\int_t^{(t+s)\wedge \tau_\partial}h^1_u\,e^{\int_{u}^T (b(Y^1_v)-\kappa(Y^1_v))\,\mathrm dv}\,f(Y^1_T)e^{-\int_t^u h_v\,\mathrm dv}\,\mathrm du\right]\,\ind_{t<\tau_1}\\
&\qquad+\texttt E^t\left[e^{-\int_t^{\tau_\dd} h_u\,\mathrm du}\ind_{\tau_\dd\leq t+s}e^{\int_{\tau_\dd}^{T} (b(Y^1_v)-\kappa(Y^1_v))\,\mathrm dv }f(Y^1_{T})\right]\,\ind_{t<\tau_1}\\
&=\texttt E^t\left[e^{\int_{t}^T (b(Y^1_v)-\kappa(Y^1_v))\,\mathrm dv}\,f(Y^1_T)\,\left(\int_t^{(t+s)\wedge\tau_\dd}h^1_u\,e^{-\int_t^u h^1_v\,\mathrm dv}\mathrm du+e^{-\int_t^{\tau_\dd} h^1_v\,\mathrm dv}\ind_{\tau_\dd\leq t+s}\right)\right]\,\ind_{t<\tau_1},
\end{align*}
where we used the assumption $b(Y^1_v)-\kappa(Y^1_v)=b^1((Y^i_u)_{i\in\bar S_0})-\kappa^1((Y^i_u)_{i\in\bar S_0})$. But 
\[
e^{-\int_t^{t+s} h^1_u\,\D u}\ind_{t+s<\tau_\dd}+\int_t^{(t+s)\wedge\tau_\dd}h^1_u\,e^{-\int_t^u h^1_v\,\mathrm dv}\mathrm du+e^{-\int_t^{\tau_\dd} h^1_v\,\mathrm dv}\ind_{\tau_\dd\leq t+s}=1
\]

Hence we deduce from~\eqref{eq:MartStep0} that
\begin{align*}
\E_x^t(\mathbb{M}^{1,0}_{t+s})&=(1-\mathbf 1_{D(1)}+\mathbf{1}_{B(1)})\left(Q_{T-\tau_1}f(X_{\tau_1}) -Q_{T}f(X_0)\right)\ind_{\tau_1\leq t}\\
&\qquad+ \left(\texttt E^t\left[e^{\int_{t}^T (b(Y^1_v)-\kappa(Y^1_v))\,\mathrm dv}\,f(Y^1_T)\right]-Q_Tf(x) \right) \ind_{t<\tau_1}\\
                              &=(1-\mathbf 1_{D(1)}+\mathbf{1}_{B(1)})Q_{T-\tau_1}f(X_{\tau_1})\ind _{\tau_1 \le t}  + Q_{T-t}f(X^1_t)\,\ind_{t<\tau_1}-Q_Tf(x)\\
  & = \mathbb{M}^{1,0}_{t}.
\end{align*}
This concludes Step~1.

\bigskip\noindent\textit{Step 2. A first martingale decomposition.}

Fix $t \ge 0$ and recall that $C_t = \inf\{n \ge 0 : t<\tau_{n+1}\}$.  In what follows, we denote by $\rho_1,\rho_2,\ldots$ the sequence of times at which a resampling event occurs, and by $\sigma_1,\sigma_2,\ldots$ the sequence of times at which a selection event occurs. Note that, if $\tau_n\in\{\rho_1,\rho_2,\ldots\}$ or $\tau_n\in\{\sigma_1,\sigma_2,\ldots\}$, then $N_n=N_{n-1}\in\{N_{min},\ldots,N_{max}\}$.

\bigskip

We will show that, for all $t \ge 0$, we have
\begin{align}
\mathbb M_t:=\sum_{n=0}^{C_t}\sum_{i \in \bar S_t}\mathbb{M}_t^{i,n} = &\sum_{i\in\bar S_t}Q_{T-t}f(X_t^{i}) -\sum_{i\in\bar S_0}Q_{T}f(X_0^{i}) \notag\\
&\qquad- \sum_{n=1}^{A_t}Q_{T-\rho_{n}}f(X_{\rho_{n}}^{i_{n}}) + \sum_{n=1}^{B_t}Q_{T-\sigma_n}f(X_{\sigma_n}^{j_{n-1}}),
\label{mgdecomp1}
\end{align}
where we recall that $A_t$ is the number of resampling events up to time $t$, $B_t$ is the number of selection events up to time $t$, and where we have set $i_{n}$ to denote the index of the particle added at time $\rho_n$ and $j_{n-1}$ is the particle killed (due to a selection event) at time $\sigma_n$ (as above, $X_{\rho_{n}}^{j_{n-1}}$ denotes the position of the particle before the selection event).

\smallskip

We will show that \eqref{mgdecomp1} holds for $n = 0, 1$, since the general formula then follows using similar arguments. 

\smallskip

\noindent Noting that for all $t\in[\tau_0,\tau_1)$, $C_t=A_t=B_t=0$, in this case we have
 \begin{align*}
\sum_{n=0}^{C_t} \sum_{i \in  \bar S_t}\mathbb{M}_t^{i,n}= \sum_{i\in \bar S_0}\mathbb{M}_t^{i,0}=\sum_{i\in\bar S_0}Q_{T-t}f(X_t^{i})-\sum_{i\in\bar S_0}Q_{T}f(X_0^{i}),
 \end{align*}
as required.

\smallskip
 
Let us now describe the evolution of the martingale at time $\tau_1$, depending on the type of event that occurs.
 
 \medskip\noindent\textbf{Case  $\tau_1=\rho_1$.}
In this case, for all $t\in[\tau_1,\tau_2)$, we have $A_t = 1$ and $B_t = 0$, and $\bar S_t=\{1+\max \bar S_0\}\cup \bar S_0\setminus\{i'_0\}$, where $i'_0$ is the particle removed at time $\tau_1$. Thus
 \begin{align*}
\sum_{n=0}^{C_t} \sum_{i \in \bar S_t}\mathbb{M}_t^{i,n} &=  \sum_{i\in \bar S_0}\mathbb{M}_t^{i,0}+\sum_{i\in \bar S_{\tau_1}} \mathbb{M}_t^{i,1}\\
&= \sum_{i\in \bar S_0} \left[(1-\ind_{i=i'_0})Q_{T-\tau_1}f(X_{\tau_1}^{i})-Q_{T}f(X_0^{i})\right]+ \sum_{i\in \bar S_{\tau_1}} \big[Q_{T-t}f(X^{i}_t) - Q_{T-\tau_1}f(X_{\tau_1}^{i})\big]\\
&= \sum_{i\in \bar S_{\tau_1}} Q_{T-t}f(X_{t}^{i})-\sum_{i\in \bar S_0} Q_{T}f(X_0^{i}) - Q_{T-\tau_1}f(X_{\tau_1}^{1+\max \bar S_0})\\
&= \sum_{i\in\bar S_t} Q_{T-t}f(X_{t}^{i})-\sum_{i\in \bar S_0} Q_{T}f(X_0^{i})-Q_{T-\tau_1}f(X_{\tau_1}^{i_1}),
 \end{align*}
as claimed.

 \medskip\noindent\textbf{Case  $\tau_1=\sigma_1$.}
 Similarly, if $\tau_1=\sigma_1$, then, for all $t\in[\tau_1,\tau_2)$, we have $A_t = 0$ and $B_t = 1$, and $\bar S_t=\{1+\max \bar S_0\}\cup \bar S_0\setminus\{j_0\}$. Denoting by $j'_0$ the particle duplicated at time $\tau_1$, we have
  \begin{align*}
 \sum_{n=0}^{C_t} \sum_{i \in \bar S_t}\mathbb{M}_t^{i,n} 
    &= \sum_{i\in \bar S_0}\mathbb{M}_t^{i,0}+\sum_{i \in \bar S_{\tau_1}}\mathbb{M}_t^{i,1}\\
 &= \sum_{i\in \bar S_0}\left[(1+\ind_{i=j'_0})Q_{T-\tau_1}f(X_{\tau_1}^{i})-Q_{T}f(X_0^{i})\right]
 + \sum_{i\in \bar S_{\tau_1}} \big[Q_{T-t}f(X^{i}_t) - Q_{T-\tau_1}f(X_{\tau_1}^{i})\big]\\
    &= \sum_{i\in S_1} Q_{T-t}f(X^{i}_t)-\sum_{i\in S_{0}}Q_{T}f(X_0^{i})+Q_{T-\tau_1}f(X^{j_0}_{\tau_1})\\
    &= \sum_{i\in \bar S_t} Q_{T-t}f(X^{i}_t)-\sum_{i\in \bar S_{0}}Q_{T}f(X_0^{i})+Q_{T-\tau_1}f(X^{j_0}_{\tau_1}).
 \end{align*}

\medskip\noindent\textbf{Case  $\tau_1\notin \{\rho_1,\sigma_1\}$.} If $\tau_1\notin \{\rho_1,\sigma_1\}$, we have, for all $t\in[\tau_1,\tau_2)$, $A_t=B_t=0$. If $\tau_1$ is a branching event of the particle with index $j'_0\in S_0$, then $\bar S_t=\{1+\max \bar S_0\}\cup\bar S_0$
\begin{align*}
\sum_{n=0}^{C_t} \sum_{i \in \bar S_t}\mathbb{M}_t^{i,n}&=\sum_{i\in \bar S_0} \left[(1+\ind_{i=j'_0})Q_{T-\tau_1}f(X^{i}_{\tau_1})-Q_{T}f(X_0^{i})\right]+\sum_{i\in \bar S_{\tau_1}} \left[Q_{T-t}f(X^{i}_t)-Q_{T-\tau_1}f(X_{\tau_1}^{i})\right]\\
&=\sum_{i\in \bar S_{\tau_1}} Q_{T-t}f(X^{i}_t)-\sum_{i\in \bar S_0} Q_{T}f(X^{i}_{0}),
\end{align*}
where we used the fact that $X^{j'_0}_{\tau_1}=X^{1+\max \bar S_0}_{\tau_1}$.

\smallskip

Continuing in this manner, we obtain~\eqref{mgdecomp1}.

\bigskip\noindent\textit{Step 3. A second martingale decomposition.}

For all time $t\geq 0$ and all bounded measurable function $f:E\to \mathbb{R}$, we
denote the integral of $f$ with respect to the occupation measure of the particle system at time $t$ by
\begin{equation}
\mu_t^f \coloneqq \sum_{i \in \bar S_{t}}Q_{T-t}f(X_t^{i}).
\label{empdist}
\end{equation}
Further, recall the definitions of $\Pi_t^A$, $\Pi_t^B$ and $\nu_t^f = \Pi_t^A\Pi_t^B\mu_t^f$. 

\bigskip

In this step, we prove that equation \eqref{mgdecomp2} holds with $\mathbb{M}$ defined as in the previous steps, and $\mathcal{M}$ to be defined shortly. To this end, we first note that from Step~2 we obtain
\begin{align}
\mu_t^f - \mu_0^f = \sum_{n = 0}^{C_t}\sum_{i\in \bar S_t}\mathbb{M}_t^{i,n} &+  \sum_{n=1}^{A_t}Q_{T-\rho_{n}}f(X_{\rho_{n}}^{i_{n-1}})-\sum_{n=1}^{B_t}Q_{T-\sigma_n}f(X_{\sigma_n}^{j_{n-1}}).
\end{align}
For $n \in \mathbb{N}$, consider the following martingale increments. We define
\begin{equation}
\mathcal{M}_{\rho_n} - \mathcal{M}_{\rho_n-} \coloneqq 
Q_{T - \rho_n}f(X_{\rho_n}^{i_{n-1}}) 
- \frac{1}{\bar N_{\rho_n} - 1}\sum_{i\in \bar S_{\rho_n-}\setminus\{ i'_{n-1}\}}Q_{T-\rho_n}f(X_{\rho_n}^{i}),
\label{mg2rho}
\end{equation}
where $i_{n-1}$ is the index of the particle duplicated at time $\rho_n$ and $i'_{n-1}$ is the index of the particle removed at time $\rho_n$ (note also that $\bar N_{\rho_n}=\bar N_{\rho_n-}=|\bar S_{\rho_n-}|$). We also define
\begin{equation}
\mathcal{M}_{\sigma_n} - \mathcal{M}_{\sigma_n-} \coloneqq 
\frac{1}{\bar N_{\sigma_n}+1}\left(\sum_{i\in \bar S_{\sigma_n-}}Q_{T-\sigma_n}f(X_{\sigma_n}^{i})+Q_{T-\sigma_n}f(X_{\sigma_n}^{j'_{n-1}})\right) - Q_{T-\sigma_n}f(X_{\sigma_n}^{j_{n-1}}),
\label{mg2sigma}
\end{equation}
where we recall that $j_{n-1}$ is the particle removed at time $\sigma_n$ and $j_{n-1}'$ is the index of the particle duplicated during the branching event at this time (note also that $\bar N_{\sigma_n}=\bar N_{\sigma_n-}=|\bar S_{\sigma_n-}|$). We assume further that $\mathcal{M}$ is constant except at these jump times.

Then, 
\begin{align}
\mu_t^f - \mu_0^f = \mathbb{M}_t &+ \mathcal{M}_t + \sum_{n=1}^{A_t}\frac{1}{\bar N_{\rho_n} - 1}\sum_{i\in \bar S_{\rho_n-}\setminus\{i_{n-1}'\}}Q_{T-\rho_n}f(X_{\rho_n}^{i})\notag \\
& - \sum_{n=1}^{B_t}\frac{1}{\bar N_{\sigma_n}+1}\left(\sum_{i\in \bar S_{\sigma_n-}}Q_{T-\sigma_n}f(X_{\sigma_n}^{i})+Q_{T-\sigma_n}f(X_{\sigma_n}^{j'_{n-1}})\right).
\end{align}

 For each $n\geq 1$, we define $\mu_{\tau_n-}^f$ and $\nu_{\tau_n-}^f$ by
\[
\mu_{\tau_n-}^f=\sum_{i\in S_{n-1}} f(X^{i}_{\tau_n})-\ind_{\tau_n\in\mathcal K} f(X^{i'_{n-1}}_{\tau_n})+\mathbf \ind_{\tau_n\in \mathcal B}f(X^{j'_{n-1}}_{\tau_n}) \text{ and }\nu_{\tau_n-}=\Pi_{\tau_{n-1}}^A\Pi_{\tau_{n-1}}^B \mu_{\tau_n-}
\]
where $\mathcal K$ is the set of killing times and $\mathcal B$ is the set of branching times. Informally, $\mu_{\tau_n-}$ and $\nu_{\tau_n-}$ represent the state of the particle system at time $\tau_n$ before the resampling or selection eventually occurring at time $\tau_n$. Note that, at any time $\tau_n$ that is not a resampling or a selection time, we have $\mu_{\tau_n-}=\mu_{\tau_n}$ and $\nu_{\tau_n-}=\nu_{\tau_n}$. Hence,  denoting by $(\theta_n)_n$ the sequence of events from the families $\rho$ or $\sigma$, we obtain for all $t\geq 0$,
\begin{align*}
\nu_t^f - \nu_0^f &= \nu_t-\nu_{\tau_{C_t}} +\sum_{n=1}^{C_t} \left(\nu_{\tau_n}-\nu_{\tau_{n}-}\right)
					 +\sum_{n=1}^{C_t} \left(\nu_{\tau_n-}-\nu_{\tau_{n-1}}\right)\\
                  &= \Pi_t^A\Pi_t^B(\mu_t-\mu_{\tau_{C_t}})+\sum_{n=1}^{A_t} 
                  	\nu_{\rho_n}-\nu_{\rho_{n}-}+\sum_{n=1}^{B_t} \nu_{\sigma_n}-\nu_{\sigma_{n}-}+\sum_{n=1}^{C_t} \nu_{\tau_n-}-\nu_{\tau_{n-1}}.
\end{align*}

At resampling times $\rho_n$, we have $\Pi_{\rho_n}^B - \Pi_{\rho_n-}^B = 0$ and hence,
\begin{align}
\nu_{\rho_n}^f - \nu_{\rho_n-}^f
&= \Pi_{\rho_n}^A\Pi_{\rho_n}^B\mu_{\rho_n}^f-\Pi_{\rho_{n-1}}^A\Pi_{\rho_{n-1}}^B\mu_{\rho_n-}^f \notag\\
 &= \Pi_{\rho_n}^A\Pi_{\rho_n}^B(\mu_{\rho_n}^f - \mu_{\rho_n-}^f) + \Pi_{\rho_n}^B\mu_{\rho_n-}^f\left(\Pi_{\rho_n}^A - \Pi_{\rho_{n-1}}^A \right)\notag \\
&= \Pi_{\rho_n}^A\Pi_{\rho_n}^B \,Q_{T-\rho_n}f(X_{\rho_n}^{i_{n-1}})+ \Pi_{\rho_n}^B\mu_{\rho_n-}^f\left(\Pi_{\rho_n}^A - \Pi_{\rho_{n-1}}^A \right)\label{nu2}
\end{align}
For the increment of $\Pi^A$, we note that
\begin{equation*}
	\Pi_{\rho_n}^A-\Pi_{\rho_{n-1}}^A = \Pi_{\rho_n}^A-\left(\frac{\bar N_{\rho_n}}{\bar N_{\rho_n}-1}\right)\Pi_{\rho_n}^A=-\frac{1}{\bar N_{\rho_n}-1}\Pi_{\rho_n}^A.
\end{equation*}
Substituting this into~\eqref{nu2} and using~\eqref{mg2rho}, we obtain
\begin{align}
\nu_{\rho_n}^f - \nu_{\rho_n-}^f &= \Pi_{\rho_n}^A\Pi_{\rho_n}^B\left(Q_{T-\rho_n}f(X_{\rho_n}^{i_{n-1}}) - \frac{1}{\bar N_{\rho_n}-1}\mu_{\rho_n-}^f \right) \notag\\
&= \Pi_{\rho_n}^A\Pi_{\rho_n}^B\left(\mathcal{M}_{\rho_n} - \mathcal{M}_{\rho_n-}\right).
\end{align}

Similarly, for $n \in \{1, \dots, B_t\}$,
\begin{align}
\nu_{\sigma_n}^f - \nu_{\sigma_n-}^f = \Pi_{\sigma_n}^A\Pi_{\sigma_n}^B(\mu_{\sigma_n}^f - \mu_{\sigma_n-}^f) + \Pi_{\sigma_n}^A\mu_{\sigma_n-}^f\left(\Pi_{\sigma_n}^B - \Pi_{\sigma_n-}^B \right). \label{nu3}
\end{align}
In this case, we have
\begin{equation}
\Pi_{\sigma_n}^B - \Pi_{\sigma_n-}^B = \Pi_{\sigma_n}^B-\left(\frac{\bar N_{\sigma_n}}{\bar N_{\sigma_n}+1}\right)\Pi_{\sigma_n}^B = \frac{1}{\bar N_{\sigma_n}+1}\Pi_{\sigma_n}^B,
\end{equation}
and
\begin{equation}
\mu_{\sigma_n}^f - \mu_{\sigma_n-}^f =  - Q_{T-\sigma_n}f(X_{\sigma_n}^{j_{n-1}}),
\end{equation}
where $j_{n-1}$ is the index of the particle removed from the system at time $\sigma_n$. Again, substituting these equalities back into~\eqref{nu3} and using~\eqref{mg2sigma} gives,
\begin{align}
\nu_{\sigma_n}^f - \nu_{\sigma_n-}^f &= \Pi_{\sigma_n}^A\Pi_{\sigma_n}^B\left( - Q_{T-\sigma_n}f(X_{\sigma_n}^{j_{n-1}})+\frac{1}{\bar N_{\sigma_n}+1}\mu_{\sigma_n-}^f\right)\notag\\
&=  \Pi_{\sigma_n}^A\Pi_{\sigma_n}^B(\mathcal{M}_{\sigma_n} - \mathcal{M}_{\sigma_n-}).
\end{align}
Finally, we have $\nu_{\tau_n-}-\nu_{\tau_{n-1}}=\Pi^A_{\tau_{n-1}}\Pi^B_{\tau_{n-1}}(\mathbb M_{\tau_n}-\mathbb M_{\tau_{n-1}})$ for all $n\geq 1$, hence
\begin{align*}
\nu_t^f - \nu_0^f &= \Pi_t^A\Pi_t^B(\mathbb M_t-\mathbb M_{\tau_{C_t}})+\sum_{n=1}^{A_t} \Pi_{\rho_n}^A\Pi_{\rho_n}^B\left(\mathcal{M}_{\rho_n} - \mathcal{M}_{\rho_n-}\right)+\sum_{n=1}^{B_t} \Pi_{\sigma_n}^A\Pi_{\sigma_n}^B(\mathcal{M}_{\sigma_n} - \mathcal{M}_{\sigma_n-})\\
&\phantom{ \Pi_t^A\Pi_t^B(\mu_t-\mu_{\tau_{C_t}}}\phantom{+\sum_{n=1}^{B_t} \Pi^A_{\sigma_{n-1}}\Pi^B_{\rho_{n-1}}(\mathbb M_{\rho_n}-\mathbb M_{\rho_{n-1}})}
+\sum_{n=1}^{C_t}\Pi^A_{\tau_{n-1}}\Pi^B_{\tau_{n-1}}(\mathbb M_{\tau_n}-\mathbb M_{\tau_{n-1}}),
\end{align*}
which entails~\eqref{mgdecomp2} and thus concludes Step~3.

\bigskip
\noindent\textit{Step 4. Control of the $L^2$ norm of the martingale increments.}
In this section, we assume without loss of generality (taking $e^{-\|b\|_\infty T}f/\|f\|_\infty$ instead of $f$) that
\[
\sup_{t\in[0,T]} \|Q_t f\|_{\infty}\leq 1.
\]

We now show that the $L^2$ norm of each of these quantities on the righthand side of \eqref{mgdecomp2} is finite. To this end, let us first consider $\mathcal{M}_{\rho_n} - \mathcal{M}_{\rho_n-}$. We have, according to~\eqref{mg2rho},
\begin{align}
|\mathcal{M}_{\rho_n} - \mathcal{M}_{\rho_n-}|&=\left| Q_{T-\rho_n}f(X_{\rho_n}^{i_{n-1}}) - \frac{1}{\bar N_{\rho_n} - 1}\sum_{i\in \bar S_{\rho_n-}\setminus\{i'_{n-1}\}}Q_{T-\rho_n}f(X_{\rho_n}^{i})\right|
 \le 2.
\end{align}
Then, setting $R_T=\max_{t\in[0,T]} N_t$, we obtain
\begin{align}
\mathbb{E}&\left[\left(\sum_{n=1}^{A_T}\Pi_{\rho_n}^A\Pi_{\rho_n}^B\left(\mathcal{M}_{\rho_n} - \mathcal{M}_{\rho_n-}\right) \right)^2 \right]\notag \\
&\quad\quad=\mathbb{E}\left[\left(\frac{\bar N_{\rho_n} - 1}{\bar N_{\rho_n}}\,\sum_{n=1}^{A_T}\Pi_{\rho_n-}^A\Pi_{\rho_n-}^B\left(\mathcal{M}_{\rho_n} - \mathcal{M}_{\rho_n-}\right) \right)^2 \right]\notag \\
&\quad\quad \le 4 \,
\mathbb{E}\left[ \sum_{n=1}^{A_T}\left(\frac{R_T - 1}{R_T}\right)^{2 A_{\rho_n}}(\Pi_{\rho_n-}^B)^2\right]\notag\\
&\quad\quad \le 4\,
\mathbb{E}\left[ (\Pi_T^B)^2\,\sum_{n=1}^{A_T}\left(\frac{R_T - 1}{R_T}\right)^{2 n}\right].\label{bounda1} 
\end{align}
To control the sum term, since $1/R_T \le \frac12$, we have
\begin{align*}
  \sum_{n=1}^{A_T} \left(1-\frac{1}{R_T}\right)^{2n}
  &\leq \sum_{n=1}^{+\infty} \left(1-\frac{1}{R_T}\right)^{2n}= \frac{(1-\frac{1}{R_T})^2}{1-(1-\frac{1}{R_T})^2}\\
  &= R_T\frac{1-2/R_T+R_T^{-2}}{2-\frac{1}{R_T}}
  \leq R_T/2\leq (N_0+\beta_T)/2,
\end{align*}
where $\beta_T$ denotes the total number of branching events before time $T$.
We now make use of Lemma~\ref{lem:usefulbound}, stated and proved at the end of this section, which entails that
\begin{align}
\label{boundmomB}
\mathbb{E}\left[(\Pi_T^B)^2 (N_0+\beta_T)/2\right]&\leq  N_0\exp(c \|b\|_\infty  T),
\end{align}
for some (explicit) constant $c>0$.
Substituting these estimates into \eqref{bounda1}, we obtain
\begin{align}
\mathbb{E}\left[\left(\sum_{n=1}^{A_T}\Pi_{\rho_n}^A\Pi_{\rho_n}^B\left(\mathcal{M}_{\rho_n} - \mathcal{M}_{\rho_n-}\right) \right)^2 \right]
 \le 4\,N_0 \,
 \exp\left(c \|b\|_{\infty}T\right).\label{bounda2}
\end{align}

\bigskip

Let us now consider $\Delta_{\sigma_n}\mathcal{M}$. Using ~\eqref{mg2sigma}, we have 
\begin{align*}
|\mathcal{M}_{\sigma_n}-\mathcal{M}_{\sigma_{n-}}|&\le 2.
\end{align*}
Then
\begin{align}
\mathbb{E}&\left[\left(\sum_{n=1}^{B_T}\Pi_{\sigma_n}^A\Pi_{\sigma_n}^B(\mathcal{M}_{\sigma_n}-\mathcal{M}_{\sigma_{n-}})\right)^2 \right]\\
&\quad\quad\leq \mathbb{E}\left[\left(\sum_{n=1}^{B_T}\Pi_{\sigma_{n-}}^A\Pi_{\sigma_{n-}}^B\,\frac{N_{\sigma_{n-}}+1}{N_{\sigma_{n-}}}(\mathcal{M}_{\sigma_n}-\mathcal{M}_{\sigma_{n-}}) \right)^2 \right]\notag \\
&\quad\quad \le 4\,
\mathbb{E}\left[ \sum_{n=1}^{B_T}\left(\Pi^B_{\sigma_n}\right)^2\right] \le 
4\,\mathbb{E}\left[ \beta_T\,\left(\Pi^B_T\right)^2\right].\label{boundb1}
\end{align}
Using Lemma~\ref{lem:usefulbound} stated and proved at the end of this section, we deduce that
\begin{align}
\mathbb{E}\left[\left(\sum_{n=1}^{B_T}\Pi_{\sigma_n}^A\Pi_{\sigma_n}^B(\mathcal{M}_{\sigma_n}-\mathcal{M}_{\sigma_{n-}})\right)^2 \right]
 \le 4\, N_0 \exp(c \|b\|_\infty T),\label{bounda3}
\end{align}
where $c$ can be chosen to be the same as in~\eqref{bounda2}.

\medskip

It remains to control the increments of $\mathbb{M}$. For all $n\geq 1$, let $E_n\subset \mathbb N$ be the set of indices of particles defined by
\[
E_n=\{i\in\mathbb N\text{ s.t. the $i^{th}$ particle is branching, killed, born, removed or duplicated at time $\tau_n$}\}.
\]
Observe that $|E_n|\leq 4$ for all $n\geq 1$. We also define, for all $i\in\mathbb N$, the sequence of times $(\tau_n^i)_{n\geq 0}$ by $\tau^i_0=0$ and
\begin{align*}
\tau_{n+1}^i=\min\{\tau_k>\tau^i_n,\text{ with }k\geq 0\text{ and }i\in E_k\}.
\end{align*}
We write 
\begin{align*}
\mathbb M_T=\sum_{i\in\mathbb N}\sum_{n\geq 1} \mathbb M^i_{\tau_{n}^i\wedge T}-\mathbb M^i_{\tau_{n-1}^i\wedge T}.
\end{align*}
Since all the martingales increments are orthogonal, we have, for all $i\in\mathbb N$,
\begin{multline*}
\Delta:=\mathbb E\left[\left(\Pi_T^A\Pi_T^B(\mathbb M_T-\mathbb M_{\tau_{C_T}})+\sum_{n=1}^{C_T}\Pi^A_{\tau_{n-1}}\Pi^B_{\tau_{n-1}}(\mathbb M_{\tau_n}-\mathbb M_{\tau_{n-1}})\right)^2\right]\\
\begin{aligned}
&=\mathbb E\left[\left(\sum_{i\in\mathbb N}\sum_{n\geq 1} \Pi_{\tau_{n-1}^i}^{A}\Pi_{\tau_{n-1}^i}^{B}(\mathbb M^i_{\tau_{n}^i\wedge T}-\mathbb M^i_{\tau_{n-1}^i\wedge T})\right)^2\right]\\
&=\mathbb E\left[\sum_{i\in \mathbb N}\sum_{n\geq 1}\left( \Pi_{\tau_{n-1}^i}^{A}\Pi_{\tau_{n-1}^i}^{B}\right)^2\left(\mathbb M^i_{\tau_{n}^i\wedge T}-\mathbb M^i_{\tau_{n-1}^i\wedge T}\right)^2\right].
\end{aligned}
\end{multline*}
But the increments $\mathbb M^i_{\tau_{n+1}^i\wedge T}-\mathbb M^i_{\tau_n^i\wedge T}$ are almost surely bounded by $3$, so that
\begin{align}
\Delta&\leq 9\, \mathbb E\left[\sum_{i\in \mathbb N}\sum_{n\geq 1}\left( \Pi_{\tau_{n-1}^i}^{A}\Pi_{\tau_{n-1}^i}^{B}\right)^2\ind_{\tau_{n-1}^i<T}\right]\notag\\
&\leq 9\, \mathbb E\left[\sum_{n\geq 1}\sum_{i\in E_n}\left( \Pi_{\tau_{n-1}}^{A}\Pi_{\tau_{n-1}}^{B}\right)^2\ind_{\tau_{n-1}<T}\right]\notag\\
&\leq 36\, \mathbb E\left[\sum_{n\geq 1}\left( \Pi_{\tau_{n-1}}^{A}\Pi_{\tau_{n-1}}^{B}\right)^2\ind_{\tau_{n-1}<T}\right]\notag
\end{align}

 Hence, denoting by $\beta_T$ (resp. $\kappa_T$) the number of branching (resp. killing without resampling) events occuring before time $T\geq 0$ (so that $C_T=\beta_T+\kappa_T+A_T$), we have
\begin{align}
\label{eq:Deltabound}
\Delta\leq 36\, \mathbb{E}\left[ \left(\Pi_T^B\right)^2 (1+\beta_T+\kappa_T+\sum_{n=1}^{A_T} \left(\Pi^A_{\rho_{n-1}}\right)^2)\right]
\end{align}
We already computed that, almost surely,
\[
\sum_{n=1}^{A_T} \left(\Pi^A_{\rho_{n-1}}\right)^2\leq \sum_{k=0}^{+\infty} \left(\frac{R_T-1}{R_T}\right)^{2k}\leq R_T/2\leq (N_0+\beta_T)/2.
\]
Now, observe that the number of particles in the system as time $T$ is upper bounded by $N_0+\beta_T-\kappa_T$ and lower bounded by $0$, so that $\kappa_T\leq  N_0+\beta_T$, and hence
\begin{multline*}
\mathbb{E}\left[ \left(\Pi_T^B\right)^2 (1+\beta_T+\kappa_T+\sum_{n=1}^{A_T} \left(\Pi^A_{\rho_{n-1}}\right)^2)\right]
\leq \mathbb{E}\left[ \left(\Pi_T^B\right)^2(\beta_T+3N_0)\right]
\leq 4 N_0 \exp(c\|b\|_\infty T),
\end{multline*}
where we used, as above, Lemma~\ref{lem:usefulbound}. 
 We thus obtained, using the last inequality and~\eqref{eq:Deltabound}, that
\begin{align*}
\Delta
\leq 144 \, N_0 \,\exp\left(c\|b\|_\infty T\right).
\end{align*}

Using~\eqref{mgdecomp2},~\eqref{bounda2},~\eqref{bounda3} and the last  inequality, we deduce that
\begin{align*}
\E\left[\left(\nu_{T}^f - \nu_{0}^f\right)^2\right]\leq 152 N_0  \exp\left(c \|b\|_\infty T\right).
\end{align*}

Without assuming that $\sup_{t\in[0,T]} \|Q_t f\|_{\infty}\leq 1$, we thus obtain
\begin{align}
\label{eq:endofstep4}
\E\left[\left(\nu_{T}^f - \nu_{0}^f\right)^2\right]\leq  152 N_0  \exp\left(c \|b\|_\infty T\right) \left(\sup_{t\in[0,T]} \|Q_t f\|_{\infty}\right)^2.
\end{align}

\bigskip\noindent\textit{Step 5. Conclusion.}

Recall the occupation measure $m_t$ defined in \eqref{occmeas} for all $t\geq 0$. With this notation, the conclusion of Step~4 reads
\[
\E\left[\left(m_0Q_T f - \Pi_T^A\Pi_T^B m_T f\right)^2\right]\leq  152 N_0  \exp\left(c \|b\|_\infty T\right) \left(\sup_{t\in[0,T]} \|Q_t f\|_{\infty}\right)^2.
\]
We deduce that
\begin{align*}
&\left\| m_0 Q_T\mathbf{1}_E \, \frac{m_T f}{m_T\mathbf{1}_E}\mathbf 1_{\bar N_T\neq 0} -m_0Q_T f \right\|_2 \\
&\hspace{2cm}\leq \left\|(m_0 Q_T\mathbf{1}_E-\Pi_T^A\Pi_T^B m_T\mathbf{1}_E)\,\frac{m_T f}{m_T\mathbf{1}_E}\mathbf 1_{\bar N_T\neq 0}\right\|_2+\left\|\Pi_T^A\Pi_T^B m_T f-m_0Q_T f\right\|_2\\
&\hspace{2cm}\leq 2\,\sqrt{152 N_0} \exp\left(c \|b\|_\infty T/2\right) \|f\|_\infty\sup_{t\in[0,T]} \|Q_t \ind_E\|_{\infty}.
\end{align*} 
We conclude that
\begin{align*}
	\left\|\frac{m_0Q_T f}{m_0 Q_T\mathbf{1}_E}- \frac{m_Tf}{m_T(\mathbf{1}_E)}\mathbf 1_{\bar N_T\neq 0}\right\|_2&\leq \frac{2\,\sqrt{152} \exp\left(c \|b\|_\infty T/2\right) \|f\|_\infty\sup_{t\in[0,T]} \|Q_t \ind_E\|_{\infty}}{ m_0 Q_T\mathbf{1}_E/N_0}\,\frac{1}{\sqrt{N_0}}.
\end{align*}
Observing that $\sup_{t\in[0,T]} \|Q_t \ind_E\|_{\infty}\leq e^{\|b\|_\infty T}$, this concludes the proof of Theorem~\ref{thm:main}.
\end{proof}

\bigskip

\begin{proof}[Proof of Lemma~\ref{lem:usefulbound}]
    We use the same notations as in the proof of Theorem~\ref{thm:main}.

    Fix $C> c_\ell \|b\|_\infty$ and $\delta>0$ such that $C-\delta C^2 2^\ell\geq c_\ell \|b\|_\infty$. We have
    \begin{align*}
    \E\left(\left(\Pi^B_{\tau_1\wedge\delta}\right)^\ell e^{-C(\tau_1\wedge\delta)}\right)
    &=\E\left(\left(\frac{N_0+1}{N_0}\right)^\ell\mathbf 1_{\sigma_1=\tau_1\wedge\delta}e^{-C(\tau_1\wedge\delta)}\right)+\E\left(\mathbf 1_{\sigma_1>\tau_1\wedge\delta}e^{-C(\tau_1\wedge\delta)}\right)\\
    &=\left(\frac{N_0+1}{N_0}\right)^\ell\E\left(e^{-C(\tau_1\wedge\delta)}\right)+\left(1-\left(\frac{N_0+1}{N_0}\right)^\ell\right)\,\E\left(\mathbf 1_{\sigma_1>\tau_1\wedge\delta}e^{-C(\tau_1\wedge\delta)}\right).
    \end{align*}
    On the one hand, we have 
    \begin{align*}
    \E\left(e^{-C(\tau_1\wedge\delta)}\right)\leq 1-e^{-C\delta}C\E(\tau_1\wedge\delta)\leq 1-(C-\delta C^2)\E(\tau_1\wedge\delta),
    \end{align*}
    and, on the other hand,
    \begin{align*}
    \E\left(\mathbf 1_{\sigma_1>\tau_1\wedge\delta}e^{-C(\tau_1\wedge\delta)}\right)\geq \E\left(e^{-(\tau_1\wedge \delta) (N_0\|b\|_\infty+C)}\right)\geq 1-(N_0\|b\|_\infty+C)\E\left(\tau_1\wedge \delta \right).
    \end{align*}
    Setting, for the sake of readability, $\alpha=\E(\tau_1\wedge \delta)$ and $\beta=\left(\frac{N_0+1}{N_0}\right)^\ell$, we deduce that
    \begin{align}
    \E\left(\left(\Pi^B_{\tau_1\wedge\delta}\right)^\ell e^{-C(\tau_1\wedge\delta)}\right)
    &\leq 1-\alpha\left(C-\delta C^2\beta+N_0\|b\|_\infty(1-\beta)\right)\notag\\
    &\leq 1-\alpha\left(C-\delta C^22^\ell-\|b\|_\infty c_\ell\right)\leq 1,\label{eq:step1}
    \end{align}
    where we used the fact that $\beta\leq2^\ell$ and $\beta-1\leq c_\ell/N_0$.

    We define the sequence of stopping times $(\theta_n)_{n\geq 0}$ by $\theta_0=0$ and
    \begin{align*}
    \theta_{n+1}:=\inf\{((n+1)\delta)\wedge \tau_k,\ k\geq 1, \tau_k>\theta_n\}.
    \end{align*}
    Then 
    \begin{multline*}
    \E\left(\left(\Pi_{T\wedge \theta_{n+1}}^B\right)^\ell e^{-C(T\wedge \theta_{n+1})}\mid \mathcal F_{T\wedge \theta_{n}}\right)
    =\left(\Pi_{T\wedge \theta_{n}}^B\right)^\ell\,e^{-C(T\wedge \theta_{n})}\\
    \times \E_{\mathbb X_{T\wedge \theta_{n}}}\left(\left(\Pi_{(T-u)\wedge\theta_1}^B\right)^\ell e^{-C((T-u)\wedge \theta_{1})}\right)_{\vert u=T\wedge \theta_{n}}.
    \end{multline*}
    Since $T-u\leq \delta$ almost surely in the right hand term, we deduce from~\eqref{eq:step1} that
    \begin{align*}
    \E\left(\left(\Pi_{T\wedge \theta_{n+1}}^B\right)^\ell e^{-C(T\wedge \theta_{n+1})}\mid \mathcal F_{T\wedge \theta_{n}}\right)
    &\leq \left(\Pi_{T\wedge \theta_{n}}^B\right)^\ell \,e^{-C(T\wedge \theta_{n})}.
    \end{align*}
    In particular, using Fatou's Lemma, we deduce, letting $n\to+\infty$, that
    \begin{align*}
    \E\left(\left(\Pi_{T}^B\right)^\ell\,e^{-CT}\right)\leq 1.
    \end{align*}
    Since this is true for any $C> c_\ell \|b\|_\infty$, this concludes the proof of the first part of the lemma.
    
    For the second part, we observe that $\beta_T$ is stochastically dominated by a continuous time counting process $(Z_t)_{t\geq 0}$ on $\mathbb N$ starting from $N_0$ and with increment rate $\|b\|_\infty n$ from state $n$ to state $n+1$. In particular, denoting by $L$ the infinitesimal generator of this process and $V(n)=n^2$ for all $n\geq 1$, we obtain
    \[
    LV(n)=\|b\|_\infty n((n+1)^2-n^2)=\|b\|_\infty n(2n+1)\leq 3\|b\|_\infty V(n),
    \]
    and deduce that 
    \[
    \E(\beta_T^2)\leq \E(Z_T^2)\leq e^{3\|b\|_\infty T} Z_0^2=e^{3\|b\|_\infty T} N_0^2.
    \]
    Hence, we obtain
    \begin{align*}
    \E\left(\beta_T\left(\Pi_T^B\right)^2\right)\leq \sqrt{\E(\beta_T^2)}\sqrt{\E\left(\left(\Pi_T^B\right)^4\right)}\leq N_0\exp((3+c_4) \|b\|_\infty T/2).
    \end{align*}
    This concludes the proof of the second part of the Lemma.
\end{proof}

\begin{rem}  
    Our proof adapts to the case where $\tau_\infty<+\infty$ with positive probability (that is when Assumption~2 does not holds true), using the following strategy. In this situation, one shows using the same approach that, for all $n\geq 0$,
    \[
    m_0 Q_T f=\E\left(\Pi^A_{\tau_n\wedge T}\,\Pi^B_{\tau_n\wedge T} m_{\tau_n\wedge T} Q_{T-\tau_n\wedge T}f\right).
    \]
    Since the branching rate is uniformly bounded and since $\Pi^A_{\tau_n\wedge T}$ is bounded by $1$, one checks that $\Pi^A_{\tau_n\wedge T}\,\Pi^B_{\tau_n\wedge T} m_{\tau_n\wedge T} Q_{T-\tau_n\wedge T}f$ can be uniformly bounded (in $n$) by an integrable random variable. Since in addition $\Pi^A_{\tau_n\wedge T}\,\Pi^B_{\tau_n\wedge T} m_{\tau_n\wedge T} Q_{T-\tau_n\wedge T}f\to 0$ on the event $\tau_\infty\leq T$ (since this events necessarily corresponds to the accumulation of resampling events) and $\Pi^A_{\tau_n\wedge T}\,\Pi^B_{\tau_n\wedge T} m_{\tau_n\wedge T} Q_{T-\tau_n\wedge T}f\to \Pi^A_T\,\Pi^B_T m_T f$ on the complementary event, one deduces from the dominated convergence theorem that
    \[
    m_0 Q_T f=\E\left(\Pi^A_T\,\Pi^B_T m_T f \mathbf 1_{T<\tau_\infty}\right).
    \]
    The rest of the proof (see Steps~4 and~5) then proceeds almost identically.
\end{rem}

\begin{rem}
    We comment on some direct and natural continuations of the proof, which are easily derived from the fact that the proof is based on the control of martingale increments. 
    
    First, if the martingale $\nu_t^f-\nu_0^f$ is c\`adl\`ag, then Doob's maximal inequality and~\eqref{eq:endofstep4} entail that
    \[
    \E\left[\sup_{t\in[0,T]} \left(\nu_{T}^f - \nu_{0}^f\right)^2\right]\leq  c_T\,\|f\|_\infty^2 \sup_{t\in [0,T]}\Vert Q_t\mathbf{1}_E\Vert_{\infty}^2 N_0
    \]
    for some (explicit) constant $c_T$.
    Then the conclusion reads
    \begin{align*}
    \left\|\sup_{t\in{[0,T]}} \left|\frac{m_0Q_t f}{m_0 Q_t\mathbf{1}_E}- \frac{m_t f}{m_t(\mathbf{1}_E)}\mathbf 1_{\bar N_t\neq 0}\right|\right\|_2&\leq \frac{c_T\,\|f\|_\infty}{m_0 Q_T\mathbf{1}_E/N_0}\,\frac{1}{\sqrt{N_0}},
    \end{align*}

    Second, if the killing rate is bounded, then the statement of Lemma~\ref{lem:usefulbound} holds true with $\beta_T$ replaced by the total number of killing/branching events $C_T$, and with different constants. Then, for any $p\geq 2$, using the same approach as in Step~4 of the proof, but using the Burkholder-Davis-Gundy inequality instead, one deduces similarly that 
    \[
    \E\left[\left|\nu_{T}^f - \nu_{0}^f\right|^p\right]\leq c_T N_0^{p/2},
    \]
    and hence
    \begin{align*}
    \left\|\frac{m_0Q_t f}{m_0 Q_t\mathbf{1}_E}- \frac{m_t f}{m_t(\mathbf{1}_E)}\right\|_p&\leq \frac{c'_T\,\|f\|_\infty}{ m_0 Q_T\mathbf{1}_E/N_0}\,\frac{1}{\sqrt{ N_0}},
    \end{align*}
    for some (explicit) constants $c_T$ and $c'_T$. 
    Note that the assumption that the killing rate is uniformly bounded is technical: after using Burkholder-Davis-Gundy inequality, one needs to control the term $\E\left(C_T^{p/2}\left(\frac{N+1}{N}\right)^{pC_T}\right)$, however, depending on the particular process at hand, there may be other ways to obtain upper bounds on this quantity which do not require the killing rate to be bounded.
\end{rem}

\begin{acks}[Acknowledgments]
The authors would like to thank the anonymous referee for their constructive comments that improved the exposition and quality of this paper.
\end{acks}

\begin{funding}
The first author was support by EPSRC Grants EP/P009220/1 and EP/W026899/1. The second author was also support by EPSRC grant EP/W026899/1.
\end{funding}


\bibliographystyle{imsart-number} 
\bibliography{biblio-denis}       

%
%
%
%
%
\end{document}


\begin{frontmatter}
\title{Binary branching processes with Moran type interactions: a formal construction of the particle system}

%
%

\begin{abstract}
We provide a formal description of the binary branching process with Moran type interactions (BBMMI), as introduced in \cite{BBMMI}, and show that it defines a Markov process with respect to its natural filtration.
\end{abstract}

\begin{abstract}[language=french]
Nous fournissons une description formelle du processus de branchement binaire avec interactions de type Moran (BBMMI), tel qu'introduit dans \cite{BBMMI}, et montrons qu'il définit un processus de Markov par rapport à sa filtration naturelle.
\end{abstract}

%

\end{frontmatter}
\section{Formal construction of the BBMMI}

We will use the same notation as introduced in \cite[Section 2]{BBMMI}. Consider the measurable space
\begin{align*}
F=\bigcup_{s\in \mathcal P_f(\mathbb N)} \{s\}\otimes \Omega^{s}\otimes \mathbb R_+^{s}\otimes \mathbb R_+^{s},
\end{align*}
where, as usual, denumerable spaces are endowed with the discrete $\sigma$-fields, $\mathbb R_+$ is endowed with its Borel $\sigma$-field, and product spaces are endowed with the product $\sigma$-field. We denote generic elements of $F$ by
\begin{align*}
\theta=(s,(\omega_i)_{i\in s},(e^b_i)_{i\in s},(e^\kappa_i)_{i\in s}),
\end{align*} 
where $s\in\mathcal P_f(\mathbb N)$ is used to enumerate the particle system, where, for all $i\in s$, $X(\omega_i):t\in [0,+\infty)\to X_t(\omega_i)\in E\cup\partial$ is the trajectory of the $i^{th}$ particle, and where $e^b_i\in\mathbb R_+$ and $e^\kappa_i\in\mathbb R_+$ are auxiliary components used to define the branching and soft-killing times.

Fix $s\in \mathcal P_f(\mathbb N)$ and
let us consider the probability kernel $N^1$ from $\{s\}\otimes E^s$ to $F$ defined by
\begin{align*}
N^1(s,(x_i)_{i\in s}):= \delta_{s}\otimes \left(\bigotimes_{i\in s} \bP_{x_i}\right)\otimes \mathcal E^{\otimes s}\otimes  \mathcal E^{\otimes s},
\end{align*}
where $\mathcal E$ denotes the exponential distribution with parameter $1$. This kernel is used to model the evolution of the particles before the occurence of a branching/killing event.

Given an element $\theta\in F$, our aim is now to define the position of the system at the first time of branching/killing. In order to do so, we define, for all $\theta\in F$ and all $i\in s$, the  functions
\begin{align*}
&\tau^{i,b}(\theta)=\inf\{t\geq 0,\ \int_0^t b^{i_0}((X_u(\omega_j))_{j\in s})\,\mathrm du=e_i^b\},\\
&\tau^{i,\kappa}(\theta)=\inf\{t\geq 0,\ \int_0^t \kappa^{i_0}((X_u(\omega_j))_{j\in s})\,\mathrm du=e_i^\kappa\},\\
&\tau^{i,\partial}(\theta)=\inf\{t\geq 0,\ X_t(\omega_{i})\in\partial\}.
\end{align*}
Note that, for all $t\geq 0$, $(u,\omega)\in [0,+\infty)\times \mathcal F \mapsto X_u(\omega)$ is measurable with respect $\mathcal B([0,+\infty))\otimes \mathcal F$, which entails that $\omega\mapsto \int_0^t b^{i}((X_u(\omega_j))_{j\in s})\,\mathrm du$ is $\mathcal F^{\otimes s}$-measurable. In addition, since $t\mapsto \int_0^t b^{i}((X_u(\omega_j))_{j\in s})\,\mathrm du$, $t\mapsto \int_0^t \kappa^{i}((X_u(\omega_j))_{j\in s})\,\mathrm du$ and $t\mapsto \ind_{X_t(\omega_{i})\in\partial}$ are right-continuous, the infima in the above definitions can be taken over $\mathbb Q$ and are thus measurable functions of $(\omega_j)_{j\in s}$, $(e^b_j)_{j\in s}$ and $(e^\kappa_j)_{j\in s}$ and thus of $\theta$. This is also the case for the functions $\tau$ and $i_0$ defined as
\begin{align*}
&\tau(\theta):=\min_{i\in s} \tau^{i,b}(\theta)\wedge \tau^{i,\kappa}(\theta)\wedge \tau^{i,\partial}(\theta),\\
&i_0(\theta):=\min\{i\in S,\ \tau^{i,b}(\theta)\wedge \tau^{i,\kappa}(\theta)\wedge \tau^{i,\partial}(\theta)=\tau(\theta)\}.
\end{align*}
For all $\theta\in F$, we define the \textit{final position of $\theta$}, denoted by $Y(\theta)\in (E\cup\partial)^s$, by 
\[
Y(\theta):=(X_{\tau(\omega)}(\omega_{i}))_{i\in s},
\]
which is also a measurable function of $\theta$.

Then we construct a kernel which, to any given $\theta\in F$ and its final position $Y(\theta)$, associates its position after the branching/killing and selection/resampling event. In order to do so, we define the kernel $N^2$ from $F$ to $\cup_{s\in \mathcal P_f(\mathbb N)}\{s\}\times (E\cup\dd)^s$ by
\begin{align*}
N^2(\theta)&=\ind_{\tau(\theta)=\tau^{i_0,b}}p^{b,i_0(\theta)}(Y(\theta))\frac{1}{|s|+1}\sum_{i\in s_+}\delta_{s_+^{-i}}\otimes \delta_{Y(\theta)^{+i_0,-i}}\\
&\qquad+\ind_{\tau(\theta)=\tau^{i_0,b}}\left(1-p^{b,i_0(\theta)}(Y(\theta))\right)\delta_{s_+}\otimes \delta_{Y(\theta)^{+i_0}}\\
&\qquad+\ind_{\tau(\theta)=\tau^{i_0,\kappa}\wedge \tau^{i_0,\dd}}\,p^{\kappa,i_0(\theta)}(Y(\theta))\frac{1}{|s|-1}\sum_{i\in  s^{-i_0}}\delta_{s_+^{-i_0}}\otimes \delta_{Y(\theta)^{+i,-i_0}}\\
&\qquad+\ind_{\tau(\theta)=\tau^{i_0,\kappa}\wedge \tau^{i_0,\dd}}\,\left(1-p^{\kappa,i_0(\theta)}(Y(\theta))\right)\delta_{s^{-i_0}}\otimes \delta_{Y(\theta)^{-i_0}}
\end{align*}
where $s_+=\{1+\max s\}$, $s_+^{-i}=\{1+\max s\} \cup s\setminus \{i\}$, $Y(\theta)^{i_0,-i}\in (E\cup\dd)^{s_+^{-i}}$ is such that
\[
Y(\theta)^{i_0,-i}_j=\begin{cases}
Y(\theta)_j&\text{ for all }j\in s\setminus \{i\},\\
Y(\theta)_{i_0}&\text{ for }j=1+\max s,
\end{cases}
\]
$Y(\theta)^{i_0}\in (E\cup\dd)^{s_+}$ is such that
\[
Y(\theta)^{i_0}_j=\begin{cases}
Y(\theta)_j&\text{ for all }j\in s,\\
Y(\theta)_{i_0}&\text{ for }j=1+\max s,
\end{cases}
\]

According to Ionescu Tulcea's extension Theorem, for any initial distribution $\mu_0$ on $F_0:=\cup_{s\in \mathcal P_f(\mathbb N)}\{s\}\times (E\cup\dd)^s$, there exists a probability measure $\mathbb P_{\mu_0}$ on $\bar\Omega=F_0\times F^{\mathbb N}$, such that, for all $n\geq 2$,
\begin{multline*}
\mathbb P_{\mu_0}(A_0\in \mathrm da_0,\Theta_1\in\mathrm d\theta_1,\,\ldots,\,A_{n-1}\in\mathrm da_{n-1},\,\Theta_n\in\mathrm d\theta_n)\\
=\mu_0\circ N^1\circ N^2\circ N^1\circ \cdots\circ N^2\circ N^1(\mathrm da_0,\mathrm d\theta_1,\cdots,\,\mathrm d a_{n-1},\mathrm d\theta_n).
\end{multline*}

Then, we can define, for any $\bar\omega=(a_0,\theta_1,\ldots)$, the sequence of times $(\tau_n)_{n\in\mathbb N}$ as
\[
\tau_0(\bar\omega)=0\text{ and }\tau_n(\bar\omega)=\tau(\theta_1)+\cdots+\tau(\theta_n),\ \forall n\geq 0
\]
and, for all $t\geq 0$,
\[
\bar S_t(\bar\omega)=\sum_{n\geq 0}\ind_{\tau_{n-1}\leq t<\tau_n} s(\theta_n).
\]
and, for all $i\in \bar S_t(\bar\omega)$,
\[
X^i_t(\bar\omega)=X_{t-\tau_{n-1}(\bar\omega)}(\omega_i(\theta_n)).
\]
The process $(\bar S_t,(X^i_t)_{i\in\bar S_t})_{t\geq 0}$, defined on the probability space $(\bar\Omega,\mathbb P)$ and taking its values in $F_0$, is then the size constrained branching process, which concludes the construction of the process.

\begin{rem}
    Observe that the particle system is well defined without Assumptions 1 and 2. However, if Assumption~1 holds true, then $(\bar S_t,(X^i_t)_{i\in\bar S_t})_{t\geq 0}$ takes its values in $\bigcup_{s\in \mathcal P_f(\mathbb N)} \{s\}\otimes E^s$ almost surely. Also, if Assumption~2 holds true, then $\bar S_t$ is non-empty almost surely, for all $t\in[0,+\infty)$.
\end{rem}


The above construction makes use of the fact that $X$ is a measurable adapted process. Using the fact that it is also assumed progressively measurable, we are in position to prove the following.

\begin{prop}
    The process $(\bar S_t,(X^i_t)_{i\in\bar S_t})_{t\geq 0}$ is a Markov process with respect to its natural filtration.
\end{prop}

\begin{proof}
    Fix $t\geq s\geq 0$, and $A$ a measurable subset of $F_0$. For all $k\geq 1$, all $s_1<s_2<\cdots<s_k= s$ and all measurable subsets $A_1,\ldots, A_k$ of $F_0$, we have, setting $\XX_t:=(\bar S_t,(X^i_t)_{i\in\bar S_t})$,
    \begin{multline}
    \label{eq:startMarkov}
    \mathbb P_{\mu_0}\left(\XX_{s_1}\in A_1,\ldots,\XX_{s_k}\in A_k,\XX_{t}\in A\right)
    \\
    =\sum_{n_1\leq \cdots\leq n_k=m\leq n} \mathbb P_{\mu_0}\left(\XX_{s_1}\in A_1,\ldots,\XX_{s_k}\in A_k,\,s_i\in[\tau_{n_i-1},\tau_{n_i}),\,\XX_{t}\in A,\,t\in[\tau_n,\tau_{n+1})\right).
    \end{multline}
    
    Then, conditioning with respect to $\mathcal G_m:=\sigma(\Theta_1,\ldots,\Theta_{m},A_{m})$, we have
    \begin{multline}
    \label{eq:discretMarkovm}
    \mathbb P_{\mu_0}\left(\XX_{s_1}\in A_1,\ldots,\XX_{s_k}\in A_k,\,s_i\in[\tau_{n_i-1},\tau_{n_i}),\,\XX_{s+t}\in A,\,s+t\in[\tau_n,\tau_{n+1})\mid \mathcal G_m\right)\\
    =\ind_{\XX_{s_1}\in A_1,\ldots,\XX_{s_\ell}\in A_\ell,\,s_i\in[\tau_{n_i-1},\tau_{n_i}),\tau_m\leq s}\times f(A_m,\tau_m),
    \end{multline}
    where
    \begin{align*}
    f(a,u):= \mathbb P_{a}\left(\XX_{s_{\ell+1}-u}\in A_{\ell+1},\ldots,\XX_{s_k-u}\in A_k,\,s_i-u\in[0,\tau_1),\XX_{t-u}\in A,\,t-u\in[\tau_{n-m},\tau_{n-m+1})\right).
    \end{align*}
    There are now two possible situations, either $m=n$ or $m<n$, which we consider separately.
    
    \medskip\noindent\textit{Case $m=n$.} In the former case, we observe that, by definition of $N^1$,
    \begin{align*}
    f(a,u)&= \mathbb P_{a}\left(\XX_{s_{\ell+1}-u}\in A_{\ell+1},\ldots,\XX_{s_k-u}\in A_k,\,s_i-u\in[0,\tau_1),\XX_{t-u}\in A,\,t-u\in[0,\tau_1)\right)\\
    &=\texttt P_{a}\left(\YY_{s_{\ell+1}-u}\in A_{\ell+1},\ldots,\YY_{s_k-u}\in A_k,\,s_i-u\in[0,\texttt t),\YY_{t-u}\in A,\,t-u\in[0,\texttt t)\right),
    \end{align*}
    where $\texttt P_{a}$ is the product measure $\Omega'=\mathbb P^{\otimes s}\otimes \mathcal E(1)^s$ on the probability space $\Omega^s(\mathbb R_+)^s\otimes (\mathbb R_+)^s$ endowed with the filtration defined  by $\mathcal H_t:=\mathcal F_t^{\otimes s}\otimes\{\emptyset,\mathbb R_+\}^s\otimes\{\emptyset,\mathbb R_+\}^s$, with the variables $\YY$ and $\texttt t$ being defined, for all $\theta'=((\omega_i)_{i\in s},(e^b_i)_{i\in s},(e^\kappa_i)_{i\in s})\in \Omega'$, by
    \begin{align*}
    \YY_t=(s,(X_t(\omega_i)_{i\in s}),\ \forall t\geq 0,
    \end{align*}
    and $\texttt t=\min_{i\in s}\texttt t^{b,i}\wedge\texttt t^{\kappa,i}\wedge t^{\dd,i}$, with
    \begin{align*}
    &\texttt t^{b,i}=\inf\{t\geq 0,\ \int_0^t b^{i_0}((X_u(\omega_j))_{j\in s})\,\mathrm du=e_i^b\},\\
    &\texttt t^{\kappa,i}=\inf\{t\geq 0,\ \int_0^t \kappa^{i_0}((X_u(\omega_j))_{j\in s})\,\mathrm du= e_i^\kappa\},\\
    &\texttt t^{\partial,i}=\inf\{t\geq 0,\ X_t(\omega_i)\in\partial\}.
    \end{align*}
    We emphasise that the second component of $\YY$ is distributed as a vector indexed by $s$ of independent copies of $X$ starting from $x_i$, for all $i\in s$, and that it is a progressively measurable Markov process with respect to the filtration $\mathcal H$. In addition, the random variables $e^b_i$ and $e^\kappa_i$, $i\in s$, are independent from each others and from $\mathcal H_t$, for all $t\geq 0$.
    
    Using the Markov property at time $s-u$ for the process $\YY$, we obtain
    \begin{align*}
    \texttt P_a&\left(\YY_{s_{\ell+1}-u}\in A_{\ell+1},\ldots,\YY_{s_k-u}\in A_k,\,s_i-u\in[0,\texttt t),\YY_{t-u}\in A,\,t-u\in[0,\texttt t)\right)\\
    &\qquad = \texttt E_a\left(\ind_{\YY_{s_{\ell+1}-u}\in A_{\ell+1},\ldots,\YY_{s_k-u}\in A_k}\ind_{\YY_{t-u}\in A\setminus \dd}\,e^{-\int_0^{t-u} h(\YY_v)\,\mathrm dv}\right)\\
    &\qquad = \texttt E_a\left(\ind_{\YY_{s_{\ell+1}-u}\in A_{\ell+1},\ldots,\YY_{s_k-u}\in A_k}\texttt E_a\left(\ind_{\YY_{t-u}\in A\setminus \dd}\,e^{-\int_0^{t-u} h(\YY_v)\,\mathrm dv}\mid \mathcal H_{s-u}\right)\right)\\
    &\qquad = \texttt E_a\left(\ind_{\YY_{s_{\ell+1}-u}\in A_{\ell+1},\ldots,\YY_{s_k-u}\in A_k}e^{-\int_0^{s-u} h(\YY_v)\,\mathrm dv}\texttt E_{\YY_{s-u}}\left(\ind_{\YY_{t-s}\in A\setminus \dd}\,e^{-\int_0^{t-s} h(\YY_v)\,\mathrm dv}\right)\right)\\
    &\qquad = \texttt E_a\left(\ind_{\YY_{s_{\ell+1}-u}\in A_{\ell+1},\ldots,\YY_{s_k-u}\in A_k}\ind_{s_k-u\in[0,\texttt t)}\texttt E_{\YY_{s-u}}\left(\ind_{\YY_{t-s}\in A\setminus \dd}\,\ind_{t-s\in[0,\texttt t)}\right)\right)
    \end{align*}
    where $h(s,(x_i)_{i\in s}):=\sum_{i\in s} \kappa^i(x_i)+b^i(x_i)$, and where we used the fact that $\YY$ is progressively measurable to deduce that $\int_0^{s-u} h(\YY_v)\,\mathrm dv$ is $\mathcal H_{s-u}$-measurable. We deduce that
    \begin{align*}
    f(a,u)=\mathbb E_{a}\left(\ind_{\XX_{s_{\ell+1}-u}\in A_{\ell+1},\ldots,\XX_{s_k-u}\in A_k,\,s_i-u\in[0,\tau_1)}\,\mathbb E_{\XX_{s-u}}\left(\XX_{t-s}\in A,\,t-s\in[0,\tau_1)\right)\right)
    \end{align*}
    Taking the expectation in~\eqref{eq:discretMarkovm}, we finally deduce that
    \begin{multline*}
    \mathbb P_{\mu_0}\left(\XX_{s_1}\in A_1,\ldots,\XX_{s_k}\in A_k,\,s_i\in[\tau_{n_i-1},\tau_{n_i}),\,\XX_{s+t}\in A,\,s+t\in[\tau_n,\tau_{n+1})\right)\\
    = \mathbb E_{\mu_0}\left(\ind_{\XX_{s_1}\in A_1,\ldots,\XX_{s_\ell}\in A_\ell,\,s_i\in[\tau_{n_i-1},\tau_{n_i}),\tau_m\leq s}\times f(A_m,\tau_m)\right),
    \end{multline*}
    and then, by definition of the law of $(A_0,\Theta_1,\cdots,A_{m-1},\Theta_m,A_m,\ldots)$, we deduce that
    \begin{multline}
    \label{eq:casem=n}
    \mathbb P_{\mu_0}\left(\XX_{s_1}\in A_1,\ldots,\XX_{s_k}\in A_k,\,s_i\in[\tau_{n_i-1},\tau_{n_i}),\,\XX_{s+t}\in A,\,s+t\in[\tau_n,\tau_{n+1})\right)\\
    = \mathbb E_{\mu_0}\left(\ind_{\XX_{s_1}\in A_1,\ldots,\XX_{s_k}\in A_k,\,s_i\in[\tau_{n_i-1},\tau_{n_i})}\mathbb P_{\XX_{s}}\left(\XX_{t-s}\in A,\,t-s\in[0,\tau_1)\right)\right).
    \end{multline}

    \medskip\noindent\textit{Case $m<n$.} We can now proceed with the case $m<n$. In this case, we write
    \begin{align*}
    f(a,u)=\mathbb E_a\left(\ind_{\XX_{s_{\ell+1}-u}\in A_{\ell+1},\ldots,\XX_{s_k-u}\in A_k,\,s_i-u\in[0,\tau_1)}\,g_u(A_1,\tau_1)\right),
    \end{align*}
    where
    \[
    g_u(b,v):=\mathbb E_{A_1}\left(\XX_{t-u-v}\in A,\,t-u-v\in[\tau_{n-m-1},\tau_{n-m})\right)
    \]
    Using the same strategy as above (with the same definition of $\YY$), we obtain, using also the Markov property at time $s_k-u=s-u$,
    \begin{align*}
    f(a,u)&=\texttt E_a\left(\ind_{\YY_{s_{\ell+1}-u}\in A_{\ell+1},\ldots,\YY_{s_k-u}\in A_k,\,s_i-u\in[0,\texttt t)}\,g_u(\YY_{\texttt t},\texttt t)\right)\\
    &=\texttt E_a\left(\ind_{\YY_{s_{\ell+1}-u}\in A_{\ell+1},\ldots,\YY_{s_k-u}\in A_k,\,s_i-u\in[0,\texttt t)}\,\texttt E_{\YY_{s-u}}(g_u(\YY_{\texttt t},\texttt t))\right)\\
    &=\mathbb E_a\left(\ind_{\XX_{s_{\ell+1}-u}\in A_{\ell+1},\ldots,\XX_{s_k-u}\in A_k,\,s_i-u\in[0,\tau_1)}\,\mathbb E_{\XX_{s-u}}(g_u(\XX_{\tau_1},\tau_1))\right)
    \end{align*}
    Using again the definition of the law of  $(A_0,\Theta_1,\cdots,A_{m-1},\Theta_m,A_m,\ldots)$, we deduce that
    \begin{multline}
    \label{eq:casem<n}
    \mathbb P_{\mu_0}\left(\XX_{s_1}\in A_1,\ldots,\XX_{s_k}\in A_k,\,s_i\in[\tau_{n_i-1},\tau_{n_i}),\,\XX_{s+t}\in A,\,s+t\in[\tau_n,\tau_{n+1})\right)\\
    = \mathbb E_{\mu_0}\left(\ind_{\XX_{s_1}\in A_1,\ldots,\XX_{s_k}\in A_k,\,s_i\in[\tau_{n_i-1},\tau_{n_i})}\mathbb P_{\XX_{s}}\left(\XX_{t-s}\in A,\,t-s\in[\tau_{n-m},\tau_{n-m+1})\right)\right).
    \end{multline}
    
    \medskip
    \noindent
    \textit{Conclusion.} Finally, using~\eqref{eq:casem=n} and~\eqref{eq:casem<n} and summing over the indices $n_i$ and $n$ as in~\eqref{eq:startMarkov}, we deduce that
    \begin{align*}
    \mathbb P_{\mu_0}\left(\XX_{s_1}\in A_1,\ldots,\XX_{s_k}\in A_k,\XX_{t}\in A\right)
    =
    \mathbb E_{\mu_0}\left(\ind_{\XX_{s_1}\in A_1,\ldots,\XX_{s_k}\in A_k}\,\mathbb P_{\XX_s}\left(\XX_{t-s}\in A\right)\right).
    \end{align*}
    Since the property holds true for all finite family of time indices $s_1<\ldots<s_k=s$ and all $A_1,\ldots,A_k$, we deduce that 
    \begin{align*}
    \mathbb P_{\mu_0}\left(\XX_{t}\in A\mid \sigma(\XX_u,\,u\leq s)\right)=\mathbb P_{\XX_s}\left(\XX_{t-s}\in A\right),
    \end{align*}
    which concludes the proof of the proposition.
\end{proof}

\begin{rem}
    Although we won't use this property, the process $(\bar S_t,(X^i_t)_{t\in\bar S_t})_{t\geq 0}$ can be shown to be a Markov process with respect to larger filtrations. On a related note, it is natural to ask whether the strong Markov property for $X$ entails the strong Markov property for the particle system. This is not clear at all in full generality and we leave it as an open problem. However, it can be shown, using a similar approach as in~\cite{IkedaNagasawaEtAl1969}, that, under additional regularity properties on the process $X$ and on the functionals $b,\kappa$,  and additional requirements on the state space $E$, the particle system is indeed a strong Markov process in a quite general setting.
\end{rem}


\bibliographystyle{imsart-number} 
\bibliography{biblio-denis}       

%
%
%
%
%